\documentclass[final,onefignum,onetabnum]{siamart171218}
\usepackage{csl}
\usepackage{bbm}
\usepackage{blkarray}
\usepackage{upgreek}
\usepackage{amssymb}
\usepackage{mathtools}
\usepackage{stmaryrd}
\usepackage{fancyhdr,subfigure}
\usepackage[small,normal,bf,up]{caption}
\usepackage{epsfig,graphicx,wrapfig,fancyvrb,listings,color,textcomp,float}
\usepackage{algorithmic}
\usepackage{hyperref,url}
\usepackage{array,multirow}
\usepackage{marvosym,ulem}
\usepackage{tikz}
\usetikzlibrary{patterns}
\usepackage{xspace}
\usepackage{calc}
\usepackage{accents}

\renewcommand{\Re}{{\mathbbm R}}
\newcommand{\one}{1\hspace{-0,9ex}1}

\newcommand{\A}{\mathbf{A}}
\renewcommand{\b}{\mathbf{b}}

\renewcommand{\c}{\mathbf{c}}

\newcommand{\f}{\mathfrak{f}}
\newcommand{\s}{\mathfrak{s}}

\newcommand{\fast}{{\{\hspace*{-0.75pt}\f\hspace*{-0.75pt}\}}}
\newcommand{\slow}{{\{\hspace*{-0.75pt}\s\hspace*{-0.75pt}\}}}
\newcommand{\fastfast}{{\{\hspace*{-0.75pt}\f,\f\hspace*{-0.75pt}\}}}
\newcommand{\fastslow}{{\{\hspace*{-0.75pt}\f,\s\hspace*{-0.75pt}\}}}
\newcommand{\slowfast}{{\{\hspace*{-0.75pt}\s,\f\hspace*{-0.75pt}\}}}
\newcommand{\slowslow}{{\{\hspace*{-0.75pt}\s,\s\hspace*{-0.75pt}\}}}

\newcommand{\ys}{y^\slow}
\newcommand{\yf}{y^\fast}
\newcommand{\Ys}{Y^\slow}
\newcommand{\Yf}{Y^\fast}

\newcommand{\funs}{f^\slow}
\newcommand{\funf}{f^\fast}
\newcommand{\etas}{\eta^\slow}
\newcommand{\etaf}{\eta^\fast}
\newcommand{\lambdas}{\lambda^\slow}
\newcommand{\lambdaf}{\lambda^\fast}
\newcommand{\ws}{w^\slow}
\newcommand{\wf}{w^\fast}
\newcommand{\zs}{z^\slow}
\newcommand{\zf}{z^\fast}

\DeclareMathAlphabet{\mathbase}{OT1}{pzc}{m}{n}
\newcommand{\Abase}{\mathbase{A}^\slow}
\newcommand{\abase}{\mathbase{a}^\slow}
\newcommand{\bbase}{\mathbase{b}^\slow}
\newcommand{\bhatbase}{\widehat{\mathbase{b}}^\slow}
\newcommand{\cbase}{\mathbase{c}^\slow}
\newcommand{\deltac}{\Delta\hspace*{-0.5pt}\mathbase{c}^\slow}

\newsiamremark{remark}{{\sc Remark}}
\newsiamremark{comment}{{\sc Comment}}
\newsiamremark{example}{{\sc Example}}

\newif\ifreport
\reporttrue

\graphicspath{ {./} } 

\ifreport
\newcommand{\astretch}{1.5}
\else
\newcommand{\astretch}{1.3}
\fi

\headers{MRI-GARK}{A. Sandu}

\title{A Class of Multirate Infinitesimal GARK Methods\thanks{
\funding{This work was funded by awards NSF CCF--1613905, NSF ACI--1709727, AFOSR DDDAS 15RT1037, and by the Computational Science Laboratory at Virginia Tech}}}
\author{
Adrian Sandu\thanks{Computational Science Laboratory, Department of Computer Science, Virginia Tech, 2202 Kraft Drive, Blacksburg, VA 24060  (\email{sandu@cs.vt.edu}, \url{http://people.cs.vt.edu/\~asandu}).}
}

\begin{document}

	\cslyear{18}
	\cslreportnumber{5}
	\cslauthor{Adrian Sandu}
	\csltitle{A Class of Multirate Infinitesimal GARK Methods}
	\cslciteas{A. Sandu, A Class of Multirate Infinitesimal GARK Methods, SIAM Journal on Numerical Analysis Vol. 5, No. 7, pp. 2300-2327, doi:10.1137/18M1205492, 2019.}
	\csltitlepage

\maketitle
\begin{abstract}
Differential equations arising in many practical applications are characterized by multiple time scales. Multirate time integration seeks to solve them efficiently by discretizing each scale with a different, appropriate time step, while ensuring the overall accuracy and stability of the numerical solution. In a seminal paper Knoth and Wolke (APNUM, 1998) proposed a hybrid solution approach: discretize the slow component with an explicit Runge-Kutta method, and advance the fast component via a modified fast differential equation. The idea led to the development of multirate infinitesimal step (MIS) methods by Wensch et al. (BIT, 2009.) G\"{u}nther and Sandu (BIT, 2016) explained MIS schemes as a particular case of multirate General-structure Additive Runge-Kutta (MR-GARK) methods. The hybrid approach offers extreme flexibility in the choice of the numerical solution process for the fast component. 

This work constructs a family of multirate infinitesimal GARK schemes (MRI-GARK) that extends the hybrid dynamics approach in multiple ways. Order conditions theory and stability analyses are developed, and practical explicit and implicit methods of up to order four are constructed. Numerical results confirm the theoretical findings. We expect the new MRI-GARK family to be most useful for systems of equations with widely disparate time scales, where the fast process is dispersive, and where the influence of the fast component on the slow dynamics is weak.
\end{abstract}

\section{Introduction}

Many dynamical systems of practical interest consist of multiple components that evolve at different characteristic time scales. As a representative model we consider a two-way additively partitioned ordinary differential equation (ODE) driven by a slow process $\slow$ and a fast process $\fast$, acting simultaneously:
\begin{equation} 
	\label{eqn:multirate-additive-ode}
	 y'= f(t,y) = \funs (t,y) +  \funf (t,y), \qquad y(t_0)=y_0.
\end{equation}
Multirate methods are efficient numerical solution strategies for \eqref{eqn:multirate-additive-ode} that use small step sizes to discretize the fast components, and large step sizes to discretize the slow components. The approach goes back to the pioneering work on multirate Runge-Kutta methods of Rice \cite{Rice_1960_splitRK} and Andrus \cite{Andrus_1979_MR,Andrus_1993_MR-stability}.

The main approach for the construction of multirate methods is to employ traditional integrators with different time steps for different components, and to couple the fast and slow components such as to ensure the overall stability and accuracy of the discretization. Following this philosophy multirate schemes have been developed in the context of Runge-Kutta methods \cite{Sandu_2007_MR-monotonic-RK2,Guenther_2001_MR-PRK,Guenther_1994_partition-circuits,Kvaerno_2000_stability-MRK,Kvaerno_1999_MR-RK}, 
linear multistep methods  \cite{Gear_1984_MR-LMM,Kato_1999,Sandu_2009_MR-monotonic-LMM},
Rosenbrock-W methods  \cite{Guenther_1997_ROW},
extrapolation methods \cite{Sandu_2010_MR-extrapolation_ECMI,Sandu_2013_MR-extrapolation,Engstler_1997_MR-extrapolation,Sandu_2009_MR-hyperbolic_ICNAAM},
Galerkin discretizations \cite{Logg_2003_MAG1}, and combined multiscale  methodologies \cite{Engquist_2005}.

Multirate time integration has been adopted by many applications including the simulation of electronic circuits \cite{Chen_2010_MRcircuits,Guenther_1994_partition-circuits}, subsurface flows \cite{Wheeler_2016_geo}, ocean modeling \cite{Dawson_2014_MRDG,Remacle_2012_MRDG}, energy conversion \cite{Penalba_2019_energy}, multibody dynamics \cite{Arnold_2007_multibody}, electromagnetic fields \cite{Goedel_2009_MRDG}, and fluids in complex geometries \cite{Klockner_2018_MRAB}, to name just a few.

In a seminal paper Knoth and Wolke \cite{Knoth_1998_MR-IMEX} proposed a hybrid approach to solve \eqref{eqn:multirate-additive-ode}. First, the slow component is discretized with an $s$-stage explicit Runge-Kutta method $(A,b,c)$ with increasing abscissae, $c_{i-1} < c_i$. Next, the solution is advanced between the consecutive stages of this method by solving a modified fast ODE system. For autonomous systems \eqref{eqn:multirate-additive-ode} the solution process reads:
\begin{equation}
\label{eqn:MIS-one-step}
\renewcommand{\arraystretch}{\astretch}
\begin{array}{l}
Y^{\{\s\}}_1 \coloneqq y_n,  \\
\begin{cases}
v_i' = f^{\{\f\}} \left( v_i \right) + \sum_{j=1}^{i-1} \frac{a_{i,j} - a_{i-1,j}}{c_{i} - c_{i-1}} f^{\{\s\}} \bigl( Y^{\{\s\}}_j \bigr), \quad \tau \in [0, (c_{i} - c_{i-1})\, H], \\
\quad  \textnormal{with } v_i(0) = Y^{\{\s\}}_{i-1},
\quad Y^{\{\s\}}_i \coloneqq v_i\bigl( (c_{i} - c_{i-1})\,H \bigr), \quad i=2,\dots,s;
\end{cases}\\
v_{s+1}' = f^{\{\f\}} \left( v \right) + \sum_{j=1}^{s} \frac{b_{j} - a_{s,j}}{1 - c_{s}} f^{\{\s\}} \bigl( Y^{\{\s\}}_j \bigr), \quad \tau \in  [0, (1 - c_{s}) H], \\
\quad \textnormal{with } v_{s+1}(0) = Y^{\{\s\}}_{s}; \quad y_{n+1} \coloneqq v_{s+1}\bigl( (1 - c_{s})\,H \bigr).
\end{array}
\renewcommand{\arraystretch}{1.0}
\end{equation}
The modified fast ODE \eqref{eqn:MIS-one-step} is composed of the original fast component from \eqref{eqn:multirate-additive-ode} plus a piece-wise constant ``slow tendency''  term. Wensch, Knoth, and Galant \cite{Wensch_2009_MIS} generalized  \eqref{eqn:MIS-one-step} by adding linear combinations of stages ($Y^{\{\s\}}_{j}-y_n$, $1 \le j < i$) to both the initial condition $v_i(0)$ and to the slow tendency term. Since the fast ODE is solved exactly (or, equivalently, is solved numerically with an infinitesimally small step size) they aptly named the resulting schemes   ``multirate infinitesimal step'' methods. The methods were interpreted as exponential integrators, and order conditions up to order three were derived. Schlegel, Knoth, Wolke, and Arnold casted MIS schemes as traditional partitioned Runge-Kutta methods \cite{Schlegel_2012_MR-atmospheric}, and furthered the theoretical and practical understanding of this family   \cite{Knoth_2014_MR-Euler,Schlegel_2009_RFSMR,Schlegel_2010_MR-imex,Schlegel_2011_MR-implementation}.

The hybrid dynamical approach has proved fruitful not only in the construction of new numerical schemes, but also in the analysis of existing ones. For example, a hybrid approach has been used to abstract away the details of subsystem integration and study only the coupling aspects of waveform relaxation methods \cite{Guenther_2001_dynamic-iterations,Bartel_2013_dynamic-iterations}. The hybrid dynamical approach is also related to the engineering concept of co-simulation \cite{Gomes_2018_cosimulation,Schops_2010_cosimulation}.

G\"{u}nther and Sandu \cite{Sandu_2016_MR-GARK} explained MIS schemes in the framework of General-structure Additive Runge-Kutta (GARK) methods \cite{Sandu_2015_GARK}. A multirate GARK (MR-GARK) method \cite{Sandu_2016_MR-GARK} integrates the slow component with a Runge--Kutta method $(A^\slowslow,b^{\{\s\}})$ and a large step size $H$, and the fast component with another Runge--Kutta method $\bigl(A^\fastfast,b^{\{\f\}}\bigr)$ and a small step size $h = H/M$. Here $M \ge 1$ represents the (integer) number of fast steps that are executed for each of the slow steps. One step of the method computes $s^{\{\s\}}$ slow stages, denoted by $Y_i^{\{\s\}}$, and $M s^{\{\f\}}$ fast stages, denoted by $Y_i^{\{\f,\lambda\}}$:
\begin{subequations}
	\label{eqn:GARK-MR}
	\begin{align}
	\nonumber
		f_j^{\{\f,\lambda\}} &\coloneqq \funf\bigl(Y_j^{\{\f,\lambda\}}\bigr), \quad f_j^{\{\s\}} \coloneqq \funs\bigl(Y_j^{\{\s\}}\bigr); \\
				\begin{split}
	        \label{eqn:GARK-MR-slow-stage}
			Y_i^{\{\s\}} &= y_n + H \, \sum_{j=1}^{s^{\{\s\}}} a_{i,j}^\slowslow f_j^{\{\s\}} 
			+ h \, \sum_{\lambda=1}^M \,\sum_{j=1}^{s^{\{\f\}}} a_{i,j}^{\{\s,\f,\lambda\}} f_j^{\{\f,\lambda\}},
			\quad {1 \le i \le s^{\{\s\}}},
		\end{split} \\
		\begin{split}
			\label{eqn:GARK-MR-fast-stage}
			Y_i^{\{\f,\lambda\}} & = \widetilde{y}_{n+\frac{\lambda-1}{M}} + H \, \sum_{j=1}^{s^{\{\s\}}} a_{i,j}^{\{\f,\s,\lambda\}} f_j^{\{\s\}}  + h \, \sum_{j=1}^{s^{\{\f\}}} a_{i,j}^\fastfast f_j^{\{\f,\lambda\}},\quad { 1 \le i \le s^{\{\f\}}},
		\end{split} \\
	        \label{eqn:GARK-MR-final-solution}
	\widetilde{y}_{n+\frac{\lambda}{M}} &= \widetilde{y}_{n+\frac{\lambda-1}{M}} + h \sum_{i=1}^{s^{\{\f\}}} b_{i}^{\{\f\}} f_i^{\{\f,\lambda\}}, \quad {\lambda=1,\ldots,M}, \\
	y_{n+1} &= \widetilde{y}_{n+\frac{M}{M}} + H \, \sum_{i=1}^{s^{\{\s\}}} b_{i}^{\{\s\}} f_i^{\{\s\}}.
	\end{align}
\end{subequations}
%
The Butcher tableau for the MR-GARK method \eqref{eqn:GARK-MR} is \cite{Sandu_2016_MR-GARK}:
%
\begin{equation}
	\label{eqn:MR-GARK-butcher}
	\bgroup
	\def\arraystretch{\astretch}
	\begin{array}{c|c}
	\A^\fastfast & \A^\fastslow  \\ \hline
	\A^\slowfast & \A^\slowslow \\ \hline 
	\b^{\{\f\}}\,^T & \b^{\{\s\}}\,^T
	\end{array} ~~ \coloneqq~~ 
	\begin{array}{ccc|c}  
	\scriptstyle M^{-1} A^\fastfast      &  \scriptstyle        \cdots &\scriptstyle  0 &\scriptstyle  A^{\{\f,\s,1\}}  \\
	\scriptstyle \vdots                                  &\scriptstyle  \ddots &   &\scriptstyle  \vdots  \\
	\scriptstyle M^{-1} \one^{\{\f\}} \, b^{\{\f\}}\,^T  &\cdots &\scriptstyle  M^{-1} A^\fastfast &\scriptstyle A^{\{\f,\s,M\}} \\
	\hline 
	\scriptstyle M^{-1} A^{\{\s,\f,1\}}  &\cdots &\scriptstyle  M^{-1} A^{\{\s,\f,M\}} &\scriptstyle  A^\slowslow   \\   \hline 
	\scriptstyle M^{-1} b^{\{\f\}}\,^T  &\cdots &\scriptstyle  M^{-1} b^{\{\f\}}\,^T &\scriptstyle  b^{\{\s\}}\,^T
	\end{array}~.
	\egroup
\end{equation}

This work constructs a family of multirate infinitesimal GARK schemes (named MRI-GARK) that extend the hybrid dynamics approach of \cite{Knoth_1998_MR-IMEX,Wensch_2009_MIS} in several ways.  Time dependent ``slow tendency''  terms are used to construct the modified fast system. The MRI-GARK general order conditions theory is developed by leveraging the GARK accuracy theory \cite{Sandu_2015_GARK}. This allows the construction of the first fourth order multirate infinitesimal methods, as well as the  first  multirate infinitesimal schemes that are implicit in the slow component. Matrix stability analyses are carried out using a new simplified two-dimensional linear test problem. 

The paper is organized as follows. The new family of MRI-GARK schemes is defined in Section \ref{sec:MRI-GARK}. The order condition theory is developed in Section \ref{sec:order}, and the stability analysis in Section \ref{sec:stability}. Practical explicit methods of order up to four are presented in Section \ref{sec:explicit-methods}, and decoupled implicit methods in Section \ref{sec:implicit-methods}. Numerical results are reported in Section \ref{sec:numerics}, and conclusions are drawn in Section \ref{sec:conclusions}.


\section{Multirate Infinitesimal GARK Methods}
\label{sec:MRI-GARK}
This work constructs multirate infinitesimal schemes defined as follows.
\begin{definition}[Infinitesimal methods]
A multirate infinitesimal method for \eqref{eqn:multirate-additive-ode} is a solution process with the following ``infinitesimally exact'' property: if $\funs \equiv 0$ then the scheme reduces to an exact integration of the fast system $y'=\funf (t,y)$.  
\end{definition}
Since we do not require the limit of step sizes to approach zero for exact integration of the fast subsystem, it follows that infinitesimal methods are hybrid solution procedures that may discretize the slow dynamics, but advance the fast dynamics in continuous time. We note that while the original MIS method \cite{Knoth_1998_MR-IMEX} is ``infinitesimally exact'', its extensions \cite{Wensch_2009_MIS} do not automatically enjoy this property.

Construction of the methods starts with a ``slow'' Runge-Kutta base method with non-increasing abscissae:
\begin{subequations}
\begin{equation}
\label{eqn:slow-base-scheme}
\renewcommand{\arraystretch}{1.3}
\begin{split}
&\mathsf{A}^\slow  \equiv
\begin{array}{c | c}
\cbase  & \Abase \\
\hline
1 & \bbase\,^T \\
\hline
1 & \bhatbase\,^T 
\end{array}~,
\qquad 
\begin{array}{l}
\cbase_1 \le \cbase_2 \le \cdots \le \cbase_{s^\slow} \le 1, \end{array}
\\
& \abase_{s^\slow+1,j} \coloneqq \bbase_{j}, \quad \cbase_{s^\slow+1} \coloneqq 1, \quad
\abase_{s^\slow+2,j} \coloneqq \bhatbase_{j}, \quad \cbase_{s^\slow+2} \coloneqq 1,
\end{split}
\end{equation}
and define the increments between consecutive stages of the base method:
\begin{equation}
\label{eqn:abscissa-increments}
\deltac_i \coloneqq
\begin{cases}
\cbase_{i+1} - \cbase_{i}, & i=1, \dots, s^\slow-1, \\
1 - \cbase_{s^\slow}, & i=s^\slow, \\
0, &  i=s^\slow+1.
\end{cases}
\end{equation}
\end{subequations}
\begin{definition}[MRI-GARK methods for additively partitioned systems]
A  Multirate Infinitesimal GARK (MRI-GARK) scheme applied to the additively partitioned system \eqref{eqn:multirate-additive-ode} advances the solution from $t_n$ to $t_{n+1} = t_{n}+H$ as follows:
\begin{subequations}
\label{eqn:MIS-additive}
\begin{eqnarray}
\label{eqn:MIS-additive-first-stage}
\qquad&& \Ys_{1} = y_{n}, \\
\label{eqn:MIS-additive-internal-ode}
&& \begin{cases}
v(0) = \Ys_{i}, \\
T_i = t_n + \cbase_i\, H, \\
v' = \deltac_i  \, \funf \left(\,T_i + \deltac_i \theta,\, v\, \right) + \sum_{j=1}^{i+1} \gamma_{i,j}\left({\scriptstyle \frac{\theta}{H}}\right) \, \funs\bigl(\,T_j,\, Y^\slow_j \, \bigr) \\
\qquad  \textnormal{for  } \theta \in  [0,  H],  \\
\Ys_{i+1} = v(H), \qquad i=1,\dots,s^\slow,
\end{cases}
\\
\label{eqn:MIS-additive-solution}
&& y_{n+1} = \Ys_{s^\slow+1}\,.
\end{eqnarray}
\end{subequations}
Linear combinations of the slow function values are added as forcing to the modified fast ODE system \eqref{eqn:MIS-additive-internal-ode}; in order to use only already computed slow stages one needs $\gamma_{i,j}(\tau) = 0$ for $j>i$.
\end{definition}

\begin{definition}[Slow tendency coefficients]
It is convenient to define the time-dependent combination coefficients as polynomials in time, and to consider their integrals:
\begin{equation}
\label{eqn:gamma-as-power-series}
\gamma_{i,j}( \tau ) \coloneqq \sum_{k \ge 0} \gamma_{i,j}^k\,\tau^k,
\quad \widetilde{\gamma}_{i,j}\left( x \right) \coloneqq \int_{0}^{x} \gamma_{i,j}\left(\tau\right) d\tau = \sum_{k \ge 0} \gamma_{i,j}^k\,\frac{x^{k+1}}{k+1}\,, \quad
\overline{\gamma}_{i,j} \coloneqq \widetilde{\gamma}_{i,j}\left( 1 \right).
\end{equation}
%
%
\end{definition}

\begin{remark}[Equal abscissae]
When two consecutive abscissae are equal,  $\cbase_{i+1} = \cbase_{i}$, the computation \eqref{eqn:MIS-additive-internal-ode} reduces to an explicit slow Runge-Kutta stage:
\begin{equation*}
\begin{split}
\Ys_{i+1} &= \Ys_{i} + H\,\sum_{j=1}^{i+1} \left( \int_{\tau=0}^1 \gamma_{i,j}\left(\tau\right)\right) \, \funs\bigl( Y^\slow_j \bigr) = \Ys_{i} + H\,\sum_{j=1}^{i+1} \overline{\gamma}_{i,j} \, \funs\bigl( Y^\slow_j \bigr),
\end{split}
\end{equation*}
whenever $\overline{\gamma}_{i,j} = 0$ for $j > i$.
We note that it is possible to choose $\overline{\gamma}_{i,i+1} \ne 0$, in which case \eqref{eqn:MIS-additive-internal-ode}  is equivalent to a singly diagonally implicit Runge-Kutta stage.
\end{remark}

\begin{remark}[Embedded method]
\label{rem:MRI-GARK-embedded}
The base slow scheme \eqref{eqn:slow-base-scheme} computes the main solution using the weights $b^\slow$, and the embedded solution using the weights $\widehat{b}^\slow$. The MRI-GARK scheme computes the main solution \eqref{eqn:MIS-additive-solution} via the last stage \eqref{eqn:MIS-additive-internal-ode}, that advances $\Ys_{s^\slow}$ to $y_{n+1}$ by solving a modified fast system with the slow weights $ \gamma_{s^\slow,j}$. An embedded solution is computed via an additional stage that advances $\Ys_{s^\slow}$ to $\widehat{y}_{n+1}$ using with the modified weights $\widehat{\gamma}_{s^\slow,j}$.
\ifreport
\begin{eqnarray*}
&& \begin{cases}
v(0) = \Ys_{s^\slow}, \\
 v' = \deltac_{s^\slow} \, \funf \left( v \right) + \sum_{j=1}^{s^\slow} \widehat{\gamma}_{s^\slow,j}\left(\frac{\theta}{H}\right) \, \funs\bigl( Y^\slow_j \bigr)
\quad  \textnormal{for}~~ \theta \in  [0,  H],  \\
\widehat{y}_{n+1} = v(H).
\end{cases}
\end{eqnarray*}
\fi
\end{remark}

\begin{definition}[MRI-GARK methods for component partitioned systems]

Consider the component partitioned system:
\begin{equation} 
\renewcommand{\arraystretch}{\astretch}
	\label{eqn:multirate-component-ode}
	 \begin{bmatrix} \yf \\ \ys  \end{bmatrix}'= 
	 \begin{bmatrix} \funf\bigl(t, \yf, \ys \bigr) \\  \funs\bigl(t, \yf, \ys \bigr) \end{bmatrix}, 
	 \quad 
	 \begin{bmatrix} \yf(t_0) \\ \ys(t_0)  \end{bmatrix}= \begin{bmatrix}  \yf_0 \\ \ys_0 \end{bmatrix}.
\end{equation}
An MRI-GARK scheme applied to \eqref{eqn:multirate-component-ode} computes the solution as follows:
\begin{subequations}
\label{eqn:MIS-component}
\begin{eqnarray}
\label{eqn:MIS-component-first-stage}
&\qquad\qquad&\Yf_{1} =  \yf_{n}, \quad \Ys_{1} = \ys_{n}, \\
\label{eqn:MIS-component-stages}
&&\begin{cases}
v^\fast_i(0) = \Yf_{i}, \\
v^\fast_i\,' = \deltac_i  \, \funf \bigg(\, T_i + \deltac_i \,\theta,\, v^\fast_i,\\
\qquad\qquad \Ys_{i}+H\,\sum_{j=1}^{i+1}  \widetilde{\gamma}_{i,j}\left({\scriptstyle \frac{\theta}{H}}\right) \, \funs\bigl(\,T_j,\, Y^\fast_j, \, Y^\slow_j\, \bigr) \bigg) 
~~ \textnormal{for  } \theta \in  [0,  H],  \\
\Yf_{i+1} = v^\fast_i(H), \\ 
\Ys_{i+1}  = \Ys_{i}+H\,\sum_{j=1}^{i+1}  \overline{\gamma}_{i,j}\, \funs\bigl(\,T_j,\, Y^\fast_j, \, Y^\slow_j\, \bigr), \qquad i=1,\dots,s^\slow,
\end{cases} \\
\label{eqn:MIS-component-solution}
&&\yf_{n+1} =\Yf_{s^\slow+1}, \quad  \ys_{n+1} = \Ys_{s^\slow+1}.
\end{eqnarray}
\end{subequations}
For $\deltac_i \ne 0$ we require that $\gamma_{i,i+1}(\tau)=0$ such that only previously computed slow stage values $\Ys_{j}$ are used in the formulation of the modified fast ODE. The additive \eqref{eqn:MIS-additive} and component-based \eqref{eqn:MIS-component} formulations are equivalent to each other.
\end{definition}

\ifreport
\begin{remark}[Equivalency of the additive and component based formulations]
Writing the system \eqref{eqn:multirate-component-ode} in additively partitioned form \eqref{eqn:multirate-additive-ode}:
\begin{equation} 
\renewcommand{\arraystretch}{\astretch}
	\label{eqn:multirate-component-ode-as-additive}
	 \begin{bmatrix} \yf \\ \ys  \end{bmatrix}'= 
	 \begin{bmatrix} 0 \\  \funs\bigl(t, \yf, \ys \bigr) \end{bmatrix}
	 +
	 \begin{bmatrix} \funf\bigl(t, \yf, \ys \bigr) \\  0 \end{bmatrix},
\end{equation}
and applying the solution method \eqref{eqn:MIS-one-step} leads to:
\begin{subequations}
\label{eqn:MIS-component-full}
\begin{eqnarray}
\label{eqn:MIS-component-first-stage-full}
&& \begin{bmatrix} \Yf_{1} \\ \Ys_{1}  \end{bmatrix} = \begin{bmatrix} \yf_{n} \\ \ys_{n} \end{bmatrix}, \\
\label{eqn:MIS-component-internal-ode}
&& \begin{cases}
v^\fast_i(0) = \Yf_{i}, \\
v^\slow_i(0) = \Ys_{i}, \\ 
 T_{i+1} = t_n + \cbase_{i+1}\,H , \\ 
 v^\fast_i\,' = \deltac_i  \, \funf \left(\, T_i + \deltac_i \,\theta\,,\,v^\fast_i, \,v^\slow_i\, \right), \\
v^\slow_i\,' = \sum_{j=1}^{i+1} \gamma_{i,j}\left({\scriptstyle \frac{\theta}{H}}\right) \, \funs\bigl(\, T_j,\,Y^\fast_j,\,Y^\slow_j\, \bigr) 
\quad  \textnormal{for}~~ \theta \in  [0,  H],  \\
 \Yf_{i+1} = v^\fast_i(H), \\ 
 \Ys_{i+1}  = v^\slow_i(H), \qquad i=1,\dots,s^\slow,
\end{cases}
\\
\label{eqn:MIS-component-solution-full}
&&  \begin{bmatrix} \yf_{n+1} \\ \ys_{n+1} \end{bmatrix} = \begin{bmatrix} \Yf_{s^\slow+1} \\ \Ys_{s^\slow+1}  \end{bmatrix}\,.
\end{eqnarray}
\end{subequations}
We have that:
\begin{equation*}
\begin{split}
v^\slow_i(\theta) &= \Ys_{i}+\sum_{j=1}^{i+1} \left(\int_{\tau=0}^{\theta} \gamma_{i,j}\left({\scriptstyle \frac{\tau}{H}}\right) d\tau\right)\, \funs\bigl(\,T_j,\, Y^\fast_j, \, Y^\slow_j\, \bigr) \\
&= \Ys_{i}+H\,\sum_{j=1}^{i+1}  \widetilde{\gamma}_{i,j}\left({\scriptstyle \frac{\theta}{H}}\right) \, \funs\bigl(\,T_j,\, Y^\fast_j, \, Y^\slow_j\, \bigr).
\end{split}
\end{equation*}
This leads directly to \eqref{eqn:MIS-component}.
\end{remark}
\fi

\begin{example}[Second order MRI-GARK methods]
\label{example:second-order-methods}
Consider the explicit midpoint method and the implicit trapezoidal method, each paired with a first order embedded scheme:
\[
\renewcommand{\arraystretch}{1.3}
\mathsf{A}^\slow_\textsc{emidp}  ~=~
\begin{array}{c | cc}
\scriptstyle   0 &\scriptstyle   0 &\scriptstyle   0 \\
 \frac{1}{2} & \frac{1}{2} &\scriptstyle   0 \\
 \hline
\scriptstyle   1 & \scriptstyle 0 & \scriptstyle 1 \\
 \hline
\scriptstyle  1 &\scriptstyle   1 &\scriptstyle   0
\end{array},
\qquad
\mathsf{A}^\slow_\textsc{itrap} ~=~
\begin{array}{c | cc}
\scriptstyle 0 &\scriptstyle  0 &\scriptstyle  0 \\
\scriptstyle 1 &   \frac{1}{2} & \frac{1}{2} \\ 
 \hline
\scriptstyle 1 &   \frac{1}{2} & \frac{1}{2} \\
\hline
\scriptstyle 1 &\scriptstyle  0 &\scriptstyle  1
\end{array}.
\]
The explicit midpoint method  is the slow component \eqref{eqn:slow-base-scheme} of the following second order {\it explicit} MRI-GARK scheme \eqref{eqn:MIS-additive}:
%
%
\begin{equation}
\label{eqn:explicit-midpoint}
\begin{split}
v_1(0) &= y_{n}; \quad
v_1' = {\scriptstyle \frac{1}{2}} \, \funf \left( v_1 \right) + {\scriptstyle \frac{1}{2}}\, \funs\bigl( y_{n} \bigr),
~~ \theta \in  [0,  H];
\quad \Ys_{2} = v_1(H), \\
v_2(0) &= \Ys_{2}; \quad
v_2' = {\scriptstyle \frac{1}{2}} \, \funf \left( v_2 \right) -{\scriptstyle \frac{1}{2}} \, \funs\bigl( y_{n} \bigr)
+  \funs\bigl( Y^\slow_2 \bigr),~~ \theta \in  [0,  H],  \\
y_{n+1} & = v_2(H), \\
v_3(0) &= \Ys_{2}; \quad
v_3' = {\scriptstyle \frac{1}{2}} \, \funf \left( v_3 \right) + {\scriptstyle \frac{1}{2}} \, \funs\bigl( y_{n} \bigr),
~~\theta \in  [0,  H];\quad
\widehat{y}_{n+1} =  v_3(H).
\end{split}
\end{equation}
The implicit trapezoidal method is the slow component \eqref{eqn:slow-base-scheme} of the following second order {\it implicit} MRI-GARK scheme \eqref{eqn:MIS-additive}:
%
\begin{equation}
\label{eqn:implicit-trapezoidal}
\begin{split}
v(0) &=y_{n}; \quad
v' = \funf \left( v \right) + \funs\bigl( y_{n} \bigr),
~~ \theta \in  [0,  H];
\quad \Ys_{2} = v(H), \\
y_{n+1} &= \Ys_{2} - {\scriptstyle \frac{1}{2}} \, \funs \left( y_{n} \right) + {\scriptstyle \frac{1}{2}} \, \funs\bigl( y_{n+1} \bigr); \\
\widehat{y}_{n+1} &= \Ys_{2} -  \funs \left( y_{n} \right) + \funs\bigl( y_{n+1} \bigr).
\end{split}
\end{equation}

\end{example}


\section{Order conditions}
\label{sec:order}
For accuracy analysis we replace the continuous fast process by a discrete Runge Kutta method $\bigl(A^\fastfast,b^\fast,c^\fast\bigr)$ of arbitrary accuracy and having an arbitrary number of stages $s^\fast$. This approach casts the MRI-GARK scheme \eqref{eqn:MIS-additive} into the multirate GARK framework \eqref{eqn:GARK-MR}, where each sub-step is carried out between one slow stage and the next.  We denote the fast sub-steps by $\lambda=1,\dots,s^\slow$, where each advances the fast system from $t_n + \cbase_{\lambda}\,H$ to $t_n + \cbase_{\lambda+1}\,H$.

\begin{remark}[Matrix notation]
It is convenient to gather the gamma coefficients in the matrices $\boldsymbol{\Gamma}^k \coloneqq \bigl[ \gamma_{i,j}^k \bigr]_{i,j } \in \Re^{s^\slow \times s^\slow}$, and express  \eqref{eqn:gamma-as-power-series} as follows:
\[
\boldsymbol{\Gamma}(\tau) = \sum_{k \ge 0} \boldsymbol{\Gamma}^k\,\tau^k, \quad
\widetilde{\boldsymbol{\Gamma}}\left( x \right) = \sum_{k \ge 0} \frac{x^{k+1}}{k+1} \, \boldsymbol{\Gamma}^k, \quad
\overline{\boldsymbol{\Gamma}} = \sum_{k \ge 0} \frac{1}{k+1}\,\boldsymbol{\Gamma}^k.
\]
The structure of the coefficient matrices is lower triangular, in the sense that:
\begin{equation}
\label{eqn:gamma-structure}
\gamma_{i,j}^k = 0 ~ \textnormal{for} ~ j \ge i+2 \quad
\textnormal{and} \quad \gamma_{i,i+1}^k\cdot\Delta  \cbase_i = 0,
\quad i=1,\dots,s^\slow-1, ~ k \ge 0.
\end{equation}
\ifreport 
Consequently:
\begin{itemize}
\item For explicit slow base methods \eqref{eqn:slow-base-scheme} the coefficient matrices are lower triangular, $\gamma_{i,j}^k = 0$ for $j \ge i+1$;
\item For diagonally implicit slow base methods \eqref{eqn:slow-base-scheme} one allows $\gamma_{i,i+1}^k \ne 0$ when $\deltac_i = 0$.
\end{itemize}
\fi 
\end{remark}

\begin{lemma}[The MRI-GARK method \eqref{eqn:MIS-additive} as a particular instance of a GARK method]
Consider the MRI-GARK scheme~\eqref{eqn:MIS-additive} where the continuous fast process is replaced by a discrete Runge Kutta method $\bigl(A^\fastfast,b^\fast,c^\fast\bigr)$ of arbitrary accuracy and having an arbitrary number of stages $s^\fast$. The resulting computational process is a multirate GARK method \eqref{eqn:GARK-MR} with the following Butcher tableau \eqref{eqn:MR-GARK-butcher} components.
\paragraph{Fast component}
\begin{subequations}
\label{eqn:MRI-butcher-fast}
\begin{eqnarray}
\label{eqn:MRI-butcher-Aff}
\qquad 
\mathbf{A}^\fastfast  &=& 
\begin{bmatrix}
\deltac_1\, A^\fastfast      &     \scriptstyle       0                   & \cdots &\scriptstyle  0 \\
\deltac_1\,  \one^\fast  b^\fast\,^T &  \deltac_2\, A^\fastfast        & \cdots &\scriptstyle  0 \\
\vdots                     &         \vdots                    & \ddots &    \vdots \\
\deltac_1\,  \one^\fast  b^\fast\,^T & \deltac_2\, \one^\fast  b^\fast\,^T   & \ldots &  \deltac_{s^{\{s\}}}\, A^\fastfast 
\end{bmatrix} \in \Re^{s \times s}, \\[6pt]
%
%
%
%
\label{eqn:MRI-butcher-bf}
\mathbf{b}^\fast &=& \deltac \otimes b^\fast 
\ifreport 
=
\begin{bmatrix}
\deltac_1 \,  b^\fast  \\
\vdots \\
\deltac_i\, b^\fast \\
\vdots \\
\deltac_{s^{\{s\}}} \, b^\fast
\end{bmatrix} 
\fi 
\in \Re^{s}, \\
%
%
\label{eqn:MRI-butcher-cf}
\mathbf{c}^\fastfast &=& \mathbf{A}^\fastfast \, \one_{s \times 1} 
\ifreport 
=
\begin{bmatrix}
\deltac_1\,c^\fast  \\
\vdots \\
c_{i}^\slow  \, \one^\fast +  \deltac_i\,c^\fast \\
\vdots \\
c_{s^\slow}^\slow  \, \one^\fast +  \deltac_{s^{\{s\}}}\,c^\fast
\end{bmatrix}
\fi 
= \cbase  \otimes \one^\fast + \deltac \otimes c^\fast \in \Re^{s}. 
%
\end{eqnarray}
\end{subequations}
Here $s = s^\slow\cdot s^\fast$ is the total number of stages of the method and $\otimes$ is the Kronecker product.
\ifreport 
From \eqref{eqn:MRI-butcher-bf} and \eqref{eqn:MRI-butcher-cf} we see immediately that:
\begin{equation*}
\mathbf{b}^\fast\,^T \, \mathbf{c}^\fastfast
= \sum_i    \deltac_i\, b^\fast\,^T\,\left(\cbase_{i}  \, \one^\fast + \deltac_i\,c^\fast \right) = \frac{1}{2}.
\end{equation*}
\fi 
\paragraph{Slow component}

\begin{subequations}
\label{eqn:MRI-butcher-slow}
\begin{eqnarray}
\label{eqn:MRI-butcher-Ass}
\qquad\qquad \A^\slowslow &\equiv&  \Abase = \mathbf{E}\,\overline{\boldsymbol{\Gamma}} \quad \textnormal{where} \quad
\mathbf{E} \in \Re^{s^\slow \times s^\slow}, \quad
\mathbf{E}_{i,j} = \begin{cases}
1 & i \ge j+1, \\
0 & \textnormal{otherwise},
\end{cases} \\
\label{eqn:MRI-butcher-bs}
\b^\slow\,^T &\equiv&\bbase\,^T = \one^\slow\,^T \, \overline{\boldsymbol{\Gamma}}, \\
\label{eqn:MRI-butcher-css}
\c^\slowslow &=& \A^\slowslow\, \one^\slow \equiv \cbase =  \mathbf{E}\,\overline{\boldsymbol{\Gamma}}\, \one^\slow.
\end{eqnarray}
\end{subequations}
\paragraph{Slow-fast coupling}

\begin{subequations}
\label{eqn:MRI-butcher-slow-fast}
\begin{eqnarray}
\label{eqn:MRI-butcher-Asf}
\qquad \mathbf{A}^\slowfast &=& \Bigl[\mathbf{A}^{\{\s,\f,1\}} ~ \cdots  \mathbf{A}^{\{\s,\f,s^\slow\}} \Bigr]  \in \Re^{s^\slow \times s}, \\
\nonumber
\mathbf{A}^{\{\s,\f,i\}} &=& 
\begin{cases}
\deltac_{i}\, \mathbf{g}_{i+1}\, b^\fast\,^T, & i=1,\dots, s^\slow-1, \\
\boldsymbol{0}_{s^\slow \times s^\fast}, & i=s^\slow,
\end{cases} \\
\label{eqn:MRI-butcher-csf}
\c^\slowfast &=& \mathbf{A}^\slowfast\,\one_{s \times 1} = \cbase,
\end{eqnarray}
\end{subequations}
where for each $i=1,\dots,s^\slow$ we define:
\[
\mathbf{e}_i^T \coloneqq \begin{blockarray}{ccccc}
\scriptstyle 1 &          &\scriptstyle   i  &   &\scriptstyle  s^\slow \\
\begin{block}{[ccccc]}
0 & \cdots &  1 &   \cdots  &  0 \\ \end{block}\end{blockarray}, \quad
\mathbf{g}_i^T \coloneqq \begin{blockarray}{cccccc}
\scriptstyle 1 &          &\scriptstyle  i-1 &\scriptstyle  i  &   &\scriptstyle  s^\slow \\
\begin{block}{[cccccc]}
0 & \cdots &  0 & 1 &   \cdots  &  1 \\ \end{block}\end{blockarray}.
\]

\paragraph{Fast-slow coupling}
\begin{subequations}
\label{eqn:MRI-butcher-fast-slow}
\begin{eqnarray}
\label{eqn:MRI-butcher-Afs}
\mathbf{A}^\fastslow & = &  \Abase \otimes \one^\fast +   \sum_{k \ge 0} \boldsymbol{\Gamma}^k \otimes \bigl( A^\fastfast \, c^\fast\,^{\times k} \bigr),
 \\
\label{eqn:MRI-butcher-cfs}
\mathbf{c}^\fastslow  &=& \mathbf{A}^\fastslow\,\one^\slow =
\cbase  \otimes \one^\fast +   \sum_{k \ge 0}  \bigl( \boldsymbol{\Gamma}^k\,\one^\slow \bigr) \otimes \bigl(A^\fastfast \, c^\fast\,^{\times k}\bigr).
\end{eqnarray}
\end{subequations}
\ifreport 
This is detailed as follows:
\begin{eqnarray*}
\mathbf{A}^\fastslow & = & 
\begin{bmatrix}
\mathbf{A}^{\{\f,\s,1\}} \\
\vdots \\
\mathbf{A}^{\{\f,\s,s^\slow\}}
\end{bmatrix}, \qquad
\mathbf{c}^\fastslow  = 
\begin{bmatrix}
\mathbf{c}^{\{\f,\s,1\}} \\
\vdots \\
\mathbf{c}^{\{\f,\s,s^\slow\}}
\end{bmatrix}, 
\\
\A^{\{\f,\s,i\}} &=& \one^\fast\, \Abase_{i,:} +   (A^\fastfast \, c^\fast\,^{\times k})\,\sum_{k \ge 0} \boldsymbol{\Gamma}_{i,:}^k,
\quad i=1,\dots,s^\slow, \\
\c^{\{\f,\s,i\}} &=&\cbase_{i}\,\one^\fast \, 
+ \sum_{k \ge 0}  \left( \boldsymbol{\Gamma}^k\,\one^\slow \right)_{i} \,(A^\fastfast \, c^\fast\,^{\times k}).
\end{eqnarray*}
Here $\Abase_{i,:},\boldsymbol{\Gamma}_{i,:}^k$ represent the $i$-th rows of the matrices, and $\left( \boldsymbol{\Gamma}^k\,\one^\slow \right)_{i}$
the $i$-th element of the vector.
\fi 
\end{lemma}

\begin{proof}
Replacing the $i$-th continuous stage \eqref{eqn:MIS-additive-internal-ode} by the equivalent $i$-th fast discrete sub-step to advance from $t_n + \cbase_{i}\,H$ to $t_n + \cbase_{i+1}\,H$ leads to:
\begin{subequations}
\label{eqn:MRI_discrete}
\begin{equation}
\label{eqn:MRI_discrete_fast_stage}
\begin{split}
V_{k}^{\{\f,i\}} 
%
%
& =  Y_{i}^\slow + H\, \deltac_{i}\, \sum_{\ell=1}^{s^\fast} a^\fastfast_{k,\ell} \, \funf \bigl( V_{\ell}^{\{\f,i\}} \bigr) \\
&\quad + H\, \sum_{j=1}^{i+1} \, \left(\sum_{\ell=1}^{s^\fast} a^\fastfast_{k,\ell} \,   \gamma_{i,j}( c_\ell^\fast ) \right)\, \funs \bigl( Y_j^\slow \bigr), \quad k=1,\dots,s^\fast,
\end{split}
\end{equation}
where $V_{k}^{\{\f,i\}}$ is the $k$-th stage of the fast discrete sub-step $i$.
The corresponding slow stages are advanced as follows:
%
\begin{equation}
\label{eqn:MRI_discrete_slow_stage}
\begin{split}
Y_{i+1}^\slow &= Y_{i}^\slow +  H\, \sum_{j=1}^{i+1} \left( \sum_{\ell=1}^{s^\fast} b^\fast_{\ell} \, \gamma_{i,j}( c_\ell^\fast ) \right)\, \funs \bigl( Y_j^\slow \bigr) \\
&\quad + H\, \deltac_{i}\, \sum_{\ell=1}^{s^\fast} b_{\ell}^\fast\, \funf \bigl( V_{\ell}^{\{\f,i\}} \bigr),  \quad i=1,\dots,s^\slow.
\end{split}
\end{equation}
%
%
Iterating after the slow stages \eqref{eqn:MRI_discrete_slow_stage} yields:
\ifreport 
\begin{equation*}
\begin{split}
 Y_{i}^\slow  
 &=  y_n +  H\,\sum_{\lambda=1}^{i-1} \,\sum_{j=1}^{\lambda+1} \left( \sum_{\ell=1}^{s^\fast} b^\fast_{\ell} \, \gamma_{\lambda,j}( c_\ell^\fast ) \right)\, \funs \bigl( Y_j^\slow \bigr) \\
&\quad + H\,\sum_{\lambda=1}^{i-1} \, \sum_{j=1}^{s^\fast} \deltac_{\lambda}\, b_{j}^\fast\, \funf \bigl( V_{j}^{\{\f,\lambda\}} \bigr)
\end{split}
\end{equation*}
and therefore:
\fi 
\begin{equation}
\label{eqn:MRI_discrete_slow_stage_iterated}
\begin{split}
Y_{i}^\slow  
&=  y_n + H\, \sum_{j=1}^{i} \sum_{p=\max(j-1,1)}^{i-1} \left( \sum_{\ell=1}^{s^\fast} b^\fast_{\ell} \, \gamma_{p,j}( c_\ell^\fast ) \right)\, \funs \bigl( Y_j^\slow \bigr) \\
&\quad + H\, \sum_{\lambda=1}^{i-1}  \sum_{j=1}^{s^\fast} \deltac_{\lambda}\, b_{j}^\fast\, \funf \bigl( V_{j}^{\{\f,\lambda\}} \bigr) \\
&=  y_n + H\, \sum_{j=1}^{i} a_{i,j}^\slowslow\, \funs \bigl( Y_j^\slow \bigr) 
+ H\, \sum_{\lambda=1}^{i-1}  \sum_{j=1}^{s^\fast} a_{i,j}^{\{\s,\f,\lambda\}}\, \funf \bigl( V_{j}^{\{\f,\lambda\}} \bigr),
\end{split}
\end{equation}
\end{subequations}
where the last equation writes \eqref{eqn:MRI_discrete_slow_stage_iterated} as a slow MRI-GARK stage \eqref{eqn:GARK-MR-slow-stage}. From this we identify the slow method coefficients as:
\begin{eqnarray*}
a_{i,j}^\slowslow &\coloneqq& \begin{cases}
0 & i=1, \\
\sum_{p=\max(j-1,1)}^{i-1}  b^\fast\,^T \, \gamma_{p,j}\big( c^\fast \big)  & 1 \le j \le i, \\
0 & i+1 \le j \le s^\slow,
\end{cases} \\
b_{j}^\slow &\coloneqq& \sum_{p=j}^{s^\slow}  b^\fast\,^T \, \gamma_{p,j}\big( c^\fast \big).
\end{eqnarray*}
Using \eqref{eqn:gamma-as-power-series} and the fact that
$b^\fast\,^T \, c^{\{\f\} \times k} = 1/(k+1)$ for any $k$ due to the arbitrary accuracy of the fast scheme, we have that:
\[
b^\fast\,^T \, \gamma_{p,j}\big( c^\fast \big) = \sum_{k \ge 0} \gamma_{p,j}^k\, b^\fast\,^T \, c^{\{\f\} \times k} = \sum_{k \ge 0} \frac{\gamma_{p,j}^k}{k+1} = \overline{\gamma}_{p,j}.
\]
Consequently, the slow method coefficients satisfy:
\begin{eqnarray*}
a_{i,j}^\slowslow &=& \begin{cases}
0 & i=1, \\
\sum_{p=\max(j-1,1)}^{i-1}  \overline{\gamma}_{p,j}  & 1 \le j \le i,~ i \ge 2, \\
0 & i \le j \le s^\slow,
\end{cases}, \qquad
b_{j}^\slow = \sum_{p=j}^{s^\slow} \overline{\gamma}_{p,j},
\end{eqnarray*}
which establishes the equations \eqref{eqn:MRI-butcher-slow}.

We identify the slow-fast coupling coefficients as follows:
\begin{eqnarray*}
a_{i,j}^{\{\s,\f,\lambda\}} &\coloneqq& \begin{cases}
\deltac_{\lambda}\, b_{j}^\fast & 1 \le \lambda \le i-1, \\
0 & i \le \lambda \le s^\slow,
\end{cases} \\
A^{\{\s,\f,\lambda\}} &=& 
\ifreport 
\deltac_{\lambda}  
\begin{blockarray}{cc}
\begin{block}{[c]c}
 0 &\scriptstyle 1 \\ \vdots & \\ b^\fast\,^T &\scriptstyle  \lambda+1 \\ \vdots & \\   b^\fast\,^T &\scriptstyle  s^\slow \\ \end{block}\end{blockarray} 
 \fi 
= \deltac_{\lambda}\, \mathbf{g}_{\lambda+1}\,b^\fast\,^T 
\in \Re^{s^\slow \times s^\fast}, \quad \lambda=1,\dots,s^\slow,
\end{eqnarray*}
which establishes \eqref{eqn:MRI-butcher-Asf}. Moreover:
\begin{eqnarray*}
\c^\slowfast &=& 
\sum_{\lambda=1}^{s^\slow-1} A^{\{\s,\f,\lambda\}}\,\one^\fast
=\sum_{\lambda=1}^{s^\slow-1} \deltac_{\lambda}\, \mathbf{g}_{\lambda+1}, \\
\c^\slowfast_i &=& \sum_{\lambda=1}^{s^\slow-1} \left( \cbase_{\lambda+1} - \cbase_{\lambda} \right)\, \left(\mathbf{g}_{\lambda+1}\right)_i
=  \sum_{\lambda=1}^{i-1} \left( \cbase_{\lambda+1} - \cbase_{\lambda} \right) = \cbase_i, 
\end{eqnarray*}
which proves \eqref{eqn:MRI-butcher-csf}.

Replacing \eqref{eqn:MRI_discrete_slow_stage_iterated} into the discrete fast stage equations \eqref{eqn:MRI_discrete_fast_stage} leads to:
\begin{eqnarray*}
V_{k}^{\{\f,i\}} & = & y_n + H\, \sum_{\lambda=1}^{i-1}  \sum_{j=1}^{s^\fast} \deltac_{\lambda}\, b_{j}^\fast\, \funf \bigl( V_{j}^{\{\f,\lambda\}} \bigr)  + H\,\sum_{j=1}^{s^\fast} a^\fastfast_{k,j} \, \deltac_{i}\,  \funf \bigl( V_{j}^{\{\f,i\}} \bigr) \\
&& + H\, \sum_{j=1}^{i} a^\slowslow_{i,j}\, \funs \bigl( Y_j^\slow \bigr)
+ H\, \sum_{j=1}^{i} \, \left(\sum_{\ell=1}^{s^\fast} a^\fastfast_{k,\ell} \,   \gamma_{i,j}( c_\ell^\fast ) \right)\, \funs \bigl( Y_j^\slow \bigr), \\
%
& = & y_n + H\, \sum_{\lambda=1}^{i}  \sum_{j=1}^{s^\fast} a_{k,j}^{\{\f,\f,i,\lambda\}}\, \funf \bigl( V_{j}^{\{\f,\lambda\}} \bigr) + H\, \sum_{j=1}^{i} a_{k,j}^{\{\f,\s,i\}} \, \funs \bigl( Y_j^\slow \bigr).
\end{eqnarray*}
where the last equation is the generic MR-GARK fast stage \eqref{eqn:GARK-MR-fast-stage}. We identify the fast scheme coefficients as:
\ifreport 
\begin{eqnarray*}
a_{k,j}^{\{\f,\f,i,\lambda\}} &\coloneqq & 
\begin{cases}
\deltac_{\lambda}\, b_{j}^\fast & 1 \le \lambda \le i-1,\\
\deltac_{i}\, a^\fastfast_{k,j} & \lambda = i, \\
0 & i+1 \le \lambda \le s^\slow,
\end{cases}
\end{eqnarray*}
or
\fi 
\begin{eqnarray*}
A^{\{\f,\f,i,\lambda\}} &\coloneqq & 
\begin{cases}
\deltac_{\lambda}\, \one^\fast\,b^\fast\,^T & 1 \le \lambda \le i-1,\\
\deltac_{i}\, A^\fastfast & \lambda = i, \\
0 & i+1 \le \lambda \le s^\slow,
\end{cases}
\quad i = 1,\dots,s^\slow,
\end{eqnarray*}
which shows \eqref{eqn:MRI-butcher-Aff}, and
\begin{equation}
\label{eqn:MRI-butcher-cffi}
\c^{\{\f,\f,i\}} \coloneqq \sum_{\lambda=1}^i A^{\{\f,\f,i,\lambda\}}\,\one^\fast =\cbase_{i}  \, \one^\fast +  \deltac_{i} \,c^\fast, 
\end{equation}
which establishes \eqref{eqn:MRI-butcher-cf}.
Moreover:
\begin{eqnarray*}
\mathbf{b}^{\{\f,i\}}\,^T & = & \deltac_{i}\, b^\fast\,^T,
\end{eqnarray*}
which proves \eqref{eqn:MRI-butcher-bf}.

For the fast-slow coupling we have that:
\ifreport 
\begin{eqnarray*}
a_{k,j}^{\{\f,\s,i\}} \coloneqq \begin{cases}
a^\slowslow_{i,j} + \sum_{\ell=1}^{s^\fast}  a^\fastfast_{k,\ell} \,   \gamma_{i,j}( c_\ell^\fast ) & 1 \le j \le i, \\
0 & i+1 \le j \le s^\slow,
\end{cases}
\end{eqnarray*}
or, in matrix form:
\fi 
\begin{eqnarray*}
A^{\{\f,\s,i\}} \coloneqq \begin{cases}
\one^\fast\, \Abase_{i,:} 
+ \sum_{k \ge 0}\big( A^\fastfast \, c^{\{\f\} \times k} \big) \, \boldsymbol{\Gamma}^k_{i,:}  & 1 \le j \le i, \\
0 & i+1 \le j \le s^\slow,
\end{cases}
\end{eqnarray*}
which establishes \eqref{eqn:MRI-butcher-fast-slow}.
\end{proof}

\subsection{Order  conditions}

In this section we derive order conditions of the MRI-GARK schemes for up to order four. First, we define a set of useful coefficients.
\begin{definition}[Some B-series coefficients]

Consider the following Butcher tree \cite{Hairer_book_I}:
\[
\mathfrak{t}_k \coloneqq [\underbrace{\tau,\dots,\tau}_{k~\textnormal{times}}] \in T,
\]
where $\tau \in T$ is the tree of order one and $[\cdot]$ is the operation of joining subtrees by a root, and the following B-series coefficients of the exact solution of the fast sub-system:
\begin{subequations}
\label{eqn:zeta-nu-mu}
\begin{equation}
\begin{split}
\zeta_k &\coloneqq 1/\gamma\bigl( [\mathfrak{t}_k] \bigr) =  1/\bigl( (k+1)(k+2) \bigr);
\\
\omega_k&\coloneqq  1/\gamma\bigl( [\tau,\mathfrak{t}_k] \bigr) = 1/\bigl( (k+1)(k+3) \bigr); 
\\
\xi_k &\coloneqq 1/\gamma\bigl( [[\mathfrak{t}_k]] \bigr) = 1/\bigl( (k+1)(k+2)(k+3) \bigr). 
\end{split}
\end{equation}
For a fast Runge-Kutta method $(A^\fastfast,b^\fast,c^\fast)$ of arbitrary accuracy it holds that:
\begin{equation}
\begin{split}
b^\fast\,^T\,A^\fastfast \, c^\fast\,^{\times k} = \zeta_k\,, \quad
b^\fast\,^T\,(c^\fast \times A^\fastfast \, c^\fast\,^{\times k}) = \omega_k, &\\
b^\fast\,^T\,A^\fastfast\,A^\fastfast \, c^\fast\,^{\times k} = \xi_k,& \quad \forall k \ge 0.
\end{split}
\end{equation}
\end{subequations}
Using \eqref{eqn:zeta-nu-mu} we define the following matrix:
\begin{equation}
\label{eqn:As-bar}
\mathfrak{A}^\slowslow \coloneqq \Abase + \sum_{k \ge 0} \zeta_k\, \boldsymbol{\Gamma}^k.
\end{equation}
\end{definition}
%
\subsubsection{Internal consistency}

\begin{theorem}[Internal consistency conditions]

The MRI-GARK method fulfills the ``internal consistency'' conditions:
%
\begin{equation}
\label{eqn:internal-consistency}
\mathbf{c}^\slowfast = \mathbf{c}^\slowslow \equiv \cbase \quad \textnormal{and} \quad
\mathbf{c}^\fastfast = \mathbf{c}^\fastslow 
\end{equation}
for any fast method iff the following conditions hold:
\begin{equation}
\label{eqn:gamma-internal-consistency}
\begin{split}
\boldsymbol{\Gamma}^0 \cdot \one^\slow &= \deltac \quad \textnormal{and}\quad
\boldsymbol{\Gamma}^k \cdot \one^\slow = 0 \quad \forall k \ge 1.
\end{split}
\end{equation}
\end{theorem}

\begin{proof}
From \eqref{eqn:MRI-butcher-css} and \eqref{eqn:MRI-butcher-csf} we see that the first equation of \eqref{eqn:internal-consistency} is automatically satisfied. 
Next, we write the second internal consistency equation \eqref{eqn:internal-consistency} in the following equivalent form by equating  \eqref{eqn:MRI-butcher-cffi} and \eqref{eqn:MRI-butcher-cfs}:
\begin{equation*}
\underbrace{c_{i}^\slow  \, \one^\fast +  \deltac_{i} \,c^\fast}_{\c^{\{\f,\f,i\}}} 
=
\underbrace{c_{i}^\slow  \, \one^\fast +  \sum_{k \ge 0} \left( \boldsymbol{\Gamma}^k\,\one^\slow \right)_i \,(A^\fastfast \, c^\fast\,^{\times k})}_{\mathbf{c}^{\{\f,\s,i\}}}.
\end{equation*}
This relation has to be satisfied for any $i=1,\dots,s^\slow$ independently of the choice of the fast discretization method.
In order to achieve this we equate separately different powers $k$ to obtain:
\begin{eqnarray*}
\bigl( \boldsymbol{\Gamma}^0\,\one^\slow \bigr)_i = \deltac_i; 
\qquad
\bigl( \boldsymbol{\Gamma}^k\,\one^\slow \bigr)_i = 0, ~~ \forall k \ge 1;
\qquad i=1,\dots,s^\slow,
\end{eqnarray*}
which establishes \eqref{eqn:gamma-internal-consistency}.
%
%
\end{proof}

\begin{remark}
Assuming that both the fast and the slow methods have order at least two, the internal consistency conditions \eqref{eqn:gamma-internal-consistency} imply that the overall scheme is second order.
\end{remark}

\subsubsection{Fast order conditions}
Since the fast component is solved exactly (or equivalently, using a base Runge-Kutta scheme of arbitrary accuracy) the fast order conditions for any order $p$ are implicitly satisfied.

\subsubsection{Slow order conditions}
The slow component of the MRI-GARK method \eqref{eqn:MRI-butcher-slow} is the Runge-Kutta scheme $\bigl( \Abase,\bbase, \cbase\bigr)$. The MRI-GARK slow order $p$ conditions are fulfilled by selecting a slow base method of order $p$.
\begin{remark}[Coefficients of the slow base method]
\ifreport
We now take a closer look at the slow method coefficients \eqref{eqn:MRI-butcher-slow}. Note that the matrix $\mathbf{E}$ in \eqref{eqn:MRI-butcher-Ass} is detailed as follows:
\begin{equation}
\label{eqn:matrix-E}
\mathbf{E} 
=\begin{bmatrix} 
0 & \cdots & \cdots & 0 \\
1 & \ddots &  & \vdots \\
\vdots & \ddots & \ddots & \vdots \\ 
1 & \dots & 1 & 0
\end{bmatrix}, \qquad
\bigl(\mathbf{I}_{s^\slow\times s^\slow} + \mathbf{E}\bigr)^{-1}  = \begin{bmatrix} 1 &  \cdots & \cdots & 0 \\
-1 & \ddots &  & \vdots \\ \vdots & \ddots & \ddots &  \\ 0 & \cdots & -1 & 1 \end{bmatrix}.
\end{equation}
The first rows of $\mathbf{E}$ and $\A^\slowslow$ are zeros. Eliminating the first row of zeros from the equation \eqref{eqn:MRI-butcher-Ass}, and appending equation \eqref{eqn:MRI-butcher-bs} as an additional row, leads to:
\[
\renewcommand{\arraystretch}{1.3}
\underbrace{
\begin{bmatrix} \mathbf{E}_{2:s^\slow,:} \\ \one^\fast\,^T \end{bmatrix}}_{\mathbf{I}_{s^\slow\times s^\slow} + \mathbf{E}} \cdot \overline{\boldsymbol{\Gamma}} 
= \begin{bmatrix} \Abase_{2:s^\slow,:} \\\bbase\,^T \end{bmatrix},
\quad \Rightarrow \quad
 \overline{\boldsymbol{\Gamma}} =\bigl(\mathbf{I}_{s^\slow\times s^\slow} + \mathbf{E}\bigr)^{-1}\, \begin{bmatrix} \Abase_{2:s^\slow,:} \\\bbase\,^T \end{bmatrix}.
\]
Equation \eqref{eqn:matrix-E} leads to the following result.
\fi
The integrated gamma coefficients \eqref{eqn:gamma-as-power-series} satisfy:
\begin{equation}
\label{eqn:gamma-bar}
\overline{\gamma}_{i,j} = 
\begin{cases} 
\abase_{i+1,j} - \abase_{i,j}, & i=1,s^\slow-1, \\[3pt]
\bbase_{j} - \abase_{s^\slow,j}, & i=s^\slow.
\end{cases}
\end{equation}
\end{remark}

\begin{remark}[Coefficients of the slow embedded method]
Consider now the embedded slow method $\big( \Abase, \widehat{b}^\slow \big)$, and the coefficients $\widehat{\boldsymbol{\Gamma}}^k$ of corresponding MRI-GARK method given by the equations \eqref{eqn:MRI-butcher-Ass} and \eqref{eqn:MRI-butcher-bs}. From the structure of the error controller discussed in \cref{rem:MRI-GARK-embedded} the coefficients of the main and the embedded methods differ only in the last row. 
\ifreport
Since $\mathbf{E}\,\widehat{\boldsymbol{\Gamma}}^k = \mathbf{E}\,\widehat{\boldsymbol{\Gamma}}^k$ for all $k$, equation \eqref{eqn:MRI-butcher-Ass} is automatically satisfied. (The last column of $\mathbf{E}$ is zero, and changing the last row of $\boldsymbol{\Gamma}$ does not modify the matrix product result.) 
\fi
Using \eqref{eqn:gamma-bar}, equation \eqref{eqn:MRI-butcher-bs} for the embedded method reads:
\begin{equation}
\label{eqn:gammahat-bar}
\widehat{\gamma}^k_{i,j} = \gamma^k_{i,j} ~~\textnormal{for}~~ i = 1,\dots,s^\slow-1,
\qquad
\widehat{\overline{\gamma}}_{s^\slow,j} = 
\widehat{b}_{j}^\slow - a_{s^\slow,j}^\slow.
\end{equation}
\end{remark}

\begin{lemma}[Some useful formulas]
%
We start with several formulas that will prove useful in the derivation of order conditions. From \eqref{eqn:MRI-butcher-Asf} and \eqref{eqn:MRI-butcher-cf} we have:
\begin{equation}
\label{eqn:MRI-butcher-Asf-cf}
\begin{split}
\A^\slowfast \c^\fast &=
\sum_{i=1}^{s^\slow}  
\left(\cbase_{i+1}-c_{i}^\slow\right)\, \mathbf{g}_{i+1} b^\fast\,^T \,\left(\cbase_{i}  \, \one^\fast +  \deltac_{i} \,c^\fast \right) 
=\frac{1}{2}\,\cbase\,^{\times 2}.
\end{split}
\end{equation}
From \eqref{eqn:MRI-butcher-Afs}:
\begin{equation}
\label{eqn:MRI-butcher-Afs-cs}
\begin{split}
\ifreport
\A^{\{\f,\s,i\}}\, \cbase 
&=  \left( \one^\fast\, \Abase_{i,:} +   (A^\fastfast \, c^\fast\,^{\times k})\,\sum_{k \ge 0} \boldsymbol{\Gamma}_{i,:}^k \right)\, \cbase \\
\fi
\A^{\{\f,\s,i\}}\, \cbase 
&=  \one^\fast\, \big(\Abase\,\cbase\big)_{i} + \sum_{k \ge 0} \big(A^\fastfast \, c^\fast\,^{\times k}\big)\, \big(\boldsymbol{\Gamma}^k\, \cbase\big)_{i},
\quad i=1,\dots,s^\slow.
\end{split}
\end{equation}
%
%
From \eqref{eqn:MRI-butcher-Aff} and \eqref{eqn:MRI-butcher-cf} we have that:
\begin{equation}
\label{eqn:MRI-butcher-Aff-cf}
\begin{split}
\ifreport
(\A^\fastfast \c^\fast)_{\lambda} 
&= \sum_{j=1}^{\lambda -1}
\deltac_j\,  \one^\fast  b^\fast\,^T\,\left( c_{j}^\slow  \, \one^\fast +  \deltac_j\,c^\fast \right) \\
& \quad + \delta\cbase_{\lambda}\, A^\fastfast\,\left( \cbase_{\lambda}  \, \one^\fast +  \deltac_\lambda\,c^\fast \right) \\
&= \sum_{j=1}^{\lambda -1}
\left( \deltac_j\, c_{j}^\slow   +  \frac{1}{2}\, \deltac_j\,^2 \right) \, \one^\fast  \\
& \quad + \delta\cbase_{\lambda}\, \cbase_{\lambda}  \, c^\fast +  \deltac_\lambda\,^2\,\left( A^\fastfast\,c^\fast \right) \\
&= \sum_{j=1}^{\lambda -1}
\frac{c_{j+1}^\slow\,^2  - c_j^\slow\,^2}{2} \, \one^\fast  + \delta\cbase_{\lambda}\, \cbase_{\lambda}  \, c^\fast +  \deltac_\lambda\,^2\,\left( A^\fastfast\,c^\fast \right) \\
\fi
(\A^\fastfast \c^\fast)_{\lambda} &=
\frac{1}{2} \cbase_{\lambda}\,^2\, \one^\fast +\cbase_{\lambda}  \, \deltac_{\lambda}\, c^\fast  + \deltac_{\lambda}\,^2\, A^\fastfast\,c^\fast. 
\end{split}
\end{equation}
From \eqref{eqn:MRI-butcher-Asf} it follows that:
\begin{equation}
\label{eqn:MRI-butcher-bs-Asf}
\begin{split}
\b^\slow\,^T \, \A^{\{\s,\f,i\}} 
&= \deltac_{i}\, \bigl(\bbase\,^T \, \mathbf{g}_{i+1} \bigr)\, b^\fast\,^T  
= \deltac_{i}\, \Bigl( \sum_{\ell=i+1}^{s^\slow}\bbase_{\ell}\Bigr)\, b^\fast\,^T.
\end{split}
\end{equation}
Finally, from \eqref{eqn:MRI-butcher-Afs} we have that:
\begin{equation}
\label{eqn:MRI-butcher-bf-Afs}
\begin{split}
\ifreport
\b^\fast\,^T \, \A^\fastslow 	 &= \sum_{i=1}^{s^\slow} \deltac_{i}\,\b^\fast\,^T \bigg( \one^\fast\, \Abase_{i,:} + \sum_{k \ge 0}  (A^\fastfast \, c^\fast\,^{\times k})\, \boldsymbol{\Gamma}_{i,:}^k \bigg) \\
&= \sum_{i=1}^{s^\slow} \deltac_{i}\, \bigg( \Abase_{i,:} + \sum_{k \ge 0}  \zeta_k\, \boldsymbol{\Gamma}_{i,:}^k \bigg) \\
\fi
\b^\fast\,^T \, \A^\fastslow 	 &= \deltac\,^T\, \mathfrak{A}^\slowslow.
\end{split}
\end{equation}
\end{lemma}

\subsubsection{Third order coupling conditions}
%
\begin{theorem}[Third order coupling condition]
An internally consistent MRI-GARK method \eqref{eqn:MIS-additive} has order three iff the slow base scheme has order at least three, and the following coupling condition holds:
\begin{equation}
\label{eqn:gamma-order-3}
\deltac\,^T\, \mathfrak{A}^\slowslow\,  \cbase = \frac{1}{6},
\end{equation}
where the coefficients $\zeta_k$ are defined in \eqref{eqn:zeta-nu-mu}, and $\mathfrak{A}^\slowslow$ in \eqref{eqn:As-bar}.
\end{theorem}

\begin{proof}
For internally consistent MR-GARK schemes there are two third order coupling conditions \cite{Sandu_2016_MR-GARK,Sandu_2015_GARK}. We proceed with checking them. If the slow base method is of at least third order, then \eqref{eqn:MRI-butcher-Asf-cf} implies that the first coupling condition is automatically satisfied:
\begin{eqnarray*}
	  \frac{1}{6} &=& \b^\slow\,^T \, \A^\slowfast  \, \c^\fast =
	   \frac{1}{2}\, \left(\bbase\,^T \,\cbase\,^{\times 2} \right) =  \frac{1}{2}\cdot\frac{1}{3}.
\end{eqnarray*}
Using  \eqref{eqn:MRI-butcher-bf-Afs} the second coupling condition reads:
\begin{eqnarray*}
\frac{1}{6} &=& \b^\fast\,^T \, \A^\fastslow  \, \cbase
= 
\deltac\,^T\, \mathfrak{A}^\slowslow\,\cbase,
\end{eqnarray*}
which establishes \eqref{eqn:gamma-order-3}.
\end{proof}
\subsubsection{Fourth order conditions}

\begin{theorem}[Fourth order coupling conditions]
An internally consistent MRI-GARK method \eqref{eqn:MIS-additive} has order four iff the slow base scheme has order at least four, and the following coupling conditions hold:
\begin{subequations}
\label{eqn:gamma-order-4}
\begin{eqnarray}
\label{eqn:gamma-order-4.A}
\qquad && \frac{1}{2}\,z^\slow \,^T\, \Abase\, \cbase 
   + \sum_{k \ge 0} \left( \deltac \times (\zeta_k\,\cbase +  \omega_k \,\deltac) \right)^T \boldsymbol{\Gamma}^k\, \cbase = \frac{1}{8},  \\
\label{eqn:gamma-order-4.C}
&& \deltac\,^T \, \mathfrak{A}^\slowslow \, \cbase\,^{\times 2} = \frac{1}{12}, \\
\label{eqn:gamma-order-4.F}
&& d^\slow\,^T \,\mathfrak{A}^\slowslow \, \cbase = \frac{1}{24},  \\
\label{eqn:gamma-order-4.H}
&& \big( \deltac\,^{\times 2} \big)^T\, \Bigl( \frac{1}{2} \Abase + \sum_{k \ge 0} \xi_k \boldsymbol{\Gamma}^k \Bigr)\,\cbase 
 + t^\slow\,^T\, \mathfrak{A}^\slowslow\,\cbase = \frac{1}{24}, \\
\label{eqn:gamma-order-4.I}
&& \deltac\,^T \, \mathfrak{A}^\slowslow\, \Abase \,\cbase = \frac{1}{24},
\end{eqnarray}
with the coefficients $\zeta_k,\omega_k,\xi_k$ defined in \eqref{eqn:zeta-nu-mu}, $\mathfrak{A}^\slowslow$ defined in \eqref{eqn:As-bar}, and:
\begin{eqnarray}
\label{eqn:gamma-order-4-z.def}
z^\slow & \coloneqq & 
\Big[ \cbase_{i+1}\,^{2}-\cbase_{i}\,^{2} \Big]_{1 \le i \le s^\slow}, \\
\label{eqn:gamma-order-4-d.def}
d^\slow & \coloneqq & 
\bigg[ \deltac_i\, \Bigl(1 - \sum_{\ell=1}^{i}\bbase_{\ell} \Bigr) \bigg]_{1 \le i \le s^\slow}, \\
\label{eqn:gamma-order-4-t.def}
t^\slow &\coloneqq& \bigg[ \sum_{j=i+1}^{s^\slow} \big( \deltac_{j}\big)^2 \bigg]_ {1 \le i \le \s^\slow}.
\end{eqnarray}
\end{subequations}
\end{theorem}

\begin{proof}

For internally consistent MR-GARK schemes there are ten coupling conditions for order four \cite{Sandu_2015_GARK}. We proceed with checking each of them.

\par{\bf Condition a.}

\begin{eqnarray*}
	  \frac{1}{8} &=& \b^\fast\,^T \, \left(\c^\fast \times \A^\fastslow \, \c^\slow \right) \\
	  \ifreport
	  &=& \b^\fast\,^T \,  \, \bigg[ \left(\cbase_{i}  \, \one^\fast +  \deltac_{i} \,c^\fast \right) \times \left(  \one^\fast\, \big(\Abase\,\cbase\big)_{i,:} + \sum_{k \ge 0} \big(A^\fastfast \, c^\fast\,^{\times k}\big)\, \big(\boldsymbol{\Gamma}^k\, \cbase\big)_{i,:}\right)\bigg]_ {1 \le i \le \s^\slow} \\
	  \fi
	  &=& \sum_{i=1}^{s^\slow} \deltac_{i} \,b^\fast\,^T \, \, \Bigl(   \one^\fast\,\cbase_{i}  \,\Abase_{i,:}\, \cbase +  \sum_{k \ge 0}  (A^\fastfast \, c^\fast\,^{\times k}) \,c_{i}^\slow  \,\boldsymbol{\Gamma}_{i,:}^k\, \cbase \\
	&&\quad + \,c^\fast\,(\deltac_{i} \,\Abase_{i,:}\, \cbase)  +  \sum_{k \ge 0}  (c^\fast \times A^\fastfast \, c^\fast\,^{\times k})\deltac_{i} \,\boldsymbol{\Gamma}_{i,:}^k\, \cbase \Bigr) \\
	\ifreport
	&=&\sum_{i=1}^{s^\slow} \deltac_{i} \, \Bigl[  \cbase_{i}  \,\Abase_{i,:}\, \cbase +  \sum_{k \ge 0} \zeta_k\,\cbase_{i}  \,\boldsymbol{\Gamma}_{i,:}^k\, \cbase \\
	&&\quad + \frac{1}{2}\,\deltac_{i} \,\Abase_{i,:}\, \cbase +  \sum_{k \ge 0}  \omega_k \, \deltac_{i} \,\boldsymbol{\Gamma}_{i,:}^k\, \cbase \Bigr] \\
	\fi
	&=&\sum_{i=1}^{s^\slow} \frac{\deltac_{i} \,\bigl( \cbase_{i+1}+\cbase_{i} \bigr) }{2}\, \Abase_{i,:}\, \cbase
	 \\
	&&\quad + \sum_{i=1}^{s^\slow} \deltac_{i} \, \sum_{k \ge 0} \bigl( \zeta_k\,\cbase_{i} + \omega_k \, \deltac_{i} \bigr) \,\boldsymbol{\Gamma}_{i,:}^k\, \cbase.
\end{eqnarray*}
Using notation \eqref{eqn:gamma-order-4-z.def} for the first term one obtains equation \eqref{eqn:gamma-order-4.A}.

\par{\bf Condition b.}
Using \eqref{eqn:MRI-butcher-Asf-cf} and the fourth order of the slow base method one checks that the second MR-GARK coupling condition is satisfied automatically:
\begin{eqnarray*}
	  \frac{1}{8} &=& \b^\slow\,^T \, \left(\cbase \times \A^\slowfast \, \c^\fast\right) = \frac{1}{2}\, \bbase\,^T \, \cbase\,^{\times 3} = \frac{1}{2}\cdot\frac{1}{4}.
\end{eqnarray*}
\par{\bf Condition c.}
Using \eqref{eqn:MRI-butcher-bf-Afs} the third coupling condition reads:
\begin{eqnarray*}
	  \frac{1}{12} &=& \b^\fast\,^T \, \A^\fastslow  \,\cbase \,^{\times 2} 
	  = \deltac\,^T\, \mathfrak{A}^\slowslow \, \cbase \,^{\times 2}, 
\end{eqnarray*}
which proves equation \eqref{eqn:gamma-order-4.C}.

\par{\bf Condition d.}
Using \eqref{eqn:MRI-butcher-slow-fast}, \eqref{eqn:MRI-butcher-cf}, and \eqref{eqn:MRI-butcher-bs-Asf} the fourth coupling condition reads:
\begin{eqnarray*}
\frac{1}{12} &=& \b^\slow\,^T \, \A^\slowfast  \,\c^\fast\,^{\times 2} \\
&=&  \sum_{i=1}^{s^\slow} \deltac_{i}\, \Bigl( \sum_{\ell=i+1}^{s^\slow}\bbase_{\ell}\Bigr) \,
b^\fast\,^T\, \bigl(\cbase_{i}  \, \one^\fast +  \deltac_{i} \,c^\fast \bigr)^{\times 2} \\
\ifreport
&=& \sum_{i=1}^{s^\slow} \deltac_{i}\, \Bigl( \sum_{\ell=i+1}^{s^\slow}\bbase_{\ell}\Bigr) \,
\Bigl(\cbase_{i}\,^{2}   + \frac{1}{3}\, \deltac_{i}\,^2 +\cbase_{i}\, \deltac_{i} \Bigr) \\
&=& \sum_{i=1}^{s^\slow} \Bigl( \sum_{\ell=i+1}^{s^\slow}\bbase_{\ell}\Bigr) \,
\Bigl( \deltac_{i}\,\cbase_{i}\,^{2}   + \frac{1}{3}\, \deltac_{i}\,^3 +\cbase_{i}\, \deltac_{i}\,^2 \Bigr) \\
&=& \sum_{i=1}^{s^\slow} \Bigl( \sum_{\ell=i+1}^{s^\slow}\bbase_{\ell}\Bigr) \,
\Bigl(  \cbase_{i+1}\,\cbase_{i}\,^{2}   -\cbase_{i}\,^{3}
+ \frac{1}{3}\, \cbase_{i+1}\,^3 - \frac{1}{3}\, \cbase_{i}\,^3 -  \cbase_{i+1}\,^2\,\cbase_{i} +  \cbase_{i+1}\, \cbase_{i}\,^2  \\
&& \qquad +\cbase_{i}\,  \cbase_{i+1}\,^2 - 2\, \cbase_{i+1}\,\cbase_{i}\,^2 + \cbase_{i}\,^3 \Bigr) \\
\fi
&=& \frac{1}{3}\, \sum_{i=1}^{s^\slow} \Bigl( \sum_{\ell=i+1}^{s^\slow}\bbase_{\ell}\Bigr) \,
\Bigl(  \cbase_{i+1}\,^3 - \cbase_{i}\,^3  \Bigr).
\end{eqnarray*}
We have that:
\begin{equation}
\label{eqn:bc3}
\begin{split}
 \sum_{i=1}^{s^\slow} \Bigl( \sum_{\ell=i+1}^{s^\slow}\bbase_{\ell}\Bigr) \,
\Bigl(  \cbase_{i+1}\,^3 - \cbase_{i}\,^3  \Bigr) 
	 &=  \sum_{\ell=2}^{s^\slow}\bbase_{\ell} \,
\Bigl(  \cbase_{\ell}\,^3 - \cbase_{1}\,^3  \Bigr) =\frac{1}{4},
 \end{split}
\end{equation}
for order four slow base methods with $\cbase_{1}=0$. Consequently, the fourth MR-GARK coupling condition is automatically satisfied.
\par{\bf Condition e.}
Using \eqref{eqn:MRI-butcher-Asf-cf} one verifies that the fifth coupling condition always holds:
\begin{eqnarray*}
	 \frac{1}{24} &=& \b^\slow\,^T \, \A^\slowslow \, \A^\slowfast \, \c^\fast = \frac{1}{2}\,\bbase\,^T \, \Abase \, \cbase\,^{\times 2} = \frac{1}{2} \cdot \frac{1}{12}.
\end{eqnarray*}
\par{\bf Condition f.}
Using \eqref{eqn:MRI-butcher-bs-Asf} and \eqref{eqn:MRI-butcher-Afs-cs}:
\begin{eqnarray*}
\frac{1}{24} &=& \b^\slow\,^T \, \A^\slowfast \, \A^\fastslow \, \cbase \\
\ifreport
&=& \sum_{i=1}^{s^\slow} \bigg(\deltac_i\, \bigg( \sum_{\ell=i+1}^{s^\slow}\bbase_{\ell} \bigg)\, b^\fast\,^T\bigg)\cdot \\
&& \cdot \bigg( \one^\fast\, \big(\Abase\,\cbase\big)_{i} + \sum_{k \ge 0} \big(A^\fastfast \, c^\fast\,^{\times k}\big)\, \big(\boldsymbol{\Gamma}^k\, \cbase\big)_{i} \bigg) \\
\fi
&=& \sum_{i=1}^{s^\slow} \bigg(\deltac_i\, \bigg( \sum_{\ell=i+1}^{s^\slow}\bbase_{\ell} \bigg)\bigg) \cdot \bigg( \big( \Abase+\sum_{k \ge 0} \zeta_k\,\boldsymbol{\Gamma}^k \big)\,\cbase \bigg)_i.
\end{eqnarray*}
Using the notation \eqref{eqn:gamma-order-4-d.def} and \eqref{eqn:As-bar} the above establishes \eqref{eqn:gamma-order-4.F}.

\par{\bf Condition g.}
Using \eqref{eqn:MRI-butcher-Aff-cf} and \eqref{eqn:MRI-butcher-bs-Asf}:
\begin{eqnarray*}
	 \frac{1}{24} &=& \b^\slow\,^T \, \A^\slowfast \, \A^\fastfast \, \c^\fast \\
	 &=&  \sum_{i=1}^{s^\slow} \left(\deltac_i\, \bigg( \sum_{\ell=i+1}^{s^\slow}\bbase_{\ell} \bigg)\, b^\fast\,^T\right)\cdot \\
	 &&\cdot \left( \frac{1}{2}\cbase_{i}\,^2\, \one^\fast +c_{i}^\slow  \, \deltac_{i}\, c^\fast  +\deltac_{i} \,^2\, A^\fast\,c^\fast  \right) \\
	 \ifreport
	 &=&  \sum_{i=1}^{s^\slow}  \bigg( \sum_{\ell=i+1}^{s^\slow}\bbase_{\ell} \bigg)\,\deltac_i\, \left( \frac{1}{2}\cbase_{i}\,^2 + \frac{1}{2}\cbase_{i}  \, \deltac_{i}  + \frac{1}{6} \deltac_{i}\,^2 \right) \\	 
	 &=&  \sum_{i=1}^{s^\slow} \bigg( \sum_{\ell=i+1}^{s^\slow}\bbase_{\ell} \bigg) \deltac_i\, \left(  \frac{1}{6}\cbase_{i} \cbase_{i+1} + \frac{1}{6}\cbase_{i}\,^2  + \frac{1}{6} \cbase_{i+1}\,^2 \right) \\
	 \fi
	 &=& \frac{1}{6}\,  \sum_{i=1}^{s^\slow}  \bigg( \sum_{\ell=i+1}^{s^\slow}\bbase_{\ell} \bigg)\,\left(\cbase_{i+1}\,^3 - \cbase_{i}\,^3\right).
\end{eqnarray*}
Using \eqref{eqn:bc3} we see that this condition is automatically fulfilled.

\par{\bf Condition h.}
Using \eqref{eqn:MRI-butcher-Afs-cs}:
\begin{eqnarray*}
\frac{1}{24} &=& \b^\fast\,^T \, \A^\fastfast \, \A^\fastslow \, \c^\slow \\
\ifreport
	 	 &=& \sum_{i=1}^{s^\slow} \left(\deltac_{i} \,^2\,b^\fast\,^T\,A^\fastfast
+\sum_{j=i+1}^{s^\slow}  \deltac_{j}\,^2\,b^\fast\,^T  \right) \cdot \\
&&\quad \cdot \left(\one^\fast\, \big(\Abase\,\cbase\big)_{i} + \sum_{k \ge 0} \big(A^\fastfast \, c^\fast\,^{\times k}\big)\, \big(\boldsymbol{\Gamma}^k\, \cbase\big)_{i} \right) \\
\fi
	 	 &=&  \sum_{i=1}^{s^\slow}  \deltac_{i} \,^2\,\left( \frac{1}{2}\, \big(\Abase\,\cbase\big)_{i} 
+ \sum_{k \ge 0} \xi_k\, \big(\boldsymbol{\Gamma}^k\, \cbase\big)_{i} \right) \\
&& +  \sum_{i=1}^{s^\slow} \bigg( \sum_{j=i+1}^{s^\slow}  \deltac_{j}\,^2 \bigg)\, \bigg( \big( \Abase\,\cbase\big)_{i} 
+\sum_{k \ge 0} \zeta_k\, \big(\boldsymbol{\Gamma}^k\, \cbase\big)_{i} \bigg).
\end{eqnarray*}
With notation \eqref{eqn:gamma-order-4-t.def} this establishes \eqref{eqn:gamma-order-4.H}.

\par{\bf Condition i.}
Using \eqref{eqn:MRI-butcher-Afs}:
\begin{eqnarray*}
\frac{1}{24} &=& \b^\fast\,^T \, \A^\fastslow \, \A^\slowslow \, \c^\slow \\
\ifreport
&=& \sum_{i=1}^{s^\slow} \deltac_{i}\,\b^\fast\,^T \bigg( \one^\fast\, \Abase_{i,:} + \sum_{k \ge 0}  (A^\fastfast \, c^\fast\,^{\times k})\, \boldsymbol{\Gamma}_{i,:}^k \bigg) \, \Abase \, \cbase \\
\fi
 &=&  \sum_{i=1}^{s^\slow} \deltac_{i}\,\bigg(\Abase_{i,:} + \sum_{k \ge 0}  \zeta_k\, \boldsymbol{\Gamma}_{i,:}^k \bigg) \, \Abase \, \cbase,
 \end{eqnarray*}
which proves \eqref{eqn:gamma-order-4.I}.

\par{\bf Condition j.}
Using \eqref{eqn:MRI-butcher-Afs} and \eqref{eqn:MRI-butcher-bf-Afs} the tenth coupling condition reads:
\begin{eqnarray*}
	 \frac{1}{24} &=& \b^\fast\,^T \, \A^\fastslow \, \A^\slowfast \, \c^\fast  
	 = \frac{1}{2}\,\deltac\,^T\, \mathfrak{A}^\slowslow \,\cbase\,^{\times 2},
\end{eqnarray*}
which is the same as condition \eqref{eqn:gamma-order-4.C}.

\end{proof}


\section{Linear stability analysis}
\label{sec:stability}
%
\subsection{Scalar stability analysis}
%
For additively partitioned systems \eqref{eqn:multirate-additive-ode} we consider the following scalar model problem:
\begin{equation}
\label{eqn:stability-test-scalar}
y' =  \lambdaf \, y + \lambdas \, y, \qquad \lambdaf, \lambdas \in \mathbb{C}^-.
\end{equation}
Let $\zf \coloneqq H\,\lambdaf$ and $\zs \coloneqq H\,\lambdas$.
The stage computation \eqref{eqn:MIS-additive-internal-ode} applied to \eqref{eqn:stability-test-scalar} solves exactly a linear ODE
\ifreport
\[
v(0) = \Ys_i, \quad
v' = \deltac_i \, \lambdaf \, v  +  \lambdas\, \sum_{j=1}^{i+1} \sum_{k\ge 0} \gamma^k_{i,j} \,
\left({\scriptstyle \frac{\theta}{H}}\right)^k \,\Ys_j,\quad \theta \in [0,H].
\]
\fi
with the following analytical solution:
\begin{equation*}
\begin{split}
\ifreport
\Ys_{i+1} &=  e^{\deltac_i \, \lambdaf\,H} \,\Ys_i  + \lambdas\, \,\sum_{j=1}^{i+1} \sum_{k\ge 0} \gamma^k_{i,j} \,\left( \int_0^H e^{\deltac_i \, \lambdaf \, (H-\theta)} \left({\scriptstyle \frac{\theta}{H}}\right)^k d\theta \right)\,\Ys_j \\
 &=  e^{\deltac_i \, \zf} \,\Ys_i + \zs\,\sum_{j=1}^{i+1} \sum_{k\ge 0} \gamma^k_{i,j} \,\left( \int_0^1 e^{\deltac_i \, \zf \, (1-t)} t^k dt \right)\,\Ys_j \\
&= \varphi_{0}\bigl( \deltac_i \, \zf \bigr) \,\Ys_i + \zs\,\sum_{j=1}^{i+1} \sum_{k\ge 0} \gamma^k_{i,j} \,\varphi_{k+1}\bigl( \deltac_i \, \zf \bigr)\,\Ys_j \\
\fi
 \Ys_{i+1} &=  \varphi_{0}\bigl( \deltac_i \, \zf \bigr)\,\Ys_i  + \zs\,\sum_{j=1}^{i+1} \mu_{i,j}\big( \zf \big)\,\Ys_j,
\end{split}
\end{equation*}
where the $\mu$ coefficients are functions of the fast variable:
\[
 \mu_{i,j}\big( \zf \big) \coloneqq \sum_{k\ge 0} \gamma^k_{i,j} \,\varphi_{k+1}\bigl( \deltac_i \, \zf \bigr),
\]
and are defined using the following family of analytical functions:
\begin{equation}
\label{eqn:phi-functions}
\begin{split}
& \varphi_0(z)\coloneqq e^z, \quad
\varphi_{k}(z) \coloneqq \int_0^1 e^{z \, (1-t)}\, t^{k-1}\, dt, \quad 
\varphi_{k+1}(z)=\frac{k\,\varphi_{k}(z)-1}{z}. 
%
\end{split}
\end{equation}

%
%

%
%
%
The MRI-GARK \eqref{eqn:MIS-component} applied to \eqref{eqn:stability-test-scalar} reveals a stability function that depends on both slow and fast variables: 
\[
y_{n+1} = R\bigl(\zf,\zs\bigr)\,y_{n}. 
\]
\begin{definition}[Scalar stability]
The scalar slow stability region is defined as:
\begin{equation}
\label{eqn:scalar-slow-stability-region}
\mathcal{S}^{\textsc{1d}}_{\rho,\alpha} = \big\{ \zs \in \mathbb{C} \mid  | R(\zf,\zs) | \le 1,~ \forall~ \zf \in \mathbb{C}^- ~ \colon~ 
|\zf| \le \rho,~
|\arg(\zf) - \pi| \le \alpha
\big\}.
\end{equation}
Thus $\zs \in \mathcal{S}^{\textsc{1d}}_\alpha$ ensures that the solution is stable for any system \eqref{eqn:stability-test-scalar} with $\lambdaf$ in an $\alpha$-wedge in the left complex semi-plane, and of absolute value bounded by $\rho$.
\end{definition}

\begin{example}[Stability functions of second order methods]
The stability function of the explicit midpoint MRI-GARK scheme \cref{eqn:explicit-midpoint} is a quadratic polynomial in $\zs$, with coefficients analytical functions of $\zf$:
\[
R_\textsc{emidp}(\zf,\zs) = \varphi_{0}\big( \zf \big)
+  \left( {\scriptstyle \frac{3}{2}} \,\varphi_{0}\big( {\scriptstyle \frac{1}{2}} \zf \big)-{\scriptstyle \frac{1}{2}}  \right)\, \varphi_{1}\big({\scriptstyle \frac{1}{2}} \zf \big) \, \zs   + {\scriptstyle \frac{1}{2}} \, \varphi_{1}^2\big({\scriptstyle \frac{1}{2}} \zf \big) \, \zs\,^2.
\,
\]
Similarly, the stability function of the implicit trapezoidal MRI-GARK scheme of \cref{eqn:implicit-trapezoidal} is a rational function in $\zs$, with coefficients analytical functions of $\zf$:
\[
R_\textsc{itrap}(\zf,\zs) = \frac{\varphi_{0}(\zf) + \left(\varphi_{1}(\zf)-{\scriptstyle \frac{1}{2}}\right) \,\zs}{1-{\scriptstyle \frac{1}{2}}\zs}.
\]
\end{example}

\subsection{Matrix stability analysis}
Following  Kv{\ae}rn{\o} \cite{Kvaerno_2000_stability-MRK}, for component partitioned systems \eqref{eqn:multirate-component-ode} we consider the following model problem:
\begin{equation}
\label{eqn:ode2d-linear}
\renewcommand{\arraystretch}{1.1}
\begin{bmatrix}
{y}^{\{\f\}} \\
{y}^{\{\s\}}
\end{bmatrix}'
=
\begin{bmatrix}
\lambdaf & \etas  \\
\etaf & \lambdas
\end{bmatrix}
\,
\begin{bmatrix} \yf \\ \ys  \end{bmatrix}
=
\underbrace{
\begin{bmatrix}
 \lambdaf & \frac{1-\xi}{\alpha } \left(\lambdaf-\lambdas\right) \\
 -\alpha \, \xi \, \left(\lambdaf-\lambdas\right) & \lambdas
\end{bmatrix}
}_{\boldsymbol{\Omega}}
\,
\begin{bmatrix} \yf \\ \ys  \end{bmatrix},
%
%
\end{equation}
%
%
where
\[
\alpha \coloneqq  \frac{\lambdaf-\lambdas+\delta}{2 \etas}, \quad
\xi  \coloneqq \frac{\lambdaf - \lambdas - \delta }{2 \left(\lambdaf-\lambdas\right)}, \quad
\delta = \sqrt{4 \etaf \etas+\left(\lambdaf-\lambdas\right)^2}.
\]   
\ifreport
The eigenvalue-eigenvector pairs of $\boldsymbol{\Omega}$ are:
\begin{equation}
\renewcommand{\arraystretch}{1.3}
\label{eqn:eigenpairs-omega}
\left\{ \xi \, \lambdaf + (1 -\xi) \, \lambdas, ~
\begin{bmatrix}
 -\frac{1}{\alpha }  \\
 \phantom{-}{\scriptstyle 1} 
\end{bmatrix} \right\},
\quad
\left\{ (1-\xi) \, \lambdaf + \xi \, \lambdas ,~
\begin{bmatrix}
 -\frac{1-\xi}{\alpha } \\
 {\scriptstyle \xi} 
\end{bmatrix} \right\}.
\end{equation}
\else
The eigenvalues of $\boldsymbol{\Omega}$ are linear combinations of the slow and fast diagonal terms, $ \xi \, \lambdaf + (1 -\xi) \, \lambdas$ and $(1-\xi) \, \lambdaf + \xi \, \lambdas$.
\fi
For $|\xi| \ll 1$ the fast sub-system has a weak influence on the slow one; the first eigenvalue is slow and the second one is fast. For $|1-\xi| \ll 1$ the slow sub-system has a weak influence on the fast one, and the first eigenvalue is fast. 

\ifreport
\begin{remark}[Scaling the system]
Let $\mu = \lambdaf/\lambdas$. We have that:
\begin{equation*}
\begin{split}
\Omega & = \lambdas\, \begin{bmatrix}
 \mu & \frac{1-\xi}{\alpha } \left(\mu-1\right) \\
 -\alpha \, \xi \, \left(\mu-1\right) & 1
\end{bmatrix} \stackrel{\mu \to 0}{\longrightarrow} 
\lambdas\, \begin{bmatrix}
 0 & \frac{\xi-1}{\alpha } \\
 \alpha \, \xi  & 1
\end{bmatrix}
\\
\Omega & =
\lambdaf\,\begin{bmatrix}
 1 & \frac{1-\xi}{\alpha } \left(1-\mu^{-1}\right) \\
 -\alpha \, \xi \, \left(1-\mu^{-1}\right) & \mu^{-1}
\end{bmatrix}  \stackrel{\mu \to \infty}{\longrightarrow} 
\lambdaf\,\begin{bmatrix}
 1 & \frac{1-\xi}{\alpha } \\
 -\alpha \, \xi & 0
\end{bmatrix}.
\end{split}
\end{equation*}
The eigenvalues of the limiting matrices are $\{\xi,1-\xi\}$ times $\lambdas$ or $\lambdaf$, respectively.
\end{remark}
\fi

\begin{remark}[Scale considerations]
In order to have the slow and fast contributions to $\ys\,'$ of similar magnitude, and the contributions to $\yf\,'$ of similar magnitude as well, one needs a coupling coefficient inversely proportional to the scale ratio of the system, $\xi \sim (|\lambdaf/\lambdas| + 1)^{-1}$.

\ifreport
Specifically, in order to have a slow evolution of $\ys$ we require that the impacts of the fast and slow terms on $\dot{y}^{\{\s\}}$ are of similar magnitude:
\[
| \alpha | \cdot |\xi | \cdot \big| \lambdaf-\lambdas\big| \sim |\lambdas| \quad \Rightarrow \quad
| \alpha |\cdot |\xi | \cdot \left| \mu-1\right| \sim 1,
\]
where $\mu = \lambdaf/\lambdas$. Similarly, we also require that the impacts of the fast and slow terms on $\dot{y}^{\{\f\}}$ are of similar magnitude:
\[
 | (1-\xi)/\alpha | \cdot \big| \lambdaf-\lambdas \big | \sim | \lambdaf |  \quad \Rightarrow \quad
 | (1-\xi)/\alpha | \cdot \left| \mu- 1 \right | \sim | \mu |. 
\]
Eliminating $\alpha$ from the two relations we obtain:
\[
| (1-\xi)/\xi | \sim |\mu| \quad \Rightarrow \quad \xi \sim \frac{1}{|\mu| + 1}.
\]
\fi
\end{remark}

The general system \eqref{eqn:ode2d-linear} is complex, and in order to ease the analysis process some simplifications are needed. First, the $\alpha$ factor represents a scaling of the fast variable, as can be seen from rewriting the system \eqref{eqn:ode2d-linear} in terms of the variables $[ \alpha^{-1}\,\yf, \ys ]$.
\ifreport
Specifically, the system \eqref{eqn:ode2d-linear} can be written as:
\begin{equation}
\label{eqn:ode2d-linear-scaled}
\renewcommand{\arraystretch}{1.3}
\begin{bmatrix}
\alpha^{-1}\,\dot{y}^{\{\f\}} \\
\dot{y}^{\{\s\}}
\end{bmatrix}
=
\underbrace{
\begin{bmatrix}
 \lambdaf & (1-\xi) \left(\lambdaf-\lambdas\right) \\
 - \xi  \left(\lambdaf-\lambdas\right) & \lambdas
\end{bmatrix}
}_{\overline{\boldsymbol{\Omega}}}
\,
\begin{bmatrix} \alpha^{-1}\,\yf \\ \ys  \end{bmatrix}.
\end{equation}
\fi
To simplify the analysis we take $\alpha=1$, i.e., we assume that the multirate method is applied to the system \eqref{eqn:ode2d-linear} after rescaling the fast variables. It is possible, however, that the stability of some multirate schemes is affected by the scaling, in the sense that the same scheme, applied to systems  with different scalings, has different stability properties. 

In order to further simplify analysis we restrict ourselves to considering the case where $\xi \in \Re$. We see from the eigenvalue analysis that, in order to have a stable coupled system $\boldsymbol{\Omega}$ for any $\lambdaf, \lambdas \in \mathbb{C}^-$, one needs to restrict $\xi \in [0,1]$. 

%

Let $\zf=H\lambdaf$, $\zs= H \lambdas$, $\ws=H\etas$, and $\wf=H\etaf$.
The scheme \eqref{eqn:MIS-component} applied to the test problem \eqref{eqn:ode2d-linear} computes the stages \eqref{eqn:MIS-component-stages} as follows:
\ifreport
\begin{equation*}
 \begin{cases}
v^{\{\f\}}_i(0) = \Yf_{i}, \\
v^{\{\f\}}_i\,' = \deltac_{i}  \,  \Big( \lambdaf\,v^{\{\f\}}_i + \etas\, \Ys_{i} \\
\qquad +  \etas\,\sum_{j=1}^{i}  \widetilde{\gamma}_{i,j}\left({\scriptstyle \frac{\theta}{H}}\right) \, \bigl( \wf\,\Yf_j + \zs\,\Ys_j \bigr) \Big) 
\quad  \textnormal{for  } \theta \in  [0,  H],  \\
\Yf_{i+1} = v^{\{\f\}}_i(H), \\  
\Ys_{i+1} = \Ys_{i}+\sum_{j=1}^{i + 1}  \overline{\gamma}_{i,j}\, \bigl(\wf\, \Yf_j + \zs \,\Ys_j \bigr), \qquad i=1,\dots,s^{\{\s\}}.
\end{cases}
\end{equation*}
The analytical solution of the linear stage ODE is as follows:
\begin{equation*}
\begin{split}
%
%
%
\Yf_{i+1} &= \varphi_0\bigl(\Delta  c^{\{\s\}}_{i}\,  \zf\bigr)\, \Yf_{i}  + \Delta  c^{\{\s\}}_{i}\,\varphi_1\bigl(\Delta  c^{\{\s\}}_{i}\,  \zf\bigr)\,  \ws\, \Ys_{i} \\
& \qquad +\ws\, \Delta  c^{\{\s\}}_{i}\,\sum_{j=1}^{i}  \nu_{i,j}(\zf) \, \bigl( \wf\,\Yf_j + \zs\,\Ys_j \bigr), \\
\Ys_{i+1} &=\Ys_{i}+\sum_{j=1}^{i+1}  \overline{\gamma}_{i,j} \, \bigl( \wf\,\Yf_j + \zs\,\Ys_j \bigr).
\end{split}
\end{equation*}
The $\nu$ coefficients are defined below.
\begin{remark}[Equal abscissae]
For stages where $\Delta  c^{\{\s\}}_{i}=0$ we have:
\begin{equation*}
\Yf_{i+1} = \Yf_{i}, \qquad
\Ys_{i+1} =\Ys_{i}+\sum_{j=1}^{i+1}  \overline{\gamma}_{i,j} \, \bigl( \wf\,\Yf_j + \zs\,\Ys_j \bigr),
\end{equation*}
and the corresponding slow stage may be implicit.
\end{remark}
The general stage evolution equations read:
\fi
\begin{equation*}
\begin{split}
\Yf_{i+1} &=  \varphi_0\bigl(\Delta  c^{\{\s\}}_{i}\,  \zf\bigr)\, \Yf_{i}  +\ws\,\wf\, \Delta  c^{\{\s\}}_{i}\,\sum_{j=1}^{i}  \nu_{i,j}(\zf) \, \Yf_j \\
& \quad + \Delta  c^{\{\s\}}_{i}\,\varphi_1\bigl(\Delta  c^{\{\s\}}_{i}\,  \zf\bigr)\,  \ws\, \Ys_{i}
 +\ws\, \zs\, \Delta  c^{\{\s\}}_{i}\,\sum_{j=1}^{i}  \nu_{i,j}(\zf) \,\Ys_j,  \\
\Ys_{i+1} 
&=\left( 1- \overline{\gamma}_{i,i+1}\,\zs \right)^{-1}\, \left(\Ys_{i} +  \wf\, \sum_{j=1}^{i+1}  \overline{\gamma}_{i,j} \,\Yf_j
+ \zs\,\sum_{j=1}^{i}  \overline{\gamma}_{i,j} \, \Ys_j \right).
\end{split}
\end{equation*}
The $\nu$ coefficients are defined as:
\begin{equation*}
\begin{split}
\ifreport
\nu_{i,j}(\zf) &\coloneqq \frac{1}{H}\,\int_0^H  e^{\Delta  c^{\{\s\}}_{i} \, \zf\, (1-\frac{\theta}{H})}\,\widetilde{\gamma}_{i,j}({\scriptstyle \frac{\theta}{H}})\, d\theta  \\
&=\int_0^1  e^{\Delta  c^{\{\s\}}_{i} \, \zf\, (1-\tau)}\,\widetilde{\gamma}_{i,j}(\tau) d\tau \\
&=  \sum_{k \ge 0} \gamma_{i,j}^k\, \int_0^1  e^{\Delta  c^{\{\s\}}_{i} \, \zf\, (1-\tau)}\,\frac{\tau^{k+1}}{k+1}\, d\tau \\
\fi
\nu_{i,j}(\zf) &=  \sum_{k \ge 0} \frac{\gamma_{i,j}^k}{k+1}\,\varphi_{k+2}\big( \Delta  c^{\{\s\}}_{i} \, \zf \big).
\end{split}
\end{equation*}
Consequently, the MRI-GARK solution advances over one step $H$ via the recurrence:
\[
\renewcommand{\arraystretch}{1.5}
\begin{bmatrix} \yf_{n+1} \\ \ys_{n+1} \end{bmatrix}
 = \mathbf{M}(\zf,\zs,\ws,\wf)\,\begin{bmatrix} \yf_{n} \\ \ys_{n} \end{bmatrix}. 
\]
\begin{definition}[Matrix stability]
The slow stability region ensures that the spectral radius of the error propagation matrix $\mathbf{M}$ has to be smaller than or equal to one:
\begin{equation}
\label{eqn:matrix-slow-stability-region}
\begin{split}
\mathcal{S}^{\textsc{2d}}_{\rho,\alpha} = \big\{ \zs \in \mathbb{C} \mid & \max | \operatorname{eig}\,\mathbf{M}(\zf,\zs) | \le 1, \\
&~ \forall~ \zf \in \mathbb{C}^- ~ \colon~  |\zf| \le \rho,~
|\arg(\zf) - \pi| \le \alpha
\big\}.
\end{split}
\end{equation}
\end{definition}

\begin{example}[Stability matrices]
The stability matrix of the explicit midpoint MRI-GARK scheme of \eqref{eqn:explicit-midpoint} is:
\begin{equation*}
\renewcommand{\arraystretch}{1.5}
\begin{split}
\mathbf{M}^\textsc{emidp} &= \left[
\begin{array}{cc}
\scriptstyle 
\varphi _0^2+\frac{3w}{4} \varphi _2 \varphi _0+\frac{w}{8}  \left(2 \varphi _1+\varphi _2 (2 \zf+w \varphi _2-2)\right) & 
\scriptstyle \mathbf{M}^\textsc{emidp}_{1,2}  \\
\scriptstyle -\frac{\alpha \, \xi \, ( \zf-\zs ) }{4}  (2 \zs+4 \varphi _0+w \varphi _2) &
\scriptstyle \frac{1}{2}\,\zs\,^2+\zs+\frac{w}{4}  (2 \varphi _1+\zf \varphi _2)+1
\end{array} 
\right], \\
\scriptstyle \mathbf{M}^\textsc{emidp}_{1,2}  & = \scriptstyle
\frac{1-\xi}{8\,\alpha}\, \bigl(\zf-\zs\bigr) \left(2 \varphi _1 (\zs+2 \varphi _0+w \varphi _2+2)+\zf \varphi _2 \bigl(2 (\zs+\varphi _0+1)+w \varphi _2\bigr)\right), 
\end{split}
\end{equation*}
%
%
where $w \coloneqq  -\xi \, (1-\xi)\,(\zf-\zs)^2$ and all the $\varphi$ functions are evaluated at $\zf/2$. 
\ifreport
We note that coefficients contain polynomial terms in the fast variable $\zf$ as well as in the coupling variable $\xi$. We expect that these terms will impact stability when the fast system is stiff, and when the two subsystems are tightly coupled.
\fi
The stability matrix of the implicit trapezoidal MRI-GARK scheme \eqref{eqn:implicit-trapezoidal} is:
\[
\renewcommand{\arraystretch}{1.3}
\mathbf{M}^\textsc{itrap} = 
\left[
\begin{array}{cc}
\scriptstyle \varphi_0 + w\, \varphi_2 &\scriptstyle   \frac{1-\xi}{\alpha } \left(\zf-\zs\right)\, \left(\varphi_1+\zf\, \varphi_2\right) \\
 -\frac{\alpha \, \xi \, \left(\zf-\zs\right)}{2-\zs} & \frac{2+\zs}{2-\zs}
\end{array}
\right],
\]
where all the $\varphi$ functions are evaluated at $\zf$. We note that the spectra of $\mathbf{M}^\textsc{emidp}$ and $\mathbf{M}^\textsc{itrap}$ are independent of the scaling factor $\alpha$.
\end{example}

%


\section{Implementation aspects}
\label{sec:implementation}

\subsection{Numerical integration of the fast subsystem}

In a practical application of MRI-GARK the internal stage ODEs \eqref{eqn:MIS-additive-internal-ode} are solved numerically. The MRI-GARK formulation allows for extreme flexibility in the choice of the fast numerical method and the sequence of fast sub-steps. 

Assume the MRI-GARK method has order $p$, and that each stage ODE \eqref{eqn:MIS-additive-internal-ode} is solved with a time discretization of order $q$. If each stage uses at most $M$ substeps (with $M$ a moderate number) then the choice $q=p-1$ preserves the overall order of the scheme. These order considerations, however, are somewhat beside the point since multirating is designed to help with the large values of the fast error constants (appearing implicitly in the order definition). The fast subsystems can be integrated with any method and any sequence of steps that leads to sufficiently accurate results.

\subsection{MRI-GARK and exponential integration}
\label{sec:exponential}
%
In some special cases the integration of the internal stage ODEs \eqref{eqn:MIS-additive-internal-ode} can be carried out analytically.
\ifreport
\begin{lemma}[Simple integration formulas]
\label{lem:linear-integration}
The linear scalar ODE:
\[
\mathbf{v}'(t) = \alpha\,\mathbf{L}\, \mathbf{v}(t) + \sum_{k \ge 0} \gamma^k\,\left(\frac{
t}{H}\right)^k\,\mathbf{u}_k, \quad \mathbf{v}(0) = \mathbf{v}_0,
\]
with $\mathbf{v}(t),\mathbf{u}_k \in \Re^n$, has the following analytical solution:
\[
v(H) = \varphi_0(\alpha\,H\mathbf{L})\, \mathbf{v}_0 + H\ \sum_{k \ge 0} \gamma^k\,\varphi_{k+1}(\alpha\,H\mathbf{L})\,\mathbf{u}_k.
\]
\end{lemma}
\fi
Consider the following semi-linear additive system partitioning \eqref{eqn:multirate-additive-ode}:
\begin{equation} 
	\label{eqn:multirate-exp-additive-ode}
	 y'= f(y) = \underbrace{\mathbf{L}\,y}_{\funf (y)} +  \underbrace{f(y)-\mathbf{L}\,y}_{\funs (y) \equiv g(y)}, \qquad y(t_0)=y_0.
\end{equation}
An MRI-GARK scheme \eqref{eqn:MIS-additive} with $\gamma_{i,i+1}(\tau)\equiv 0$ applied to \eqref{eqn:multirate-exp-additive-ode} advances the solution as follows:
\ifreport
\begin{subequations}
\label{eqn:MIS-exp-additive}
\begin{eqnarray}
\label{eqn:MIS-exp-additive-first-stage}
\qquad&& \Ys_{1} = y_{n}, \\
\label{eqn:MIS-exp-additive-internal-ode}
&& \begin{cases}
v(0) = \Ys_{i}, \\
v' = \Delta c^{\{\s\}}_i  \, \mathbf{L}\, v + \sum_{k\ge 0} \sum_{j=1}^{i} \gamma_{i,j}^k\,\left({\scriptstyle \frac{\theta}{H}}\right)^k \, g\bigl(\,Y^{\{\s\}}_j \, \bigr),\quad  \textnormal{for  } \theta \in  [0,  H],  \\
\Ys_{i+1} = v(H), 
\end{cases}
\\
\label{eqn:MIS-exp-additive-solution}
&& y_{n+1} = \Ys_{s^{\{\s\}}+1}\,.
\end{eqnarray}
\end{subequations}
Using Lemma \ref{lem:linear-integration} the method \eqref{eqn:MIS-exp-additive} is represented as follows:
\fi
\begin{equation}
\label{eqn:MIS-exp}
\begin{split}
\Ys_{1} &= y_{n}, \\
\Ys_{i+1} &= \varphi_0(\Delta c^{\{\s\}}_i\,H\mathbf{L})\, \Ys_{i}  + H\,\sum_{j=1}^{i} \, \theta_{i,j}(H\mathbf{L})\, g\bigl(\,Y^{\{\s\}}_j \, \bigr), \\
&=  \Ys_{i} 
+ \Delta \mathfrak{c}^{\{\s\}}_{i}\,H\,{\varphi_{1}\big( \Delta \mathfrak{c}^{\{\s\}}_{i}\,H\mathbf{L}\big)}\,f(\Ys_{i}) 
+ H\,\sum_{j=1}^{i-1} {\theta_{i,j}\big(H \mathbf{L}\big)}  \,g\bigl( \Ys_{j} \bigr) \\
&\qquad + H\, \overline{\theta}_{i,i} \big(H \mathbf{L}\big) \, g(\Ys_{i}),\qquad i=1, \dots, s^{\{\s\}},  \\
\\
y_{n+1} &= \Ys_{s^{\{\s\}}+1}\,,
\end{split}
\end{equation}
where
\begin{eqnarray*}
\theta_{i,j}\big(H \mathbf{L}\big) &:=& \sum_{k \ge 0}\gamma_{i,j}^k\,\varphi_{k+1}\big( \Delta \mathfrak{c}^{\{\s\}}_{i}\,H\mathbf{L}\big), \\
\overline{\theta}_{i,i}\big(H \mathbf{L}\big) &:= & \theta_{i,i}\big(H \mathbf{L}\big) 
-  \Delta \mathfrak{c}^{\{\s\}}_{i} \,  \varphi_{1}\big( \Delta \mathfrak{c}^{\{\s\}}_{i}\,H\mathbf{L}\big).
\end{eqnarray*}
Then MRI-GARK scheme applied to a semi-linear system \eqref{eqn:multirate-exp-additive-ode} gives rise to \eqref{eqn:MIS-exp}, which is an exponential Runge-Kutta method \cite{Hochbruck_2010_exp}. Therefore the MRI-GARK framework can be regarded as a generalization of exponential integrators, where one solves exactly one component component of the system \cite{Knoth_private} that can be not only {\it linear}, but also  {\it nonlinear}. 

\begin{example}[Reaction-diffusion system]
Consider a reaction-diffusion system \eqref{eqn:multirate-exp-additive-ode} with linear diffusion $\mathbf{L}y$ and nonlinear reaction term $g(y)$. An exponential method applied with the partitioning \eqref{eqn:multirate-exp-additive-ode} treats diffusion as the fast system. However, in some applications reactions are faster than diffusion. In this case, in the MRI-GARK approach, one can choose to treat the linear diffusion as the slow term, and the nonlinear reaction as the fast term to be integrated exactly. (The exponential method will need to change the partition by linearizing the fast chemistry.)
\end{example}


\section{Practical explicit methods}
\label{sec:explicit-methods}

In the presentation of the methods we denote the Butcher tableau of the base method by $\mathsf{A}^{\{\s\}}$\eqref{eqn:slow-base-scheme}, and the tableaux  of the coupling coefficients by:
\[
\renewcommand{\arraystretch}{\astretch}
\boldsymbol{\Upgamma}^k \coloneqq \left[ \begin{array}{c} \boldsymbol{\Gamma}^k \\ \hline \widehat{\gamma}^k\,^T \end{array} \right],
\quad k \ge 0.
\]
The abscissae of the main and the embedded solutions are written explicitly since they are important in the workings of infinitesimal schemes. The coefficients $\boldsymbol{\Upgamma}^k$ that are not explicitly listed below are equal to zero. For completeness we explicitly provide the base scheme $\mathsf{A}^{\{\s\}}$ for each method, even if the base scheme is completely determined by the $\boldsymbol{\Upgamma}^k$ coefficients.

\subsection{Second order explicit methods}
%
The MRI-GARK-ERK22 family is:
\begin{tiny}
\[
\renewcommand{\arraystretch}{2.5}
\mathsf{A}^{\{\s\}} = \left[
\begin{array}{c | cc}
\scriptstyle   0 &\scriptstyle   0 &\scriptstyle   0 \\
\scriptstyle  c_2 &\scriptstyle   c_2 &\scriptstyle   0 \\
 \hline
\scriptstyle   1 & \frac{2 c_2 - 1}{2 c_2} & \frac{1}{2 c_2} \\
 \hline
\scriptstyle  1 &\scriptstyle   1 &\scriptstyle   0
\end{array}
\right],
\qquad
\boldsymbol{\Upgamma}^0 = \left[
\begin{array}{cc}
\scriptstyle  c_2 &\scriptstyle   0 \\
{\scriptstyle -}\frac{2 c_2^2 - 2 c_2 + 1}{2 c_2} & \frac{1}{2 c_2} \\
 \hline
\scriptstyle   1-c_2 &\scriptstyle   0 
\end{array}
\right].
\]
\end{tiny}
For $c_2=1/2$ this is the explicit midpoint rule MRI-GARK-ERK22a \eqref{eqn:explicit-midpoint}, and for $c_2=1$ the explicit trapezoidal method MRI-GARK-ERK22b.

The scalar slow stability region $\mathcal{S}^{\textsc{1d}}_{\infty,\alpha}$  \eqref{eqn:scalar-slow-stability-region} is shown in Figure  \ref{fig:MRI-explicit-stability_1D},  and the matrix stability region $\mathcal{S}^{\textsc{2d}}_{\rho,\alpha}$  \eqref{eqn:matrix-slow-stability-region}  in Figure  \ref{fig:MRI-explicit-stability_2D}.

\subsection{An order three explicit method with three stages and equidistant abscissae}
The MRI-GARK-ERK33 family of methods has the following base slow scheme and coupling coefficients:
\begin{tiny}
\begin{equation*}
\renewcommand{\arraystretch}{2}
\mathsf{A}^{\{\s\}} = 
\left[ \begin{array}{c | ccc}
\scriptstyle 0 &\scriptstyle 0 &\scriptstyle 0 &\scriptstyle 0 \\
\frac{1}{3} & \frac{1}{3} &\scriptstyle 0 &\scriptstyle 0 \\
\frac{2}{3} &\scriptstyle 0 & \frac{2}{3} &\scriptstyle 0 \\
\hline
\scriptstyle 1 &  \frac{1}{4} &\scriptstyle 0 & \frac{3}{4} \\
\hline
\scriptstyle 1 &  \frac{1}{12} & \frac{1}{3} & \frac{7}{12}
\end{array} \right],
\quad
\boldsymbol{\Upgamma}^0  = 
\left[ \begin{array}{ccc}
 \frac{1}{3} & \scriptstyle 0 & \scriptstyle 0 \\
 \frac{-6 \delta-7}{12}  & \frac{6 \delta+11}{12}  &\scriptstyle  0 \\
 \scriptstyle 0 & \frac{6 \delta-5}{12} & \frac{3-2 \delta}{4} \\
 \hline
 \frac{1}{12} & {\scriptstyle -}\frac{1}{3} & \frac{7}{12}
\end{array}\right],
\qquad
\boldsymbol{\Upgamma}^1 = 
\left[ \begin{array}{ccc}
\scriptstyle 0 &\scriptstyle 0 &\scriptstyle 0 \\
\frac{2\delta+1}{2} & {\scriptstyle -}\frac{2\delta+1}{2} &\scriptstyle  0 \\
 \frac{1}{2} & {\scriptstyle -}\frac{2\delta+1}{2} &\scriptstyle \delta \\
 \hline
 \scriptstyle 0 &\scriptstyle 0 &\scriptstyle 0
\end{array}
\right].
\end{equation*}
\end{tiny}
The particular scheme MRI-GARK-ERK33a tested here corresponds to the choice $\delta=-1/2$.
The scalar slow stability region $\mathcal{S}^{\textsc{1d}}_{\infty,\alpha}$  \eqref{eqn:scalar-slow-stability-region}  is shown in Figure  \ref{fig:MRI-explicit-stability_1D},  and the matrix stability region $\mathcal{S}^{\textsc{2d}}_{\rho,\alpha}$  \eqref{eqn:matrix-slow-stability-region}  in Figure  \ref{fig:MRI-explicit-stability_2D}.

\ifreport
%
%
%
MRI-GARK-ERK33a computes the solution \eqref{eqn:MIS-additive-internal-ode} as follows:
\begin{eqnarray*}
&& \Ys_{1} = y_{n}, \\
&& \begin{cases}
v(0) = \Ys_{1}, \\
 v' = \frac{1}{3} \, \funf \left( v \right) + \frac{1}{3}  \, \funs\bigl( Y^{\{\s\}}_1 \bigr), \quad \theta \in  [0,  H],  \\
\Ys_{2} = v(H), 
\end{cases} \\
&& \begin{cases}
v(0) = \Ys_{2}, \\
 v' = \frac{1}{3} \, \funf \left( v \right)  {\scriptstyle -}\frac{1}{3}  \, \funs\bigl( Y^{\{\s\}}_1 \bigr)  + \frac{2}{3}  \, \funs\bigl( Y^{\{\s\}}_2 \bigr) , \quad \theta \in  [0,  H],  \\
\Ys_{3} = v(H), 
\end{cases} \\
&& \begin{cases}
v(0) = \Ys_{3}, \\
 v' = \frac{1}{3} \, \funf \left( v \right) + \left( \frac{1}{2}\,\frac{\theta}{H} \right) \, \funs\bigl( Y^{\{\s\}}_1 \bigr)  {\scriptstyle -}\frac{2}{3}   \, \funs\bigl( Y^{\{\s\}}_2 \bigr) \\
\qquad + \left( 1 {\scriptstyle -}\frac{1}{2}\,\frac{\theta}{H} \right) \, \funs\bigl( Y^{\{\s\}}_3 \bigr), \quad \theta \in  [0,  H],  \\
y_{n+1} = v(H), 
\end{cases} \\
&& \begin{cases}
v(0) = \Ys_{3}, \\
 v' = \frac{1}{3} \, \funf \left( v \right) + \frac{1}{12}\, \funs\bigl( Y^{\{\s\}}_1 \bigr)  {\scriptstyle -}\frac{1}{3}   \, \funs\bigl( Y^{\{\s\}}_2 \bigr)
+\frac{7}{12}\,\funs\bigl( Y^{\{\s\}}_3 \bigr), \quad \theta \in  [0,  H],  \\
\widehat{y}_{n+1} = v(H).
\end{cases}
\end{eqnarray*}
\fi

\ifreport
\subsection{An order three explicit method with three stages and repeated abscissae}

The MRI-GARK-ERK33b scheme has the following coefficients:
\begin{tiny}
\begin{equation*}
\renewcommand{\arraystretch}{2}
\mathsf{A}^{\{\s\}} = \left[
\begin{array}{c|cccc}
\scriptstyle 0&\hphantom{-}\scriptstyle  0 &\scriptstyle  0 &\scriptstyle  0 \\
\frac{1}{2} & \hphantom{-}\frac{1}{2} &\scriptstyle  0 &\scriptstyle  0 \\
\scriptstyle 1&\scriptstyle  -1 &\scriptstyle  2 &\scriptstyle  0 \\
 \hline
\scriptstyle  1&\hphantom{-} \frac{1}{6} &
 \frac{2}{3} &
 \frac{1}{6} \\
\hline
\scriptstyle  1&\hphantom{-} \frac{1}{4} &
 \frac{1}{2} &
 \frac{1}{4} 
\end{array} \right],
\quad
\boldsymbol{\Upgamma}^0 = \left[
\begin{array}{cccc}
 \frac{1}{2} &\hphantom{-}\scriptstyle  0 &\scriptstyle  0 \\
 \frac{1}{2} &\hphantom{-}\scriptstyle  0 &\scriptstyle  0 \\
 \frac{7}{6} & {\scriptstyle -}\frac{4}{3} & \frac{1}{6} \\ 
 \hline
 \frac{5}{4} & {\scriptstyle -}\frac{3}{2} & \frac{1}{4} 
\end{array}
\right],
\quad
\boldsymbol{\Upgamma}^1 = \left[
\begin{array}{ccccc}
\hphantom{-}\scriptstyle 0 &\scriptstyle  0 &\scriptstyle  0 \\
\scriptstyle  -4 & \scriptstyle 4 &\scriptstyle  0 \\
\hphantom{-}\scriptstyle  0 &\scriptstyle  0 &\scriptstyle  0 \\
\hline
\hphantom{-}\scriptstyle  0 &\scriptstyle  0 &\scriptstyle  0 
\end{array}
\right].
\end{equation*}
\end{tiny}
%
%
\fi

\subsection{An order four explicit method with five stages and equidistant abscissae}

The MRI-GARK-ERK45a method is defined by:
\begin{tiny}
\begin{equation*}
\renewcommand{\arraystretch}{2}
\begin{split}
\mathsf{A}^{\{\s\}}
&=
\left[
\begin{array}{c|ccccc}
\scriptstyle 0&\hphantom{-}\scriptstyle  0 &\scriptstyle  0 &\scriptstyle  0 &\scriptstyle  0 &\scriptstyle  0 \\
\frac{1}{5}& \hphantom{-}\frac{1}{5} &\scriptstyle  0 &\scriptstyle  0 &\scriptstyle  0 &\scriptstyle  0 \\
\frac{2}{5}&  \hphantom{-}\frac{1}{32} & \frac{59}{160} &\scriptstyle  0 &\scriptstyle  0 &\scriptstyle  0 \\
\frac{3}{5}&  {\scriptstyle -}\frac{1}{2} & \frac{171}{64} & {\scriptstyle -}\frac{503}{320} &\scriptstyle  0 &\scriptstyle  0 \\
\frac{4}{5}&  \hphantom{-}\frac{125773}{379760} & {\scriptstyle -}\frac{183399}{379760} & \frac{175277}{189880} & \frac{136}{4747} &\scriptstyle  0 \\
\hline
\scriptstyle 1& \hphantom{-}\frac{1}{32}  &
 \frac{1}{3}  &
 \frac{11}{48}  &
 {\scriptstyle -}\frac{1}{12}  &
 \frac{47}{96} \\
 \hline
\scriptstyle 1 &\hphantom{-}\frac{1}{8}  &
 \frac{6403}{71670}  &
 \frac{28571}{71670}  &
 {\scriptstyle -}\frac{4681}{71670}  &
 \frac{129673}{286680}
\end{array}
\right],
\\
\boldsymbol{\Upgamma}^0 &= \left[
\begin{array}{ccccc}
 \frac{1}{5} &\scriptstyle  0 &\scriptstyle  0 &\scriptstyle  0 &\scriptstyle  0 \\
 {\scriptstyle -} \frac{53}{16} & \frac{281}{80} &\scriptstyle  0 &\scriptstyle  0 &\scriptstyle  0 \\
  {\scriptstyle -}\frac{36562993}{71394880} & \frac{34903117}{17848720} & {\scriptstyle -}\frac{88770499}{71394880} &\scriptstyle  0 &\scriptstyle  0 \\
 {\scriptstyle -}\frac{7631593}{71394880} & {\scriptstyle -}\frac{166232021}{35697440} & \frac{6068517}{1519040} & \frac{8644289}{8924360} &\scriptstyle  0 \\
 \frac{277061}{303808} & {\scriptstyle -}\frac{209323}{1139280} & {\scriptstyle -}\frac{1360217}{1139280} & {\scriptstyle -}\frac{148789}{56964} & \frac{147889}{45120} \\
 \hline
{\scriptstyle -} \frac{88227}{47470}  &   \frac{756870829}{340217490} &  {\scriptstyle -} \frac{713704111}{1360869960} &  {\scriptstyle -} \frac{31967827}{340217490} &   \frac{129673}{286680}
\end{array}
\right],
\\
\boldsymbol{\Upgamma}^1 &= \left[
\begin{array}{ccccc}
\scriptstyle 0 &\scriptstyle  0 &\scriptstyle  0 &\scriptstyle  0 &\scriptstyle  0 \\
 \frac{503}{80} & {\scriptstyle -}\frac{503}{80} &\scriptstyle  0 &\scriptstyle  0 &\scriptstyle  0 \\
 {\scriptstyle -}\frac{1365537}{35697440} & \frac{4963773}{7139488} & {\scriptstyle -}\frac{1465833}{2231090} &\scriptstyle  0 &\scriptstyle  0 \\
 \frac{66974357}{35697440} & \frac{21445367}{7139488} & -\scriptstyle3 & {\scriptstyle -}\frac{8388609}{4462180} &\scriptstyle  0 \\
 {\scriptstyle -}\frac{18227}{7520} &\scriptstyle 2 &\scriptstyle 1 &\scriptstyle 5 & {\scriptstyle -}\frac{41933}{7520} \\
 \hline
 \frac{6213}{1880} & {\scriptstyle -}\frac{6213}{1880} &\scriptstyle  0 &\scriptstyle  0 &\scriptstyle 0
\end{array}
\right].
\end{split}
\end{equation*}
\end{tiny}
The scalar slow stability region $\mathcal{S}^{\textsc{1d}}_{\infty,\alpha}$  \eqref{eqn:scalar-slow-stability-region}  is shown in Figure  \ref{fig:MRI-explicit-stability_2D},  and the matrix stability region $\mathcal{S}^{\textsc{2d}}_{\rho,\alpha}$  \eqref{eqn:matrix-slow-stability-region}  in Figure  \ref{fig:MRI-explicit-stability_2D}.

\ifreport
\subsection{An order four method, five stages, with repeated abscissae, and FSAL}

We now consider a four stages slow base method, with the fifth stage computing the next step - this is the first-same-as-last (FSAL) property. The last stage is only necessary in order to build an embedded error estimator, otherwise the base method needs only four stages. The MRI-GARK-ERK45b method reads:
\begin{tiny}
\[
\renewcommand{\arraystretch}{2}
\mathsf{A}^{\{\s\}} = \left[
\begin{array}{c|ccccc}
\scriptstyle 0 &\scriptstyle  0 &\hphantom{-}\scriptstyle  0 &\scriptstyle 0 &\scriptstyle  0 &\hphantom{-}\scriptstyle 0 \\
\frac{3}{5} &  \frac{3}{5} &\hphantom{-}\scriptstyle  0 &\scriptstyle 0 &\scriptstyle  0 &\hphantom{-}\scriptstyle 0 \\
 \frac{3}{4} & \frac{39}{32} & {\scriptstyle -}\frac{15}{32} &\scriptstyle  0 &\scriptstyle 0 &\hphantom{-}\scriptstyle 0 \\
\scriptstyle 1 &  \frac{23}{27} &\hphantom{-} \frac{20}{27} & {\scriptstyle -}\frac{16}{27} &\scriptstyle  0 &\hphantom{-}\scriptstyle 0 \\
\scriptstyle 1 &  \frac{5}{27} & \hphantom{-}\frac{125}{108} & {\scriptstyle -}\frac{16}{27} & \frac{1}{4} &\hphantom{-}\scriptstyle 0 \\
 \hline
\scriptstyle 1  &  \frac{5}{27} &\hphantom{-} \frac{125}{108} & {\scriptstyle -}\frac{16}{27} & \frac{1}{4} &\hphantom{-}\scriptstyle 0 \\
 \hline
\scriptstyle 1   & \frac{1}{7} &\hphantom{-} \frac{425}{252} & {\scriptstyle -}\frac{80}{63} & \frac{23}{28} & {\scriptstyle -}\frac{8}{21}
\end{array}
\right],
\qquad
\boldsymbol{\Upgamma}^0 =
\left[
\begin{array}{ccccc}
 \hphantom{-}\frac{3}{5} &\scriptstyle  0 &\scriptstyle  0 &\scriptstyle  0 &\hphantom{-}\scriptstyle 0 \\
 {\scriptstyle -}\frac{16201}{2160} & \frac{3305}{432} &\scriptstyle  0 &\scriptstyle 0 &\hphantom{-}\scriptstyle 0 \\
 {\scriptstyle -}\frac{7859}{216} & \frac{4177}{216} & \frac{467}{27} &\scriptstyle  0 &\hphantom{-}\scriptstyle 0 \\
 {\scriptstyle -}\frac{2}{3} & \frac{5}{12} &\scriptstyle  0 & \frac{1}{4} &\hphantom{-}\scriptstyle 0 \\
\hphantom{-}\scriptstyle 0 &\scriptstyle  0 &\scriptstyle 0 &\scriptstyle  0 &\hphantom{-}\scriptstyle 0 \\
 \hline
 {\scriptstyle -}\frac{8}{189} & \frac{100}{189} & {\scriptstyle -}\frac{128}{189} & \frac{4}{7} & {\scriptstyle -}\frac{8}{21}
\end{array}
\right],
\]
\end{tiny}
\begin{tiny}
\[
\renewcommand{\arraystretch}{2}
\boldsymbol{\Upgamma}^1 =
\left[
\begin{array}{ccccc}
\scriptstyle 0 &\scriptstyle  0 &\scriptstyle 0 &\scriptstyle  0 &\scriptstyle 0 \\
 \frac{6205}{144} & {\scriptstyle -}\frac{6205}{144} &\scriptstyle  0 &\scriptstyle 0 &\scriptstyle 0 \\
 \frac{26213}{144} & {\scriptstyle -}\frac{9733}{144} & {\scriptstyle -}\frac{1030}{9} &\scriptstyle  0 &\scriptstyle 0 \\
 \scriptstyle 0 &\scriptstyle  0 &\scriptstyle 0 &\scriptstyle  0 &\scriptstyle 0 \\
\scriptstyle 0 &\scriptstyle  0 &\scriptstyle 0 &\scriptstyle  0 &\scriptstyle 0 \\
 \hline
\scriptstyle 0 &\scriptstyle  0 &\scriptstyle 0 &\scriptstyle  0 &\scriptstyle 0
\end{array}
\right],
\qquad
\boldsymbol{\Upgamma}^2 =
\left[
\begin{array}{ccccc}
\scriptstyle 0 &\scriptstyle  0 &\scriptstyle  0 &\scriptstyle  0 &\scriptstyle 0 \\
 {\scriptstyle -}\frac{725}{18} & \frac{725}{18} &\scriptstyle  0 &\scriptstyle 0 &\scriptstyle 0 \\
\scriptstyle -165 &\scriptstyle 47 &\scriptstyle 118 &\scriptstyle  0 &\scriptstyle 0 \\
\scriptstyle 0 &\scriptstyle  0 &\scriptstyle 0 &\scriptstyle  0 &\scriptstyle 0 \\
\scriptstyle 0 &\scriptstyle  0 &\scriptstyle 0 &\scriptstyle  0 &\scriptstyle 0 \\
 \hline
\scriptstyle 0 &\scriptstyle  0 &\scriptstyle 0 &\scriptstyle  0 &\scriptstyle 0 
\end{array}
\right].
\]
\end{tiny}
%
\fi

\begin{figure}[h]
\centering
\subfigure[ERK22a: $\mathcal{S}^{\textsc{1d}}_{\rho=\infty,\alpha}$]{
\includegraphics[width=0.3\linewidth]{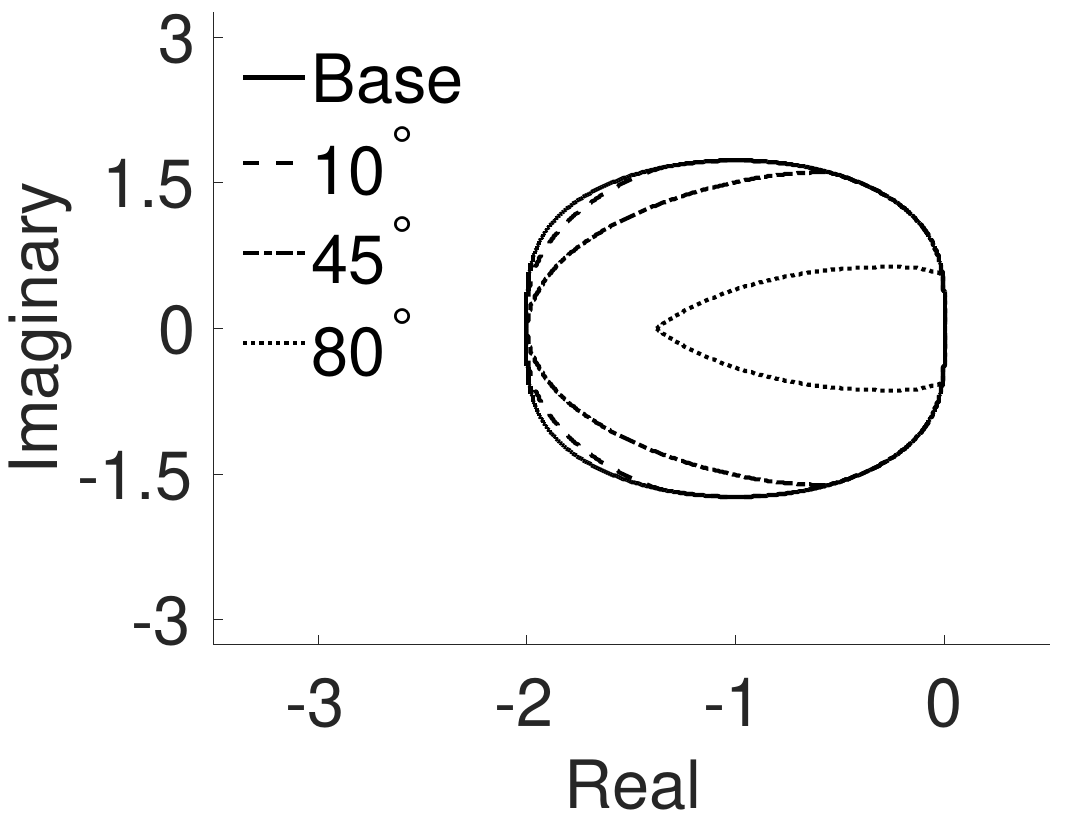}
\label{fig:stab1D-ERK22a}
}
\subfigure[ERK33a: $\mathcal{S}^{\textsc{1d}}_{\rho=\infty,\alpha}$]{
\includegraphics[width=0.3\linewidth]{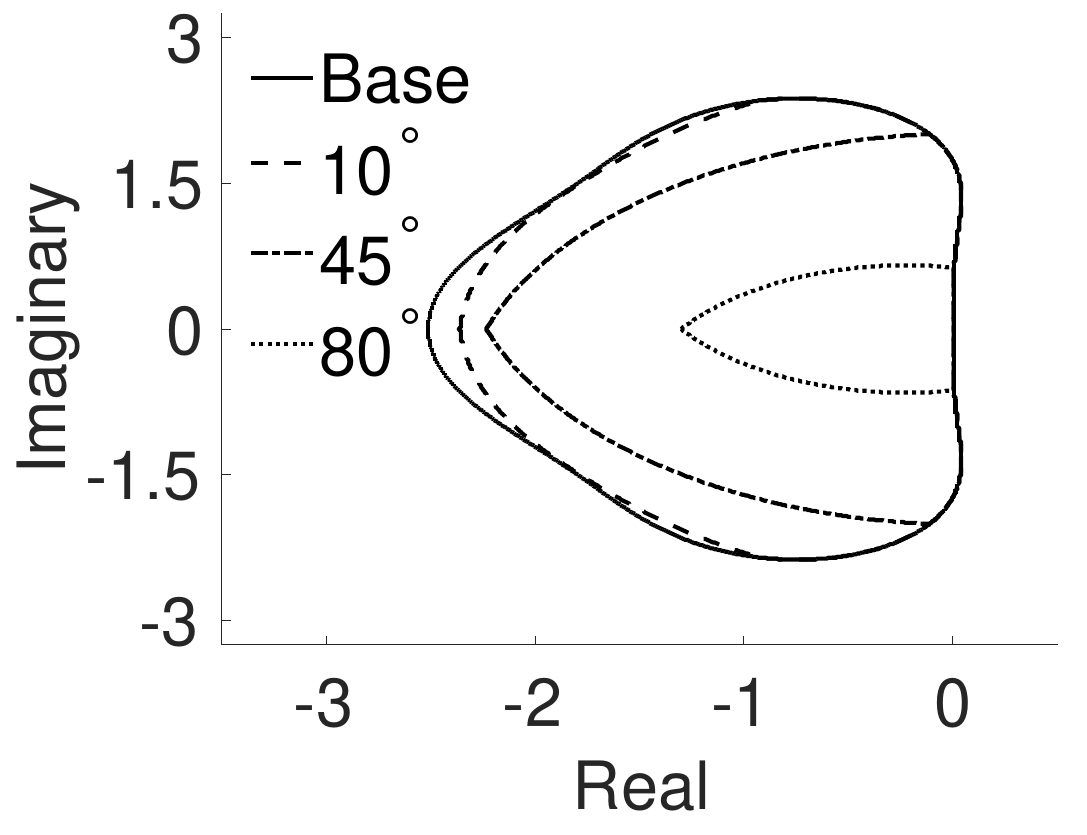}
\label{fig:stab1D-ERK33a}
}
\subfigure[ERK45a: $\mathcal{S}^{\textsc{1d}}_{\rho=\infty,\alpha}$]{
\includegraphics[width=0.3\linewidth]{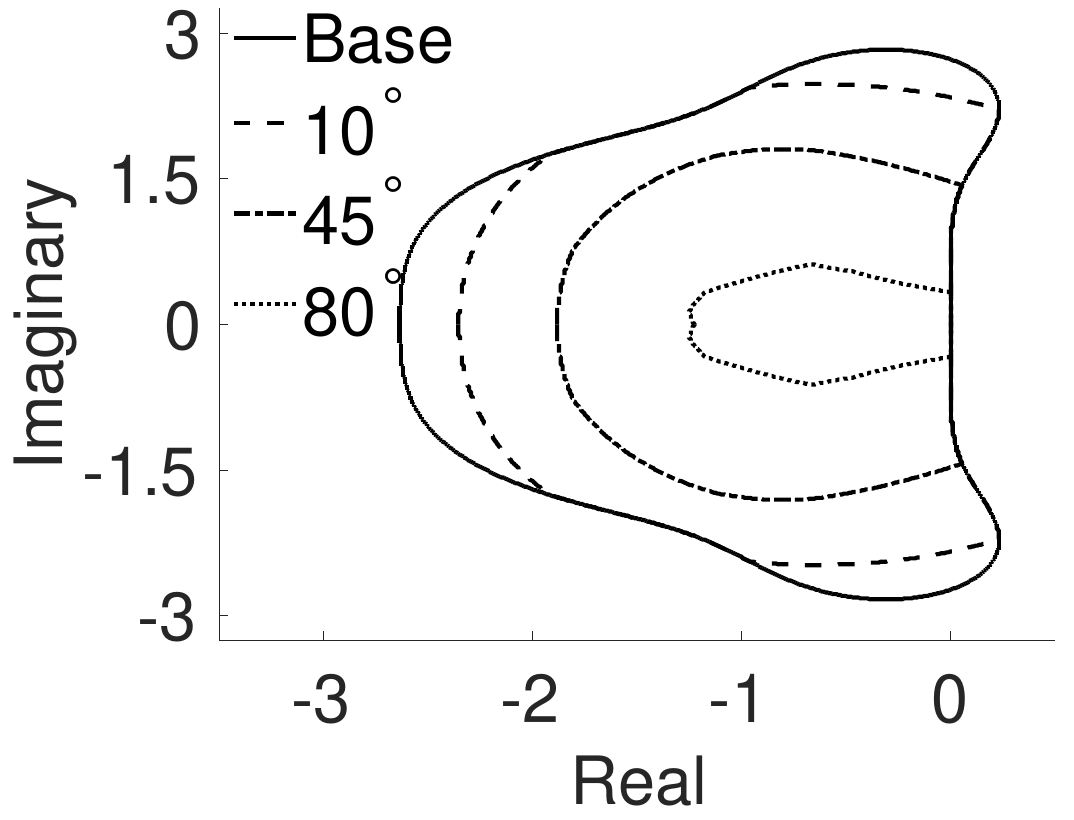}
\label{fig:stab1D-ERK45a}
}
%
\caption{Scalar slow stability regions \eqref{eqn:scalar-slow-stability-region} of explicit MRI-GARK schemes. For $\rho=\infty$ they correspond to $A(\alpha)$ unconditional stability in the fast variable. The stability degrades with increasing $\alpha$.} 
\label{fig:MRI-explicit-stability_1D}
\end{figure}

\begin{figure}[h]
\centering
\ifreport
\subfigure[ERK22a: $\mathcal{S}^{\textsc{2d}}_{\rho=10,\alpha=10^\circ}$]{
\includegraphics[width=0.3\linewidth]{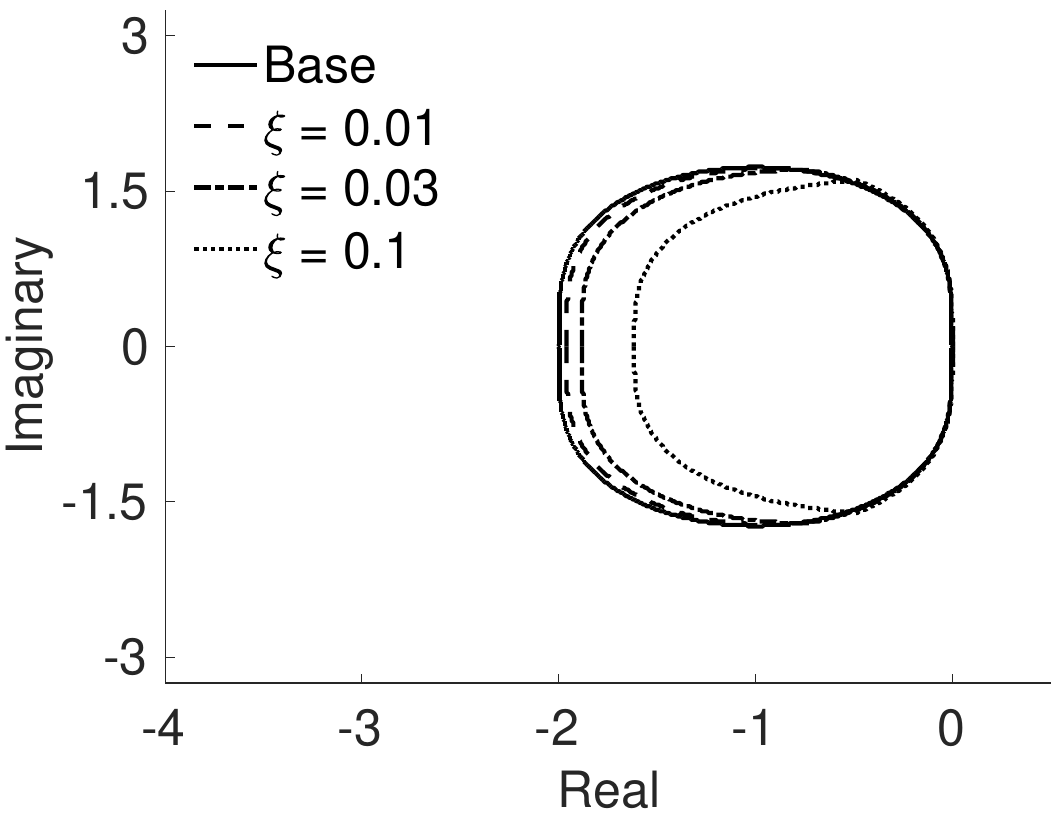}
\label{fig:stab2D-ERK22a-10}
}
\fi
\subfigure[ERK22a: $\mathcal{S}^{\textsc{2d}}_{\rho=10,\alpha=45^\circ}$]{
\includegraphics[width=0.3\linewidth]{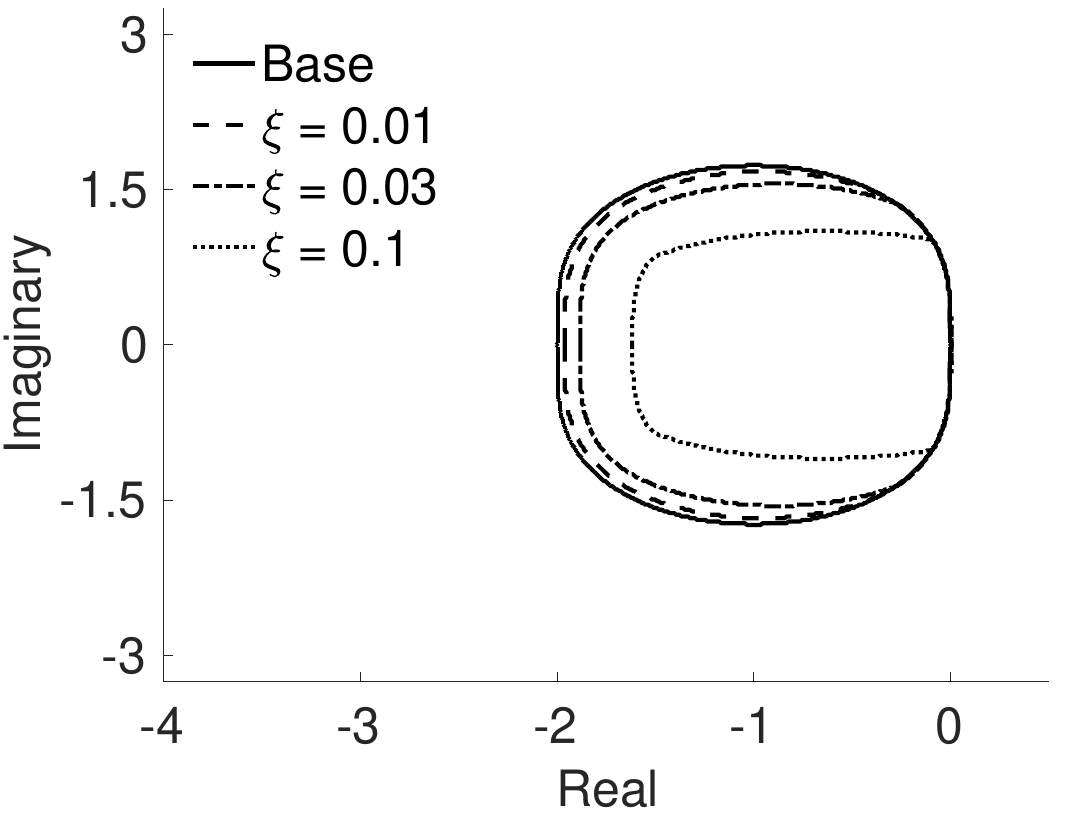}
\label{fig:stab2D-ERK22a-45}
}
\ifreport
\subfigure[ERK22a: $\mathcal{S}^{\textsc{2d}}_{\rho=10,\alpha=80^\circ}$]{
\includegraphics[width=0.3\linewidth]{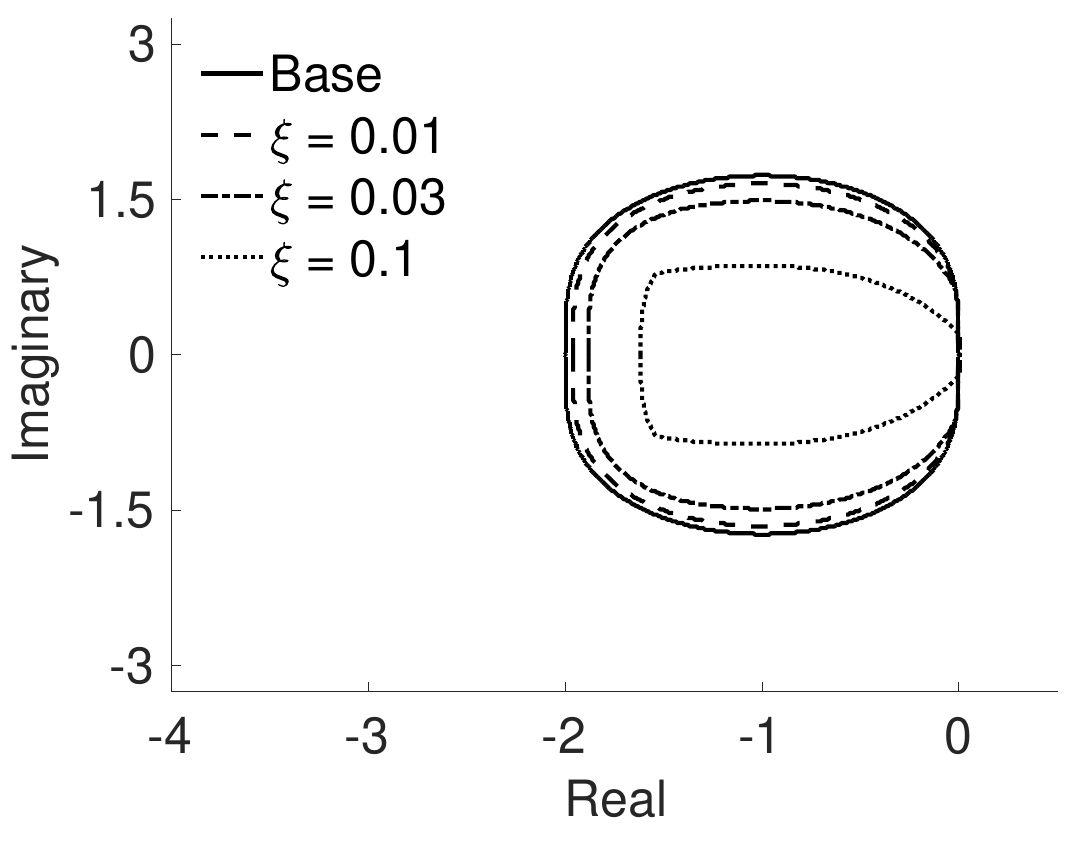}
\label{fig:stab2D-ERK22a-80}
}
\fi
\ifreport
\subfigure[ERK33a: $\mathcal{S}^{\textsc{2d}}_{\rho=10,\alpha=10^\circ}$]{
\includegraphics[width=0.3\linewidth]{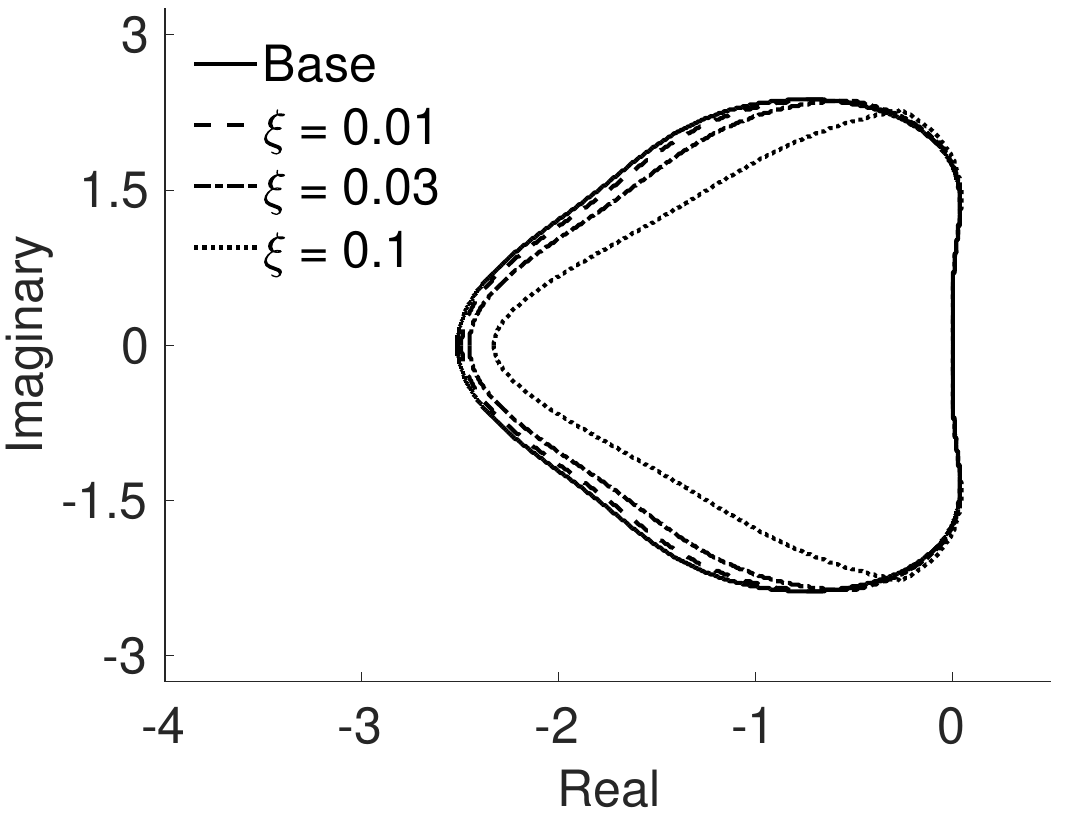}
\label{fig:stab2D-ERK33a-10}
}
\fi
\subfigure[ERK33a: $\mathcal{S}^{\textsc{2d}}_{\rho=10,\alpha=45^\circ}$]{
\includegraphics[width=0.3\linewidth]{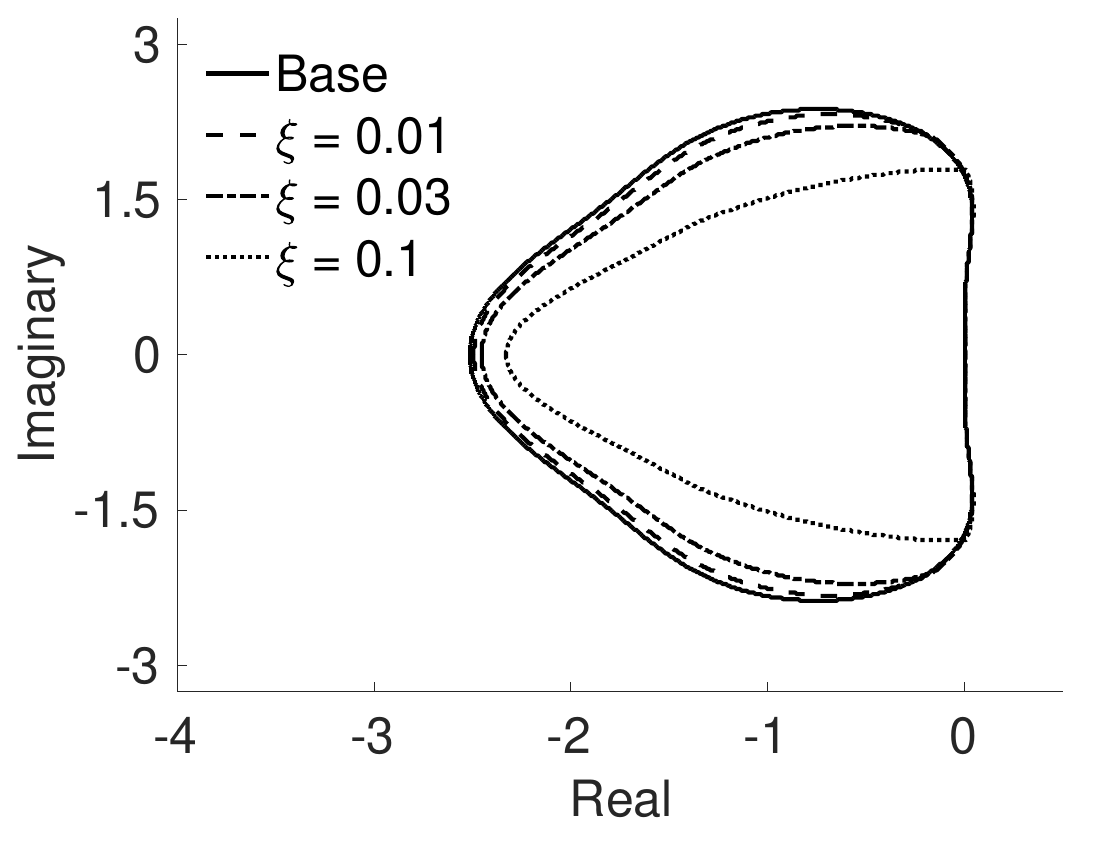}
\label{fig:stab2D-ERK33a-45}
}
\ifreport
\subfigure[ERK33a: $\mathcal{S}^{\textsc{2d}}_{\rho=10,\alpha=80^\circ}$]{
\includegraphics[width=0.3\linewidth]{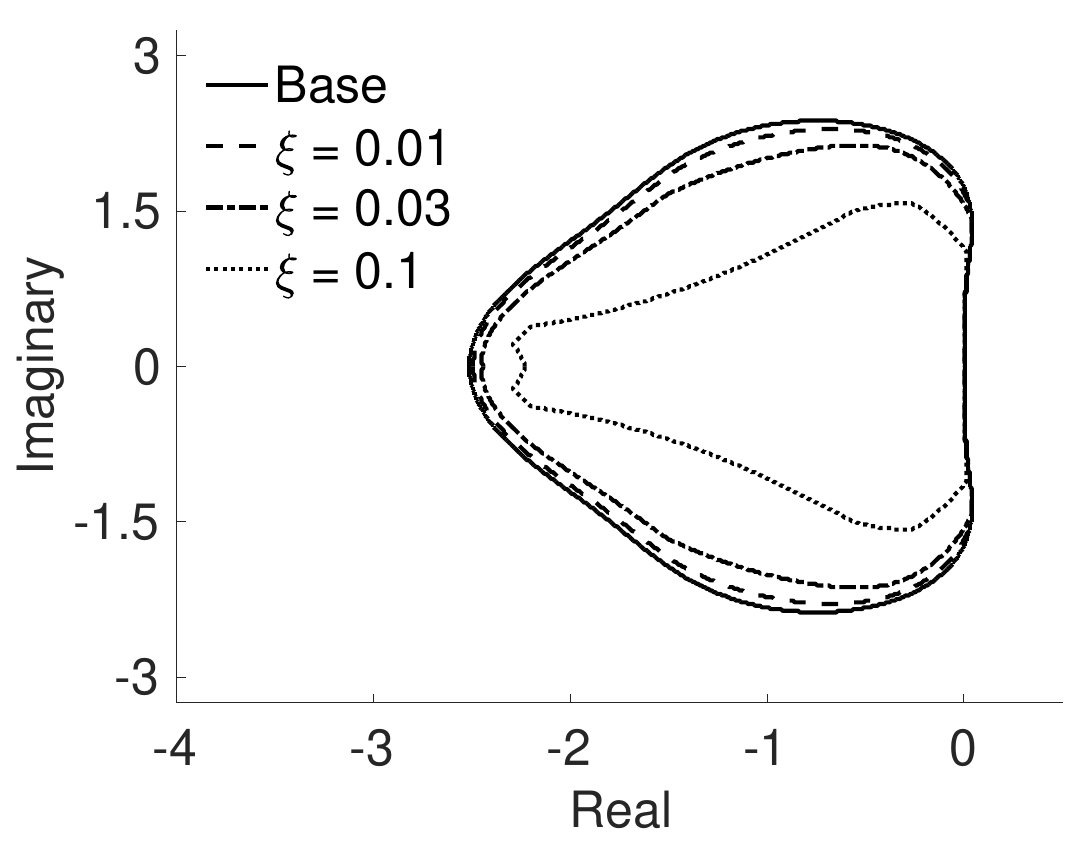}
\label{fig:stab2D-ERK33a-80}
}
\fi
\ifreport
\subfigure[ERK45a: $\mathcal{S}^{\textsc{2d}}_{\rho=10,\alpha=10^\circ}$]{
\includegraphics[width=0.3\linewidth]{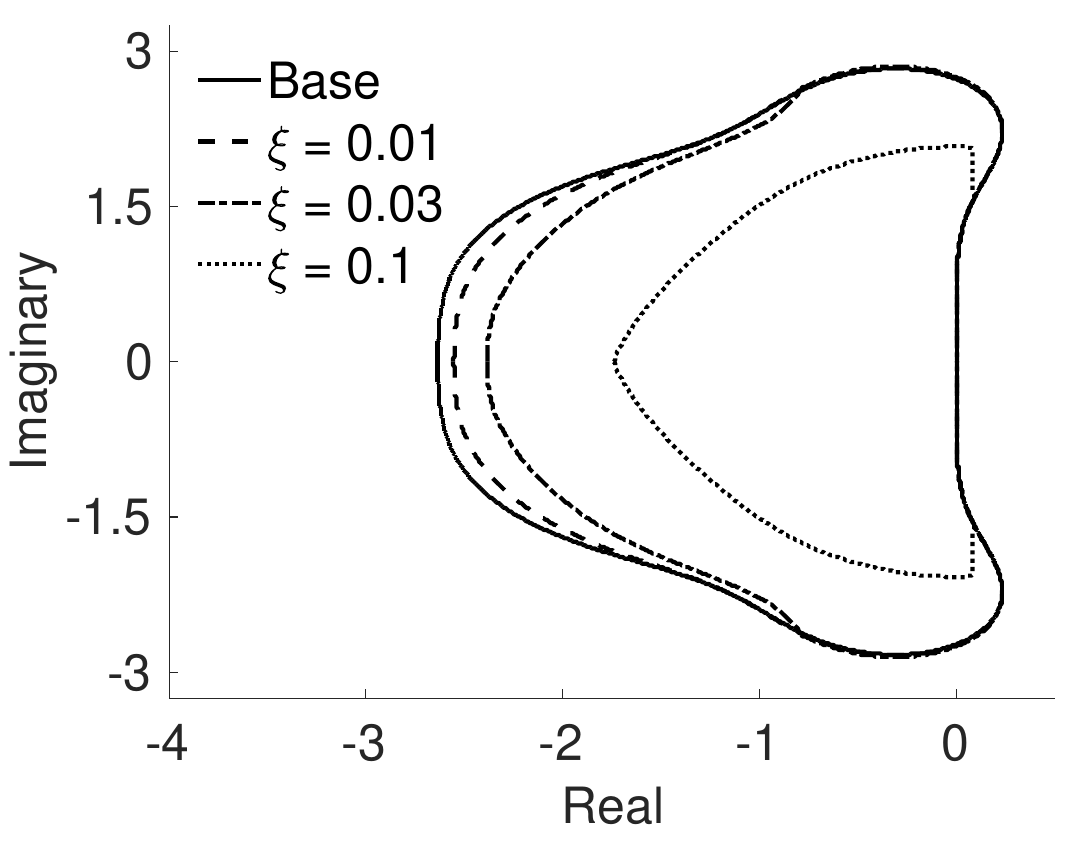}
\label{fig:stab2D-ERK45a}
}
\fi
\subfigure[ERK45a: $\mathcal{S}^{\textsc{2d}}_{\rho=10,\alpha=45^\circ}$]{
\includegraphics[width=0.3\linewidth]{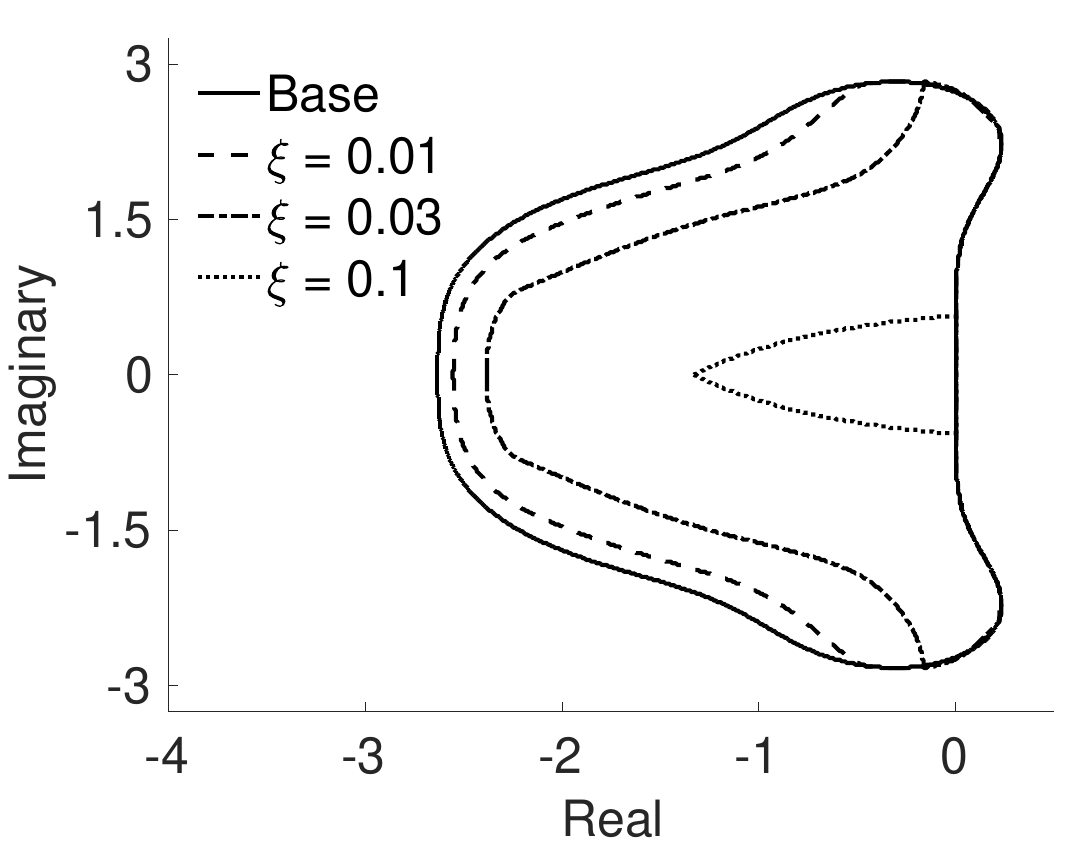}
\label{fig:stab2D-ERK45a-45}
}
\ifreport
\subfigure[ERK45a: $\mathcal{S}^{\textsc{2d}}_{\rho=10,\alpha=80^\circ}$]{
\includegraphics[width=0.3\linewidth]{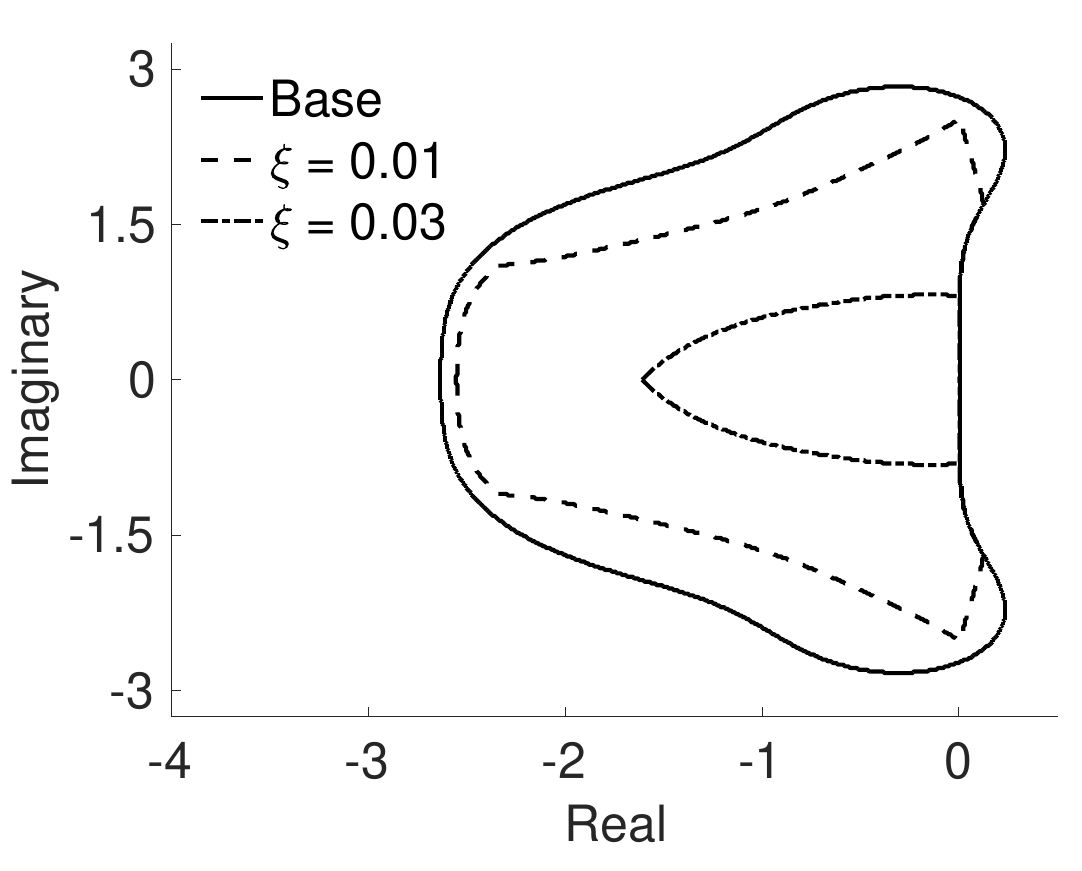}
\label{fig:stab2D-ERK45a-80}
}
\fi
\caption{Matrix slow stability regions \eqref{eqn:matrix-slow-stability-region} of explicit MRI-GARK schemes. The $\xi=0$ case corresponds to the base slow method. The stability is conditional with respect to the fast variable, and degrades with increasing $\alpha$. The stability also degrades with an increasing influence $\xi$ of the fast system on the slow one.} 
\label{fig:MRI-explicit-stability_2D}
\end{figure}


\section{Decoupled implicit methods}
\label{sec:implicit-methods}

Decoupled implicit methods compute implicitly the slow solution stages, and possible the fast solution stages. The decoupled aspect means that implicitness is within the slow and the fast processes, but no computation involves solving a fully coupled (fast plus slow) system of equations.

The idea of building a decoupled implicit scheme is to have pairs of consecutive stages with the same abscissae , such that different pairs have different values, e.g., $\cbase_{i-1}  < \cbase_{i}=\cbase_{i+1} < \cbase_{i+2}$. The stages with the lower index in each pair are explicit, and allow the fast solution to be propagated forward in time, e.g., from $\cbase_{i-1}$ to $\cbase_{i}$. The stages with the higher index in each pair are implicit, the fast solution is not advanced, and the slow solution is computed implicitly. 
\ifreport
For our example the scheme computes the two consecutive stages of \eqref{eqn:MIS-additive} as follows:
\begin{subequations}
\label{eqn:MIS-impl-additive}
\begin{eqnarray}
%
\label{eqn:MIS-impl-additive-explicit-stage}
\qquad\qquad&& \begin{cases}
v(0) = \Ys_{i-1}, \\
 v' = \Delta \cbase_{i}  \, \funf \left( v \right) + \sum_{j=1}^{i-1} \gamma_{i-1,j}\left(\frac{\theta}{H}\right) \, \funs\bigl( Y^{\{\s\}}_j \bigr), \quad \theta \in  [0,  H],  \\
\Ys_{i} = v(H), 
\end{cases}
\\
\label{eqn:MIS-impl-additive-implicit-stage}
&& \Ys_{i+1} = \Ys_{i} + H\,\sum_{j=1}^{i+1} \overline{\gamma}_{i,j} \, \funs\bigl( Y^{\{\s\}}_j \bigr).
\end{eqnarray}
\end{subequations}
To promote stability we build the slow tendency terms using linear combinations of only the implicit stages, i.e., the coefficients $\gamma_{i,j}\left(\tau\right)=0$ for all explicit stages $j$.
\fi

In designing a decoupled implicit MRI-GARK scheme one starts with a diagonally implicit slow base method, with an explicit first stage, and with increasing abscissae $\cbase_{i}<\cbase_{i+1}$. If the base method does not have an explicit first stage one appends a first row of zeros and a first column of zeros to its Butcher tableau.
The base (E)SDIRK scheme is extended with $s^{\{\s\}}$ additional stages such that $\cbase_{i}< \mathbase{c}_{i}^\textnormal{ext} = \cbase_{i+1}$ for $i=1,\cdots,s^{\{\s\}}$. In the extended scheme the base method stages are odd-numbered, and the added stages are even-numbered. We choose to use only slow function values evaluated at the base method stages (not at the additional stages). For this, in the extended matrix of coefficients we have $\mathbase{a}_{i,2j}^\textnormal{ext}=0$ for any additional stage $2j$. The additional stages are explicit, i.e., $\mathbase{a}_{2i,j}^\textnormal{ext}=0$ for any additional stage $2i$ and $j \ge 2i$. The additional entries in the vector of weights are zeros, $\mathbase{b}_{2j}^\textnormal{ext}=0$. 

\subsection{Implicit trapezoidal method}
%
The implicit trapezoidal scheme \eqref{eqn:implicit-trapezoidal}, named MRI-GARK-IRK21a, has the coefficients:
\begin{tiny}
\begin{equation*}
\renewcommand{\arraystretch}{2}
\mathsf{A}^{\{\s\}} = 
\left[
\begin{array}{c | ccc}
\scriptstyle 0 &\scriptstyle  0 &\scriptstyle  0 &\scriptstyle  0 \\
\scriptstyle 1 &\scriptstyle  1 &\scriptstyle  0 &\scriptstyle  0 \\
\scriptstyle 1 &   \frac{1}{2} &\scriptstyle  0 & \frac{1}{2} \\ 
 \hline
\scriptstyle 1 &   \frac{1}{2} &\scriptstyle  0 & \frac{1}{2} \\
\hline
\scriptstyle 1 &\scriptstyle  0 &\scriptstyle  0 &\scriptstyle  1
\end{array}
\right]
\equiv
\left[
\begin{array}{c | cc}
\scriptstyle 0  &\scriptstyle  0 &\scriptstyle  0 \\
\scriptstyle 1 &\scriptstyle  1  &\scriptstyle  0 \\
\scriptstyle 1 &   \frac{1}{2}  & \frac{1}{2} \\ 
 \hline
\scriptstyle 1 &   \frac{1}{2}  & \frac{1}{2} \\
\hline
\scriptstyle 1 &\scriptstyle  0  &\scriptstyle  1
\end{array}
\right],
\quad
{\Gamma}^0 = \left[
\begin{array}{ccc}
\hphantom{-}\scriptstyle  1 &\scriptstyle  0 &\scriptstyle  0 \\
 {\scriptstyle -}\frac{1}{2} &\scriptstyle  0 & \frac{1}{2} \\
\hphantom{-}\scriptstyle  0 &\scriptstyle  0 &\scriptstyle  0 \\
 \hline
 {\scriptstyle -}\frac{1}{2} &\scriptstyle  0 & \frac{1}{2} 
\end{array}
\right]
\equiv
\left[
\begin{array}{cc}
\hphantom{-}\scriptstyle  1  &\scriptstyle  0 \\
 {\scriptstyle -}\frac{1}{2}  & \frac{1}{2} \\
\hphantom{-}\scriptstyle  0  &\scriptstyle  0 \\
 \hline
 {\scriptstyle -}\frac{1}{2}  & \frac{1}{2} 
\end{array}
\right].
\end{equation*}
\end{tiny}
Since only the odd-numbered columns are non-zero we omit the even-numbered columns of zeros from the method coefficient tables.

The scalar slow stability region $\mathcal{S}^{\textsc{1d}}_{\infty,\alpha}$  \eqref{eqn:scalar-slow-stability-region}  for the implicit trapezoidal method is shown in Figure  \ref{fig:MRI-implicit-stability_1D}. The matrix stability region $\mathcal{S}^{\textsc{2d}}_{\rho=10,\alpha}$ \eqref{eqn:matrix-slow-stability-region}  is shown in Figure \ref{fig:MRI-implicit-stability_2D}.
%

%
\subsection{An ESDIRK order 3 method, stiffly accurate, with equidistant abscissae}
%
%
The scheme MRI-GARK-ESDIRK34a is defined by:
\begin{tiny}
\begin{equation*}
\renewcommand{\arraystretch}{2}
\begin{split}
\mathsf{A}^{\{\s\}} &= \left[
\begin{array}{c | cccc}
%
\scriptstyle  {0}&\scriptstyle   {0} &\scriptstyle   {0} & \scriptstyle   {0}  &\scriptstyle   {0} \\
\frac{1}{3}& \frac{1}{3} &\scriptstyle  0  &\scriptstyle  0 &\scriptstyle  0  \\
{\frac{1}{3}}&  {\frac{1-3\lambda}{3}}  & \scriptstyle {\lambda}   &\scriptstyle  {0}  &\scriptstyle  {0} \\
\frac{2}{3}&  \frac{-24 \lambda ^2+4 \lambda +1}{24 \lambda -6} & \frac{24 \lambda ^2+12 \lambda -5}{24 \lambda -6}  &\scriptstyle  0  &\scriptstyle  0 \\
{\frac{2}{3}}&  {\frac{\lambda }{3-12 \lambda}}  & {\frac{2 \left(6 \lambda ^2-6 \lambda +1\right)}{3-12 \lambda}}  &\scriptstyle {\lambda}   &\scriptstyle  {0} \\
\scriptstyle  1&  \frac{1}{4} & {\scriptstyle   3 \lambda }  & \frac{3-12\lambda}{4}  &\scriptstyle  0  \\
\scriptstyle  {1}& {\frac{1-4\lambda}{4}}   & \scriptstyle   {3 \lambda }   & {\frac{3-12\lambda}{4}}   &\scriptstyle {\lambda}  \\
 \hline
\scriptstyle  1&{ \frac{1-4\lambda}{4} }  & \scriptstyle   {3 \lambda}   & {\frac{3-12\lambda}{4}}  &\scriptstyle { \lambda} \\
\hline
\scriptstyle  1& \scriptstyle  \widehat{b}_1 & \scriptstyle  \widehat{b}_3 & \scriptstyle  \widehat{b}_5 & \scriptstyle  \widehat{b}_7
\end{array}
\right],
\\
{\Gamma}^0 &= \left[
\begin{array}{cccc}
 \frac{1}{3} &\scriptstyle  0 &\scriptstyle  0 &\scriptstyle  0 \\
\scriptstyle -\lambda  &\scriptstyle  \lambda  &\scriptstyle  0 &\scriptstyle  0 \\
 \frac{3-10 \lambda }{24 \lambda -6} & \frac{5-18 \lambda }{6-24 \lambda } &\scriptstyle  0 &\scriptstyle  0 \\
 \frac{-24 \lambda ^2+6 \lambda +1}{6-24 \lambda }  & \frac{-48 \lambda ^2+12 \lambda +1}{24 \lambda -6}  &\scriptstyle  \lambda   &\scriptstyle  0 \\
 \frac{3-16 \lambda }{12-48 \lambda }  & \frac{48 \lambda ^2-21 \lambda +2}{12 \lambda -3}  & \frac{3-16\lambda}{4}   &\scriptstyle  0 \\
\scriptstyle  -\lambda   &\scriptstyle  0  &\scriptstyle  0  & \scriptstyle \lambda  \\
\scriptstyle  0 &\scriptstyle  0 &\scriptstyle  0  &\scriptstyle  0 \\
\hline
\scriptstyle  \widehat{\gamma}^0_{7,1} &\scriptstyle   \widehat{\gamma}^0_{7,3} &\scriptstyle   \widehat{\gamma}^0_{7,5} &\scriptstyle   \widehat{\gamma}^0_{7,7}
\end{array}
\right],
\end{split}
\end{equation*}
\end{tiny}
with
\begin{tiny}
\begin{equation*}
\renewcommand{\arraystretch}{2}
\begin{split}
\widehat{b}_{\{1,3,5,7\}} &= 
{\scriptstyle \frac{1}{4 \left(6 \lambda ^2-6 \lambda +1\right)^2} }\,
\left[
\begin{array}{c}
{\scriptstyle 576 \lambda ^6+1008 \lambda ^5-2082 \lambda ^4+1236 \lambda ^3-333 \lambda
   ^2+42 \lambda -2} \\
  {\scriptstyle -3 \left(216
   \lambda ^6+360 \lambda ^5-762 \lambda ^4+456 \lambda ^3-129 \lambda ^2+18 \lambda
   -1\right)}\\
  {\scriptstyle 9 (1-4 \lambda )^2 \lambda  \left(6 \lambda ^3+12 \lambda ^2-13 \lambda +2\right)} \\
 \scriptstyle -12 \lambda ^2 \left(2 \lambda ^2+4 \lambda -1\right)\,\left(6 \lambda ^2-6 \lambda +1\right)
   \end{array}
\right], \\
\widehat{\gamma}^0_{7,\{1,3,5,7\}}
&=
{\scriptstyle \frac{1}{4 \left(6 \lambda ^2-6 \lambda +1\right)^2}}\,
\left[
\begin{array}{c}
\scriptstyle 576 \lambda ^6+1152 \lambda ^5-2406 \lambda ^4+1500 \lambda ^3-429 \lambda ^2+58 \lambda -3 \\
 \scriptstyle  - 6 \left(216 \lambda ^6+432 \lambda ^5-906 \lambda ^4+552 \lambda ^3-153 \lambda ^2+20 \lambda -1\right) \\ 
\scriptstyle 3 (4 \lambda -1) \left(72 \lambda ^5+162 \lambda ^4-264 \lambda ^3+111 \lambda ^2-18 \lambda +1\right) \\
\scriptstyle     -4 \lambda  \left(6 \lambda ^3+18 \lambda ^2-9 \lambda +1\right) (6 \lambda ^2-6 \lambda +1) 
\end{array}
\right].
\end{split}
\end{equation*}
\end{tiny}
The diagonal coefficient satisfies:
\begin{equation}
\label{eqn:lambda3}
-1+9 \lambda-18 \lambda^2+6 \lambda^3 = 0
\quad \Rightarrow \quad
\lambda \approx 0.435866521508458999416019.
\end{equation}
The scalar slow stability region $\mathcal{S}^{\textsc{1d}}_{\infty,\alpha}$  \eqref{eqn:scalar-slow-stability-region} is shown in Figure  \ref{fig:MRI-implicit-stability_1D}. The matrix stability region $\mathcal{S}^{\textsc{2d}}_{\rho=10,\alpha}$ \eqref{eqn:matrix-slow-stability-region}  is shown in Figure \ref{fig:MRI-implicit-stability_2D}.

\subsection{A SDIRK order 3 method, stiffly accurate}
The scheme MRI-GARK-SDIRK33a is defined by the following coefficients:
\begin{tiny}
\begin{equation*}
\renewcommand{\arraystretch}{2}
\begin{split}
\mathsf{A}^{\{\s\}} &= \left[ \begin{array}{c | cccc}
\scriptstyle 0&\scriptstyle  0  &\scriptstyle  0 &\scriptstyle  0 &\scriptstyle  0 \\
\scriptstyle  \lambda&\scriptstyle   \lambda   &\scriptstyle  0 &\scriptstyle  0 &\scriptstyle  0 \\
\scriptstyle  \lambda&\scriptstyle  0  &\scriptstyle  \lambda   &\scriptstyle  0  &\scriptstyle  0 \\
\frac{6 \lambda ^2-9 \lambda +2}{6 \lambda ^2-12 \lambda +3}&\scriptstyle  0  & \frac{6 \lambda ^2-9 \lambda +2}{6 \lambda ^2-12 \lambda +3} 0 &\scriptstyle  0 &\scriptstyle  0 \\
\frac{6 \lambda ^2-9 \lambda +2}{6 \lambda ^2-12 \lambda +3}&\scriptstyle  0 & \frac{6 \lambda ^2-9 \lambda +2}{6 \lambda ^2-12 \lambda +3}-{\scriptstyle   \lambda}   &\scriptstyle  \lambda  &\scriptstyle  0  \\
\scriptstyle 1&\scriptstyle  0  &\scriptstyle 3 \lambda   &\scriptstyle   1-3 \lambda  &\scriptstyle  0  \\
\scriptstyle 1&\scriptstyle  0  & \frac{1-4 \lambda }{-12 \lambda ^3+36 \lambda ^2-24 \lambda +4}  & {\scriptstyle -}\frac{3 \left(2 \lambda ^2-4 \lambda +1\right)^2}{4 \left(3 \lambda ^3-9\lambda ^2+6 \lambda -1\right)}  &\scriptstyle   \lambda  \\
   \hline
\scriptstyle  1&\scriptstyle  0  & \frac{1-4 \lambda }{-12 \lambda ^3+36 \lambda ^2-24 \lambda +4}  & {\scriptstyle -}\frac{3 \left(2 \lambda ^2-4 \lambda +1\right)^2}{4 \left(3 \lambda ^3-9
   \lambda ^2+6 \lambda -1\right)}  &\scriptstyle   \lambda  \\
   \hline
\scriptstyle  1&\scriptstyle  0 &
 \frac{1}{2-2 \lambda } &
\scriptstyle  0 &
 \frac{1-2 \lambda }{2-2 \lambda }
   \end{array}\right],
\\
\Upgamma^0  &= \left[
\begin{array}{cccc}
 \frac{1}{3}  &\scriptstyle  0  &\scriptstyle  0  &\scriptstyle  0 \\
 -{\scriptstyle \lambda}  & {\scriptstyle \lambda}  &\scriptstyle  0  &\scriptstyle  0 \\
 \frac{3-10 \lambda }{24 \lambda -6}  & \frac{5-18 \lambda }{6-24 \lambda }  &\scriptstyle  0  &\scriptstyle  0 \\
 \frac{-24 \lambda ^2+6 \lambda +1}{6-24 \lambda } & \frac{-48 \lambda ^2+12 \lambda +1}{24 \lambda -6}  & {\scriptstyle \lambda}   &\scriptstyle  0 \\
 \frac{3-16 \lambda }{12-48 \lambda }  & \frac{48 \lambda ^2-21 \lambda +2}{12 \lambda -3}  & \frac{3-16\lambda}{4}  &\scriptstyle  0  \\
 -{\scriptstyle \lambda}   &\scriptstyle  0  &\scriptstyle  0  & {\scriptstyle \lambda} \\
\scriptstyle   0  &\scriptstyle  0  &\scriptstyle  0  &\scriptstyle   0 \\
\hline
\scriptstyle  0 & \frac{-6 \lambda ^3+14 \lambda ^2-7 \lambda +1}{12 \lambda ^4-48 \lambda ^3+60 \lambda ^2-28 \lambda
   +4} &\frac{3 \left(2 \lambda ^2-4 \lambda +1\right)^2}{4 \left(3 \lambda ^3-9 \lambda ^2+6 \lambda
   -1\right)} &\frac{2 \lambda ^2-4 \lambda +1}{2-2 \lambda }
\end{array}
\right],
\end{split}
\end{equation*}
\end{tiny}
with the same diagonal coefficient $\lambda$ as MRI-GARK-ESDIRK34a. 


The scalar slow stability region $\mathcal{S}^{\textsc{1d}}_{\infty,\alpha}$  \eqref{eqn:scalar-slow-stability-region}  for SDIRK33a is shown in Figure  \ref{fig:MRI-implicit-stability_1D}.

\ifreport
The solution \eqref{eqn:MIS-additive-internal-ode} is computed as follows:
\begin{eqnarray*}
&& \Ys_{1} = y_{n}, \\
&& \begin{cases}
v(0) = \Ys_{1}, \\
 v' = (c_2-c_1) \, \funf \left( v \right) + \gamma^0_{1,1} \, \funs\bigl( Y^{\{\s\}}_1 \bigr), \quad \theta \in  [0,  H],  \\
\Ys_{2} = v(H), 
\end{cases} \\
&& \Ys_{3} = \Ys_{2} + \sum_{j=1}^3  \widetilde{\gamma}_{2,j}\, \funs\bigl( Y^{\{\s\}}_j \bigr) \\
&&\qquad = \Ys_{2} - \lambda\, \funs\bigl( Y^{\{\s\}}_1 \bigr) + \lambda\, \funs\bigl( Y^{\{\s\}}_3 \bigr)  \\
&& \begin{cases}
v(0) = \Ys_{3}, \\
 v' = (c_4-c_3) \, \funf \left( v \right) + \gamma^0_{3,1} \, \funs\bigl( Y^{\{\s\}}_1 \bigr) + \gamma^0_{3,3} \, \funs\bigl( Y^{\{\s\}}_3 \bigr), \quad \theta \in  [0,  H],  \\
\Ys_{4} = v(H), 
\end{cases} \\
&& \Ys_{5} = \Ys_{4} + \sum_{j=1}^5  \widetilde{\gamma}_{4,j}\, \funs\bigl( Y^{\{\s\}}_j \bigr)  \\
&&\qquad = \Ys_{4}  -\lambda\, \funs\bigl( Y^{\{\s\}}_3 \bigr) + \lambda\, \funs\bigl( Y^{\{\s\}}_5 \bigr) \\
&& \begin{cases}
v(0) = \Ys_{5}, \\
 v' = (c_6-c_5) \, \funf \left( v \right) + \gamma^0_{5,1} \, \funs\bigl( Y^{\{\s\}}_1 \bigr) + \gamma^0_{5,3} \, \funs\bigl( Y^{\{\s\}}_3 \bigr)+ \gamma^0_{5,5} \, \funs\bigl( Y^{\{\s\}}_5 \bigr), \quad \theta \in  [0,  H],  \\
\Ys_{6} = v(H), 
\end{cases}\\
&& \Ys_{7} = \Ys_{6} + \sum_{j=1}^7 \widetilde{\gamma}_{6,j}\, \funs\bigl( Y^{\{\s\}}_j \bigr)  \\
&&\qquad =  \Ys_{6} + \gamma^0_{6,3}\, \funs\bigl( Y^{\{\s\}}_3 \bigr)
+ \gamma^0_{6,5}\, \funs\bigl( Y^{\{\s\}}_5 \bigr)
+ \lambda\,\funs\bigl( Y^{\{\s\}}_7 \bigr) \\
&& y_{n+1} = \Ys_{7}.
\end{eqnarray*}
\fi

\subsection{An ESDIRK order 4 method, stiffly accurate, with equidistant abscissae}
%
%
The scheme MRI-GARK-ESDIRK46a is defined by the following coefficients:
\begin{tiny}
\[
\renewcommand{\arraystretch}{2}
\mathsf{A}^{\{\s\}} = \left[
\begin{array}{c|cccccc}
\scriptstyle 0 &\scriptstyle 0 &\scriptstyle 0 &\scriptstyle 0 &\scriptstyle 0 &\scriptstyle 0 &\scriptstyle 0 \\
\frac{1}{5}& \frac{1}{5} &\scriptstyle 0 &\scriptstyle 0 &\scriptstyle 0 &\scriptstyle 0 &\scriptstyle 0 \\
\frac{1}{5}& -\frac{1}{20} & \frac{1}{4} &\scriptstyle 0 &\scriptstyle 0 &\scriptstyle 0 &\scriptstyle 0 \\
\frac{2}{5}&\scriptstyle 0 & \frac{2}{5} &\scriptstyle 0 &\scriptstyle 0 &\scriptstyle 0 &\scriptstyle 0 \\
\frac{2}{5}& -\frac{103}{380} & \frac{8}{19} & \frac{1}{4} &\scriptstyle 0 &\scriptstyle 0 &\scriptstyle 0 \\
\frac{3}{5} &\scriptstyle 0 &\scriptstyle 0 & \frac{3}{5} &\scriptstyle 0 &\scriptstyle 0 &\scriptstyle 0 \\
\frac{3}{5} &  \frac{202381}{316160} & \frac{2199}{31616} & -\frac{1197}{3328} & \frac{1}{4} &\scriptstyle 0 &\scriptstyle 0 \\
\frac{4}{5} &   0 &\scriptstyle 0 &\scriptstyle 0 & \frac{4}{5} &\scriptstyle 0 &\scriptstyle 0 \\
\frac{4}{5} &   \frac{1978577}{3575040} & \frac{20417}{119168} & -\frac{3579}{12544} & \frac{65}{588} & \frac{1}{4} &\scriptstyle 0 \\
\scriptstyle 1 &\scriptstyle 0 &\scriptstyle 0 &\scriptstyle 0 &\scriptstyle 0 & 1 &\scriptstyle 0 \\
\scriptstyle 1 & \frac{1}{4} & -\frac{7}{24} & \frac{13}{24} & \frac{13}{24} & -\frac{7}{24} & \frac{1}{4} \\
 \hline
\scriptstyle 1 & \frac{1}{4} & -\frac{7}{24} & \frac{13}{24} & \frac{13}{24} & -\frac{7}{24} & \frac{1}{4} \\
 \hline
\scriptstyle 1 &\scriptstyle 0 & \frac{18163}{52824} & \frac{13943}{52824} & \frac{3263}{52824} & \frac{11053}{52824} & \frac{1067}{8804}
\end{array}
\right],
\]
\end{tiny}
the matrix $\Gamma^0$ is:
\begin{tiny}
\[
\renewcommand{\arraystretch}{2}
\left[
\begin{array}{cccccc}
 \frac{1}{5} &\scriptstyle 0 &\scriptstyle 0 &\scriptstyle 0 &\scriptstyle 0 &\scriptstyle 0 \\
 -\frac{1}{4} & \frac{1}{4} &\scriptstyle 0 &\scriptstyle 0 &\scriptstyle 0 &\scriptstyle 0 \\
 \frac{1771023115159}{1929363690800} & -\frac{1385150376999}{1929363690800} &\scriptstyle 0 &\scriptstyle 0 &\scriptstyle 0 &\scriptstyle 0 \\
 \frac{914009}{345800} & -\frac{1000459}{345800} & \frac{1}{4} &\scriptstyle 0 &\scriptstyle 0 &\scriptstyle 0 \\
 \frac{18386293581909}{36657910125200} & \frac{5506531089}{80566835440} & -\frac{178423463189}{482340922700} &\scriptstyle 0 &\scriptstyle 0 &\scriptstyle 0 \\
 \frac{36036097}{8299200} & \frac{4621}{118560} & -\frac{38434367}{8299200} & \frac{1}{4} &\scriptstyle 0 &\scriptstyle 0 \\
 -\frac{247809665162987}{146631640500800} & \frac{10604946373579}{14663164050080} & \frac{10838126175385}{5865265620032} &
   -\frac{24966656214317}{36657910125200} &\scriptstyle 0 &\scriptstyle 0 \\
 \frac{38519701}{11618880} & \frac{10517363}{9682400} & -\frac{23284701}{19364800} & -\frac{10018609}{2904720} & \frac{1}{4} &\scriptstyle 0 \\
 -\frac{52907807977903}{33838070884800} & \frac{74846944529257}{73315820250400} & \frac{365022522318171}{146631640500800} &
   -\frac{20513210406809}{109973730375600} & -\frac{2918009798}{1870301537} &\scriptstyle 0 \\
 \frac{19}{100} & -\frac{73}{300} & \frac{127}{300} & \frac{127}{300} & -\frac{313}{300} & \frac{1}{4} \\
\scriptstyle  0 &\scriptstyle 0 &\scriptstyle 0 &\scriptstyle 0 &\scriptstyle 0 &\scriptstyle 0 \\
 \hline
 -\frac{1}{4} & \frac{5595}{8804} & -\frac{2445}{8804} & -\frac{4225}{8804} & \frac{2205}{4402} & -\frac{567}{4402}
\end{array}
\right],
\]
\end{tiny}
and the matrix $\Gamma^1$ is:
\begin{tiny}
\[
\renewcommand{\arraystretch}{2}
\left[
\begin{array}{cccccc}
\scriptstyle  0 &\scriptstyle 0 &\scriptstyle 0 &\scriptstyle 0 &\scriptstyle 0 &\scriptstyle 0 \\
\scriptstyle  0 &\scriptstyle 0 &\scriptstyle 0 &\scriptstyle 0 &\scriptstyle 0 &\scriptstyle 0 \\
 -\frac{1674554930619}{964681845400} & \frac{1674554930619}{964681845400} &\scriptstyle 0 &\scriptstyle 0 &\scriptstyle 0 &\scriptstyle 0 \\
 -\frac{1007739}{172900} & \frac{1007739}{172900} &\scriptstyle 0 &\scriptstyle 0 &\scriptstyle 0 &\scriptstyle 0 \\
 -\frac{8450070574289}{18328955062600} & -\frac{39429409169}{40283417720} & \frac{173621393067}{120585230675} &\scriptstyle 0 &\scriptstyle 0 &\scriptstyle 0 \\
 -\frac{122894383}{16598400} & \frac{14501}{237120} & \frac{121879313}{16598400} &\scriptstyle 0 &\scriptstyle 0 &\scriptstyle 0 \\
 \frac{32410002731287}{15434909526400} & -\frac{46499276605921}{29326328100160} & -\frac{34914135774643}{11730531240064} &
   \frac{45128506783177}{18328955062600} &\scriptstyle 0 &\scriptstyle 0 \\
 -\frac{128357303}{23237760} & -\frac{35433927}{19364800} & \frac{71038479}{38729600} & \frac{8015933}{1452360} &\scriptstyle 0 &\scriptstyle 0 \\
 \frac{136721604296777}{67676141769600} & -\frac{349632444539303}{146631640500800} & -\frac{1292744859249609}{293263281001600} &
   \frac{8356250416309}{54986865187800} & \frac{17282943803}{3740603074} &\scriptstyle 0 \\
 \frac{3}{25} & -\frac{29}{300} & \frac{71}{300} & \frac{71}{300} & -\frac{149}{300} &\scriptstyle 0 \\
\scriptstyle  0 &\scriptstyle 0 &\scriptstyle 0 &\scriptstyle 0 &\scriptstyle 0 &\scriptstyle 0 \\
 \hline
\scriptstyle   0 &\scriptstyle 0 &\scriptstyle 0 &\scriptstyle 0 &\scriptstyle 0 &\scriptstyle  0 
\end{array}
\right].
\]
\end{tiny}
The scalar slow stability regions $\mathcal{S}^{\textsc{1d}}_{\infty,\alpha}$ and $\mathcal{S}^{\textsc{1d}}_{1,\alpha}$  \eqref{eqn:scalar-slow-stability-region}  are shown in Figure  \ref{fig:MRI-implicit-stability_1D}. We see that the method is unconditionally stable in the slow variable and conditionally stable in the fast variable, or vice-versa. Therefore, the method is useful for systems combining a stiff slow component with a non-stiff fast component.
The matrix stability region $\mathcal{S}^{\textsc{2d}}_{\rho=10,\alpha}$ \eqref{eqn:matrix-slow-stability-region}  is shown in Figure \ref{fig:MRI-implicit-stability_2D}.

\ifreport
\subsection{Another ESDIRK order 4 method, stiffly accurate, with equidistant abscissae}
%
%
The scheme MRI-GARK-ESDIRK46b is defined by the following coefficients:
\begin{tiny}
\[
\renewcommand{\arraystretch}{2}
\mathsf{A}^{\{\s\}} = \left[
\begin{array}{c | cccccc}
\scriptstyle 0 &\scriptstyle 0 &\scriptstyle 0  &\scriptstyle 0 &\scriptstyle 0 &\scriptstyle 0  &\scriptstyle 0 \\
\frac{1}{5} & \frac{1}{5}  &\scriptstyle 0  &\scriptstyle 0 &\scriptstyle 0 &\scriptstyle 0&\scriptstyle 0  \\
\frac{1}{5} & {\scriptstyle -}\frac{1}{20}  & \frac{1}{4}  &\scriptstyle 0 &\scriptstyle 0 &\scriptstyle 0 &\scriptstyle 0  \\
\frac{2}{5} & \frac{29143073}{5141400}  & {\scriptstyle -}\frac{27086513}{5141400}  &\scriptstyle 0 &\scriptstyle 0  &\scriptstyle 0  &\scriptstyle 0 \\
\frac{2}{5}& \frac{43}{820}  & \frac{4}{41}  & \frac{1}{4}  &\scriptstyle 0 &\scriptstyle 0 &\scriptstyle 0  \\
\frac{3}{5}& \frac{89695}{24928}  & {\scriptstyle -}\frac{4417771}{541200}  & \frac{106216283}{20565600}  &\scriptstyle 0 &\scriptstyle 0 &\scriptstyle 0  \\
\frac{3}{5}& {\scriptstyle -}\frac{47603}{78720}  & \frac{13399}{7872}  & {\scriptstyle -}\frac{287}{384}  & \frac{1}{4} &\scriptstyle 0 &\scriptstyle 0  \\
\frac{4}{5}&\scriptstyle 0  &\scriptstyle 0  & {\scriptstyle -}\frac{2882749}{1028280}  & \frac{3705373}{1028280} &\scriptstyle 0 &\scriptstyle 0  \\
\frac{4}{5}& {\scriptstyle -}\frac{998863}{551040}  & \frac{223571}{55104}  & {\scriptstyle -}\frac{4267}{2688}  & {\scriptstyle -}\frac{3}{28}  & \frac{1}{4}  &\scriptstyle 0 \\
\scriptstyle 1&\scriptstyle 0  &\scriptstyle 0  &\scriptstyle 0  & \frac{21877409}{5141400}  & {\scriptstyle -}\frac{16736009}{5141400}  &\scriptstyle 0 \\
\scriptstyle 1& {\scriptstyle -}\frac{1}{12}  & \frac{25}{24}  & {\scriptstyle -}\frac{35}{24}  & \frac{15}{8}  & {\scriptstyle -}\frac{5}{8}  & \frac{1}{4} \\
 \hline
\scriptstyle 1& {\scriptstyle -}\frac{1}{12} & \frac{25}{24}  & {\scriptstyle -}\frac{35}{24}  & \frac{15}{8}  & {\scriptstyle -}\frac{5}{8}  & \frac{1}{4} \\
\hline
\scriptstyle 1& \frac{9971}{21400} & {\scriptstyle -}\frac{9}{20}  & {\scriptstyle -}\frac{1039}{6420}  & \frac{9263}{6420}  & {\scriptstyle -}\frac{5519}{12840}  & \frac{711}{5350}
\end{array}
\right],
\]
\end{tiny}
\begin{tiny}
\[
\renewcommand{\arraystretch}{2}
\Gamma^0 =
\left[
\begin{array}{ccccccccccc}
 \frac{1}{5}  &\scriptstyle 0  &\scriptstyle 0  &\scriptstyle 0  &\scriptstyle 0  &\scriptstyle 0 \\
 {\scriptstyle -}\frac{1}{4}  & \frac{1}{4}  &\scriptstyle 0  &\scriptstyle 0  &\scriptstyle 0 &\scriptstyle 0 \\
 {\scriptstyle -}\frac{18736957}{5141400}  & \frac{19765237}{5141400} 0 &\scriptstyle 0  &\scriptstyle 0  &\scriptstyle 0 &\scriptstyle 0  \\
 {\scriptstyle -}\frac{28873463}{5141400}  & \frac{27588113}{5141400}  & \frac{1}{4} &\scriptstyle 0  &\scriptstyle 0  &\scriptstyle 0  \\
 \frac{10779}{3040}  & {\scriptstyle -}\frac{4470571}{541200}  & \frac{101074883}{20565600} &\scriptstyle 0  &\scriptstyle 0  &\scriptstyle 0 \\
 {\scriptstyle -}\frac{6286157}{1495680}  & \frac{7118603}{721600}  & {\scriptstyle -}\frac{486347707}{82262400}  & \frac{1}{4}  &\scriptstyle 0  &\scriptstyle 0 \\
 \frac{47603}{78720}  & {\scriptstyle -}\frac{13399}{7872}  & {\scriptstyle -}\frac{11275823}{5484160}  & \frac{3448303}{1028280}  &\scriptstyle 0  &\scriptstyle 0 \\
 {\scriptstyle -}\frac{998863}{551040}  & \frac{223571}{55104}  & \frac{140048273}{115167360}  & {\scriptstyle -}\frac{26708821}{7197960}  & \frac{1}{4}  &\scriptstyle 0 \\
 \frac{998863}{551040}  & {\scriptstyle -}\frac{223571}{55104}  & \frac{4267}{2688}  & \frac{156997913}{35989800}  & {\scriptstyle -}\frac{18021359}{5141400}  &\scriptstyle 0 \\
 {\scriptstyle -}\frac{1}{12}  & \frac{25}{24}  & {\scriptstyle -}\frac{35}{24}  & {\scriptstyle -}\frac{3059321}{1285350}  & \frac{6761317}{2570700}  & \frac{1}{4} \\
\scriptstyle 0 &\scriptstyle 0 &\scriptstyle 0 &\scriptstyle 0 &\scriptstyle 0 &\scriptstyle 0 \\
 \hline
 \frac{35263}{64200} & {\scriptstyle -}\frac{179}{120}  & \frac{5549}{4280}  & {\scriptstyle -}\frac{5549}{12840}  & \frac{1253}{6420}  & {\scriptstyle -}\frac{1253}{10700} 
\end{array}
\right],
\]
\end{tiny}
and:
\[
{\scriptstyle \gamma^1_{3,1} = \frac{14587}{779}}, \quad {\scriptstyle \gamma^1_{3,3} = -\frac{14587}{779}}.
\]
\fi

\begin{figure}[h]
\centering
\subfigure[IRK21a: $\mathcal{S}^{\textsc{1d}}_{\rho=\infty,\alpha}$]{
\includegraphics[width=0.3\linewidth]{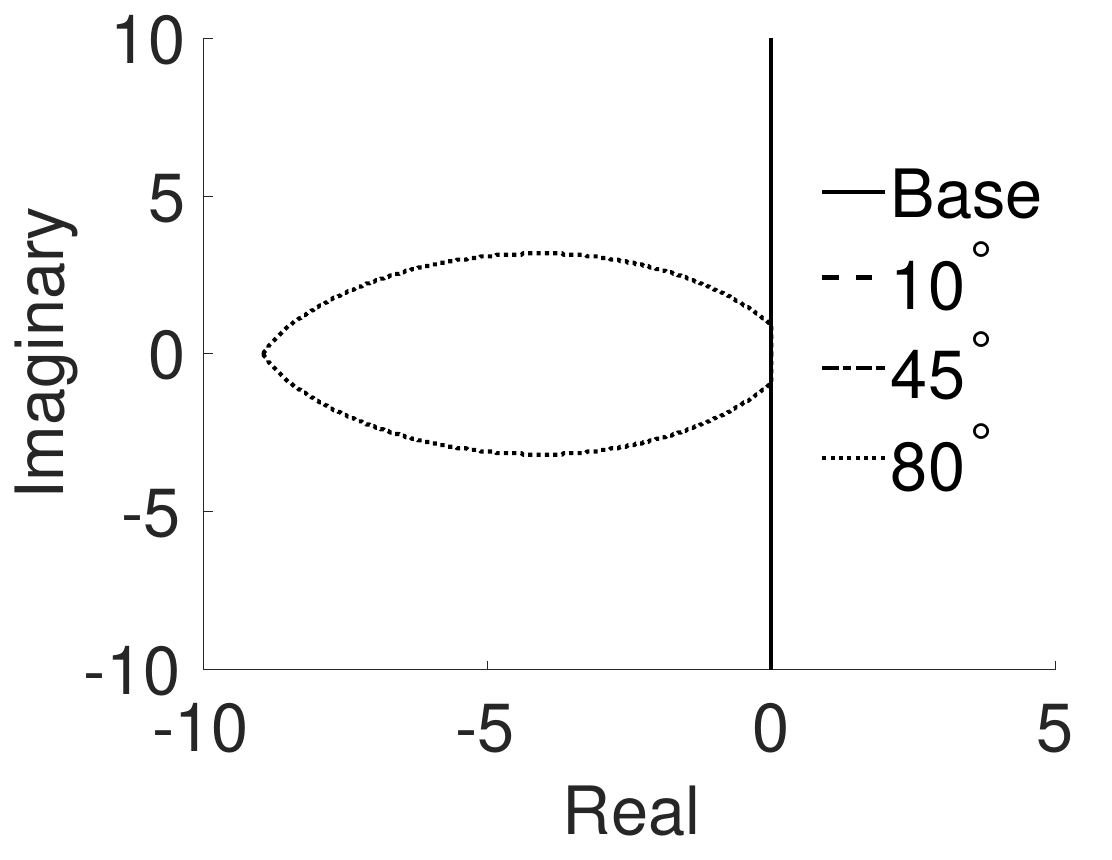}
\label{fig:stab1D-IRK21a}
}
\subfigure[ESDIRK34a: $\mathcal{S}^{\textsc{1d}}_{\rho=\infty,\alpha}$]{
\includegraphics[width=0.3\linewidth]{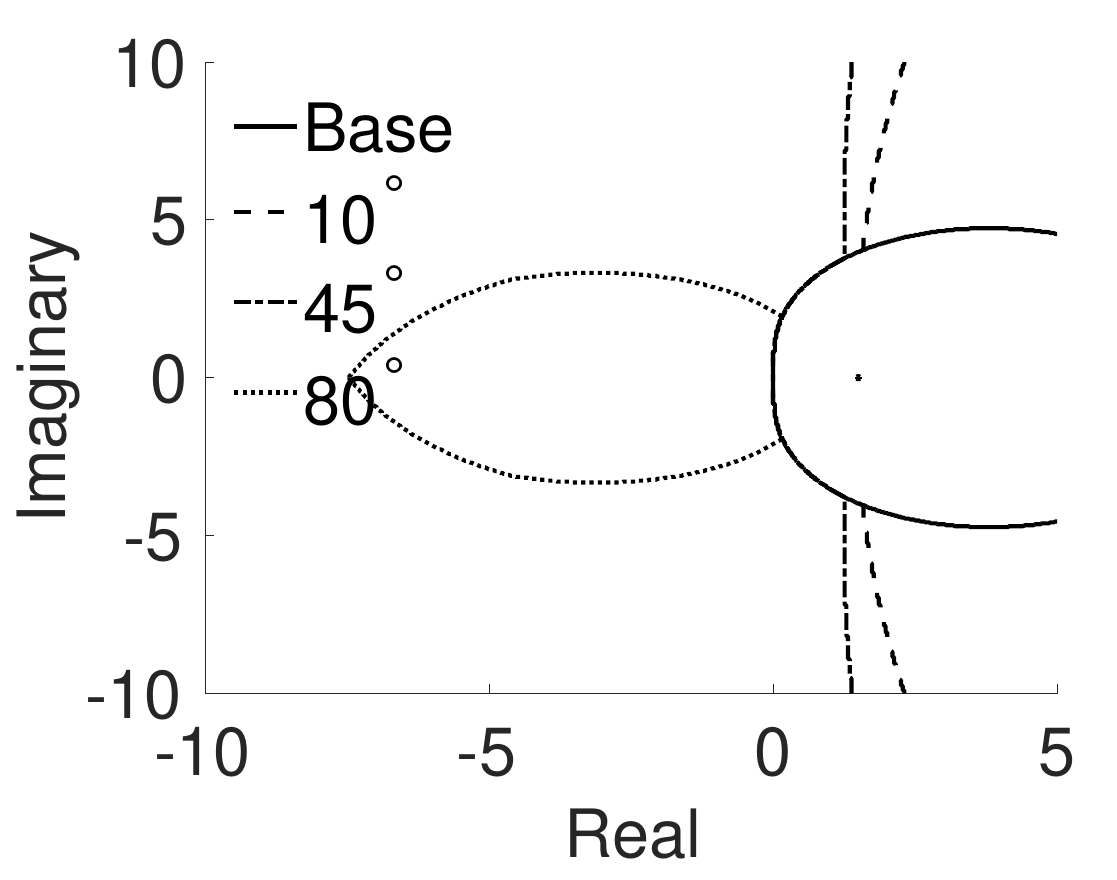}
\label{fig:stab1D-ESDIRK34a}
}
\subfigure[SDIRK33a: $\mathcal{S}^{\textsc{1d}}_{\rho=\infty,\alpha}$]{
\includegraphics[width=0.3\linewidth]{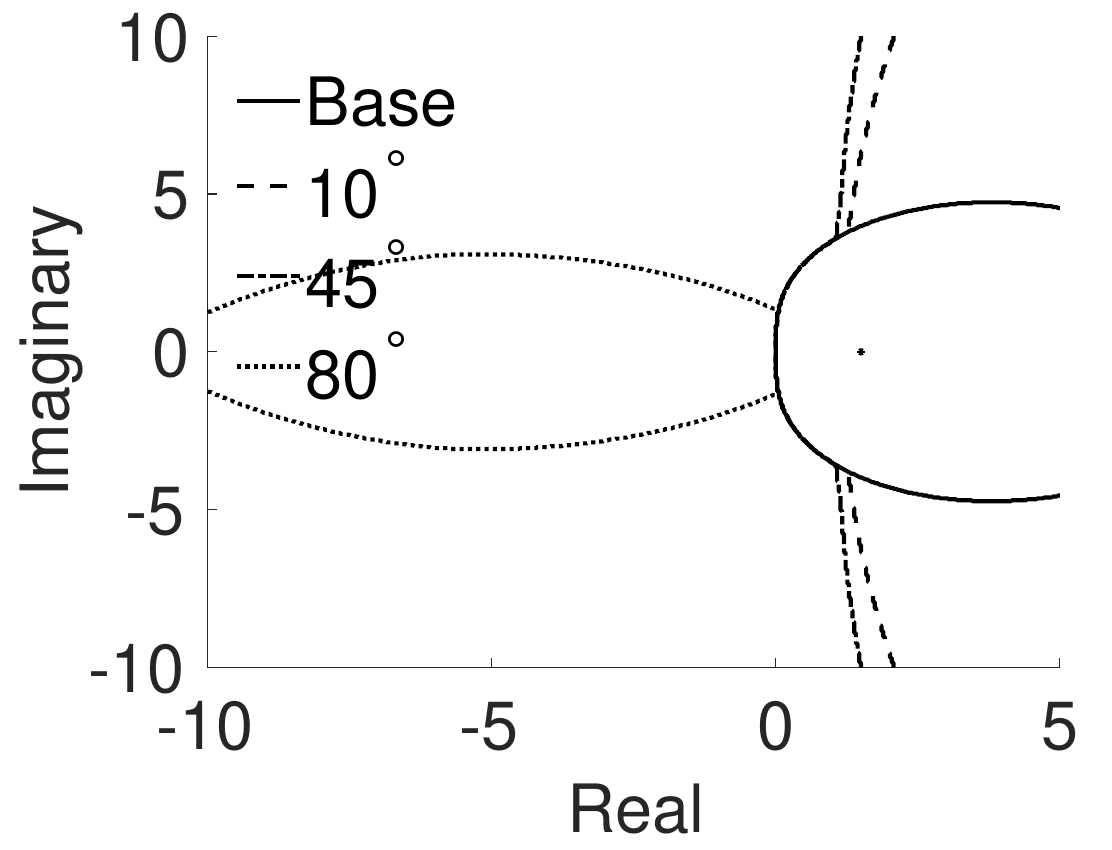}
\label{fig:stab1D-SDIRK33a}
}
\subfigure[ESDIRK46a: $\mathcal{S}^{\textsc{1d}}_{\rho=\infty,\alpha}$]{
\includegraphics[width=0.3\linewidth]{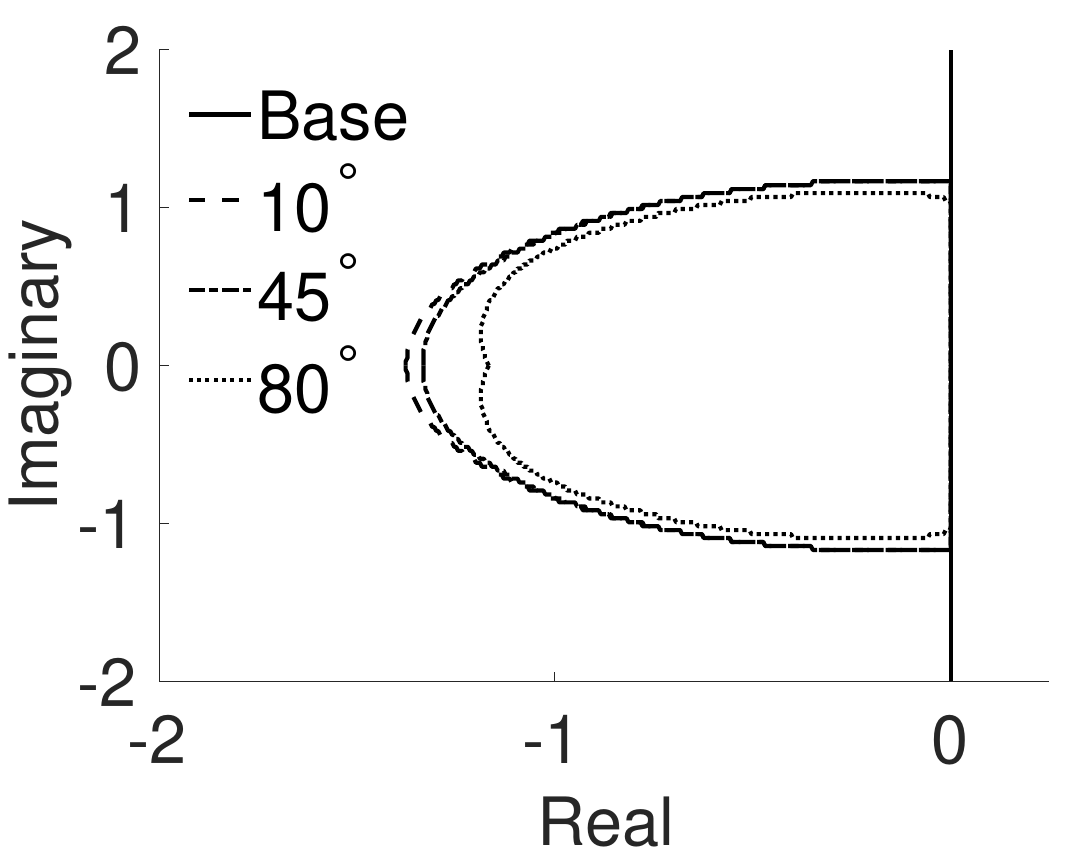}
\label{fig:stab1D-ESDIRK46a-rhoinf}
}
\subfigure[ESDIRK46a: $\mathcal{S}^{\textsc{1d}}_{\rho=1,\alpha}$]{
\includegraphics[width=0.3\linewidth]{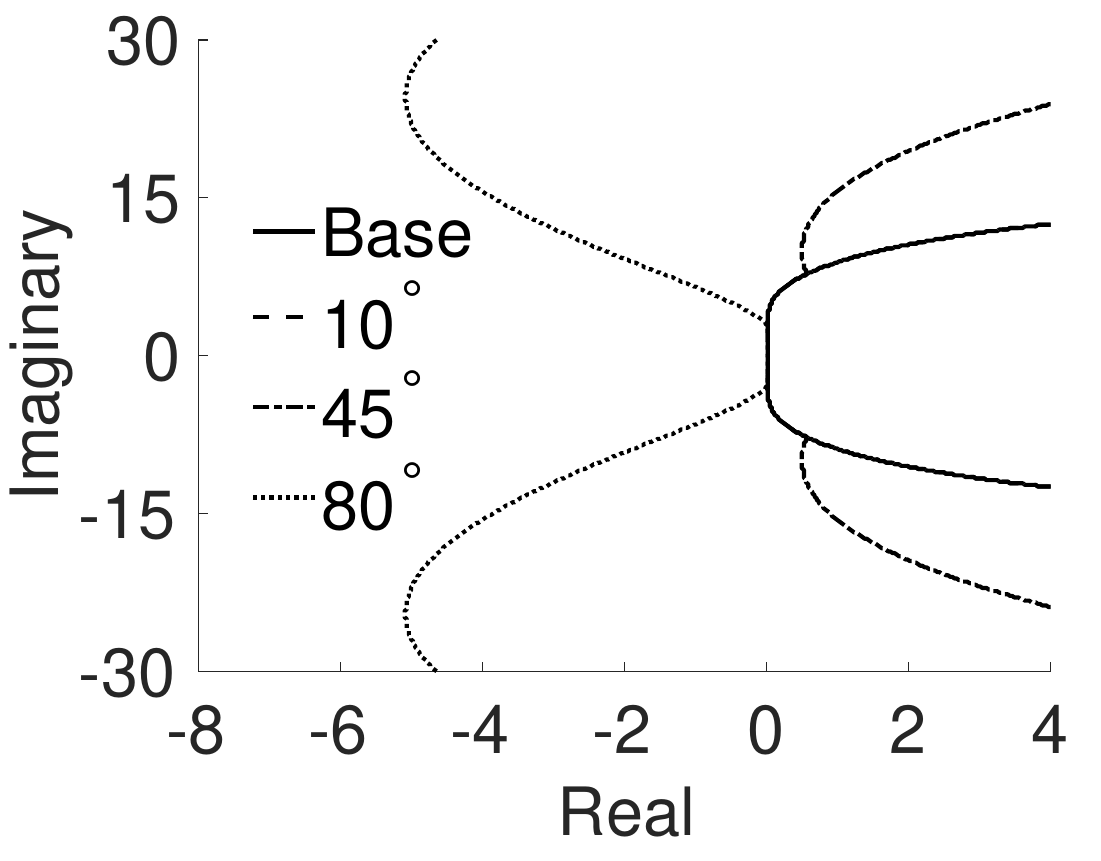}
\label{fig:stab1D-ESDIRK46a-rho1}
}

\caption{Scalar slow stability regions $\mathcal{S}^{\textsc{1d}}_{\infty,\alpha}$ of decoupled implicit MRI-GARK schemes. For $\rho=\infty$ in \eqref{eqn:scalar-slow-stability-region} they correspond to $A(\alpha)$ unconditional stability in the fast variable. The stability degrades with increasing $\alpha$, and for $\alpha = 80^\circ$ all MRI-GARK schemes lose the unconditional stability of the the slow base methods.} 
\label{fig:MRI-implicit-stability_1D}
\end{figure}

\begin{figure}[h]
\centering
\ifreport
\subfigure[IRK21a: $\mathcal{S}^{\textsc{2d}}_{\rho=10,\alpha=10^\circ}$]{
\includegraphics[width=0.3\linewidth]{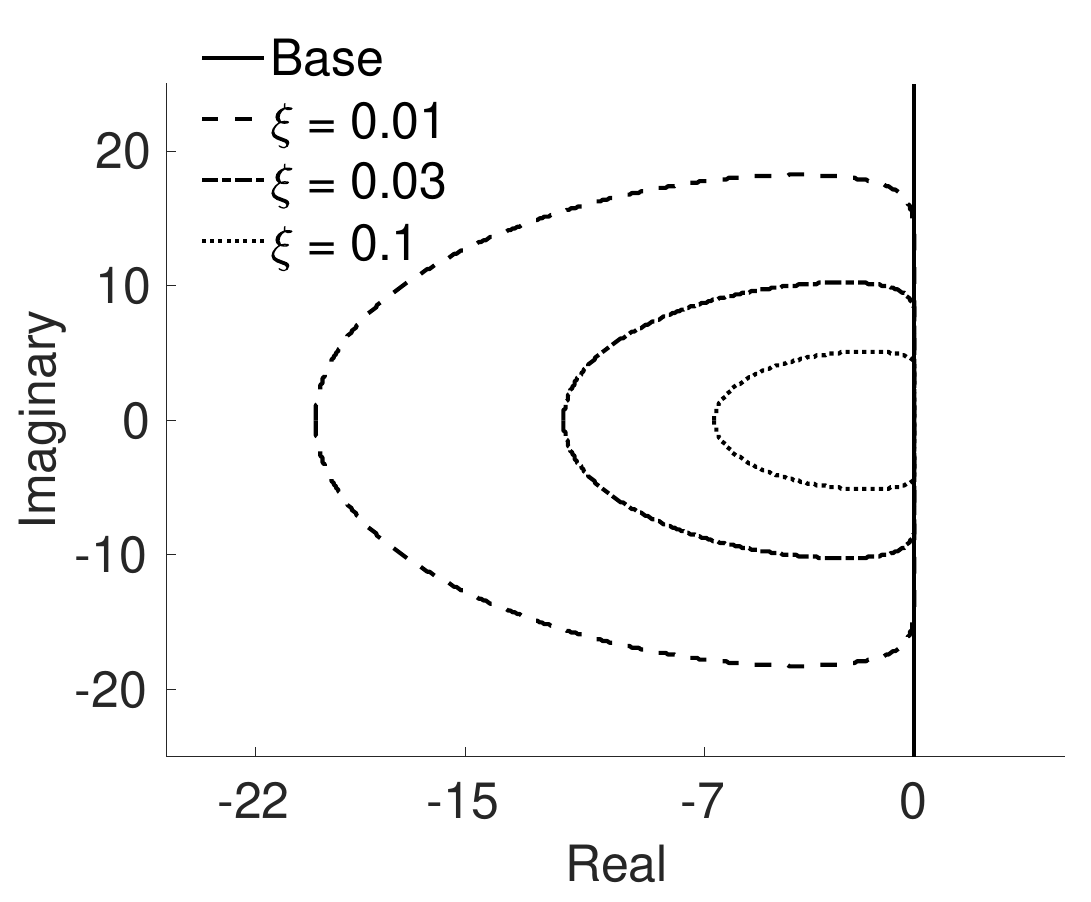}
\label{fig:stab2D-IRK21a-10}
}
\fi
\subfigure[IRK21a: $\mathcal{S}^{\textsc{2d}}_{\rho=10,\alpha=45^\circ}$]{
\includegraphics[width=0.3\linewidth]{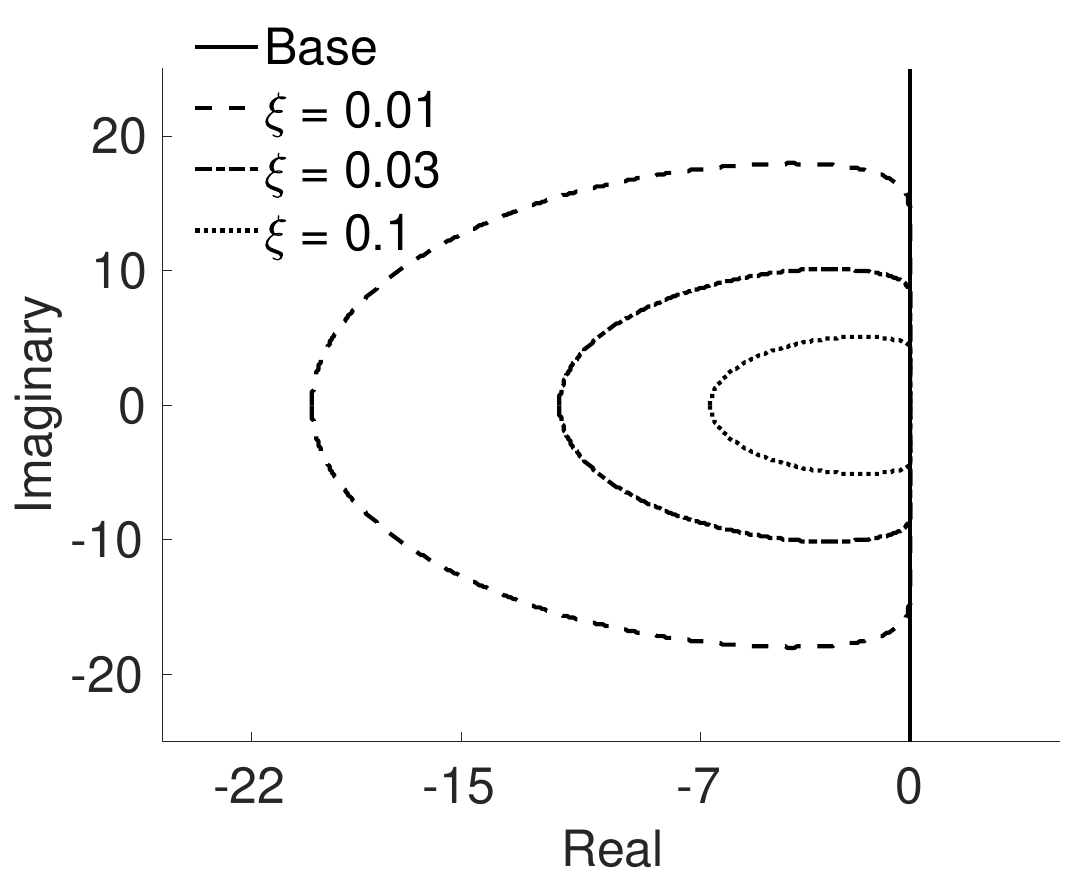}
\label{fig:stab2D-IRK21a-45}
}
\ifreport
\subfigure[IRK21a: $\mathcal{S}^{\textsc{2d}}_{\rho=10,\alpha=80^\circ}$]{
\includegraphics[width=0.3\linewidth]{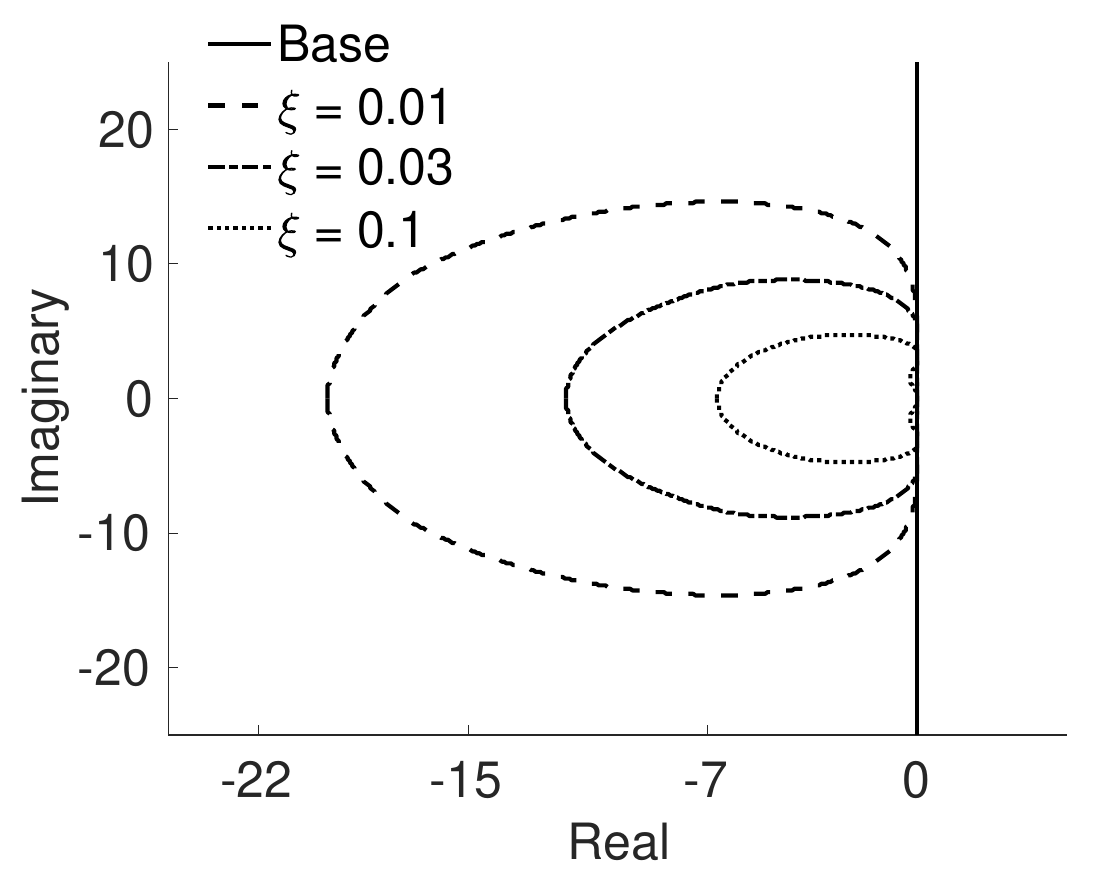}
\label{fig:stab2D-IRK21a-80}
}
\fi
\ifreport
\subfigure[ESDIRK34a: $\mathcal{S}^{\textsc{2d}}_{\rho=10,\alpha=10^\circ}$]{
\includegraphics[width=0.3\linewidth]{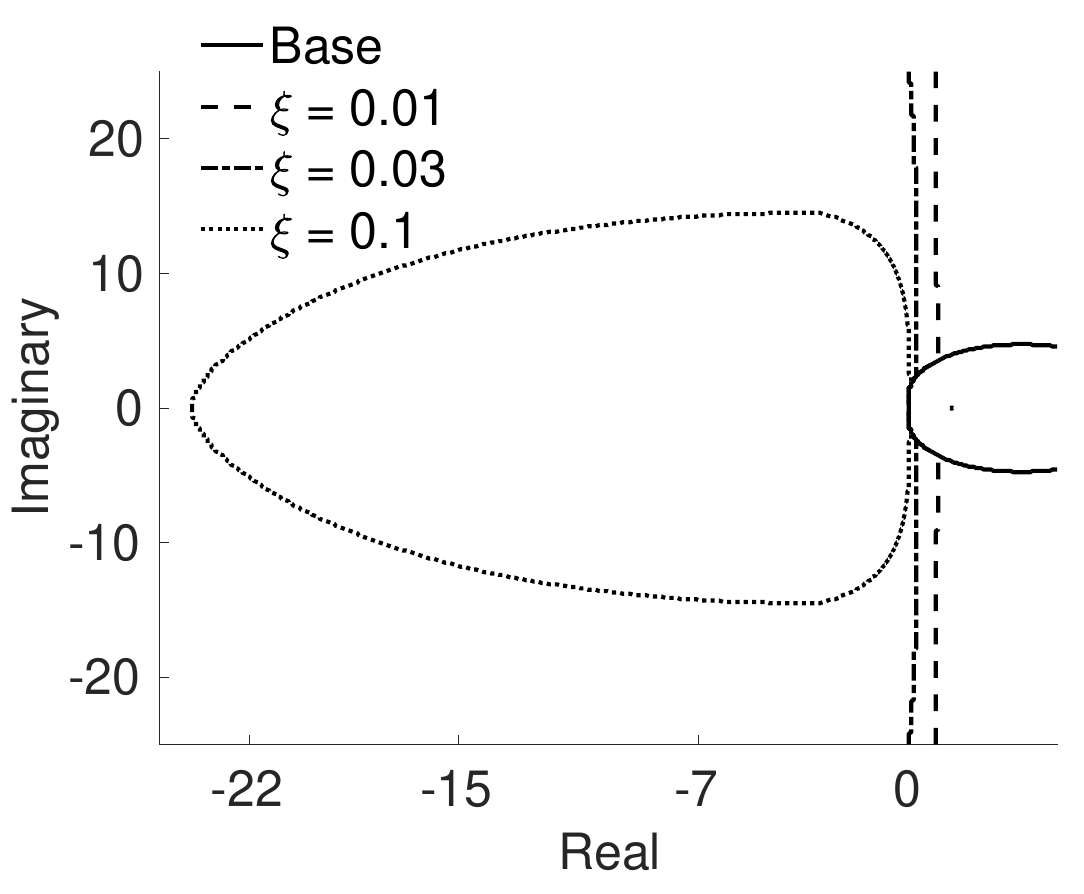}
\label{fig:stab2D-ESDIRK34a-10}
}
\fi
\subfigure[ESDIRK34a: $\mathcal{S}^{\textsc{2d}}_{\rho=10,\alpha=45^\circ}$]{
\includegraphics[width=0.3\linewidth]{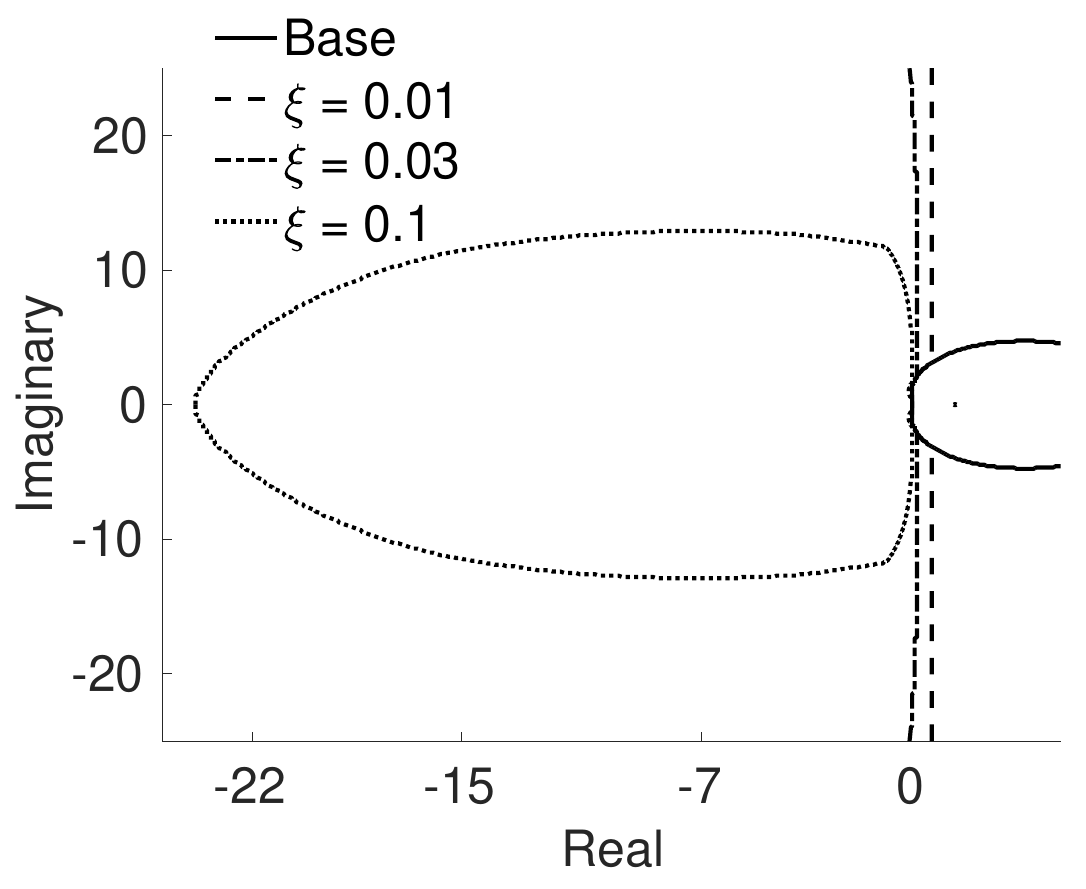}
\label{fig:stab2D-ESDIRK34a-45}
}
\ifreport
\subfigure[ESDIRK34a: $\mathcal{S}^{\textsc{2d}}_{\rho=10,\alpha=80^\circ}$]{
\includegraphics[width=0.3\linewidth]{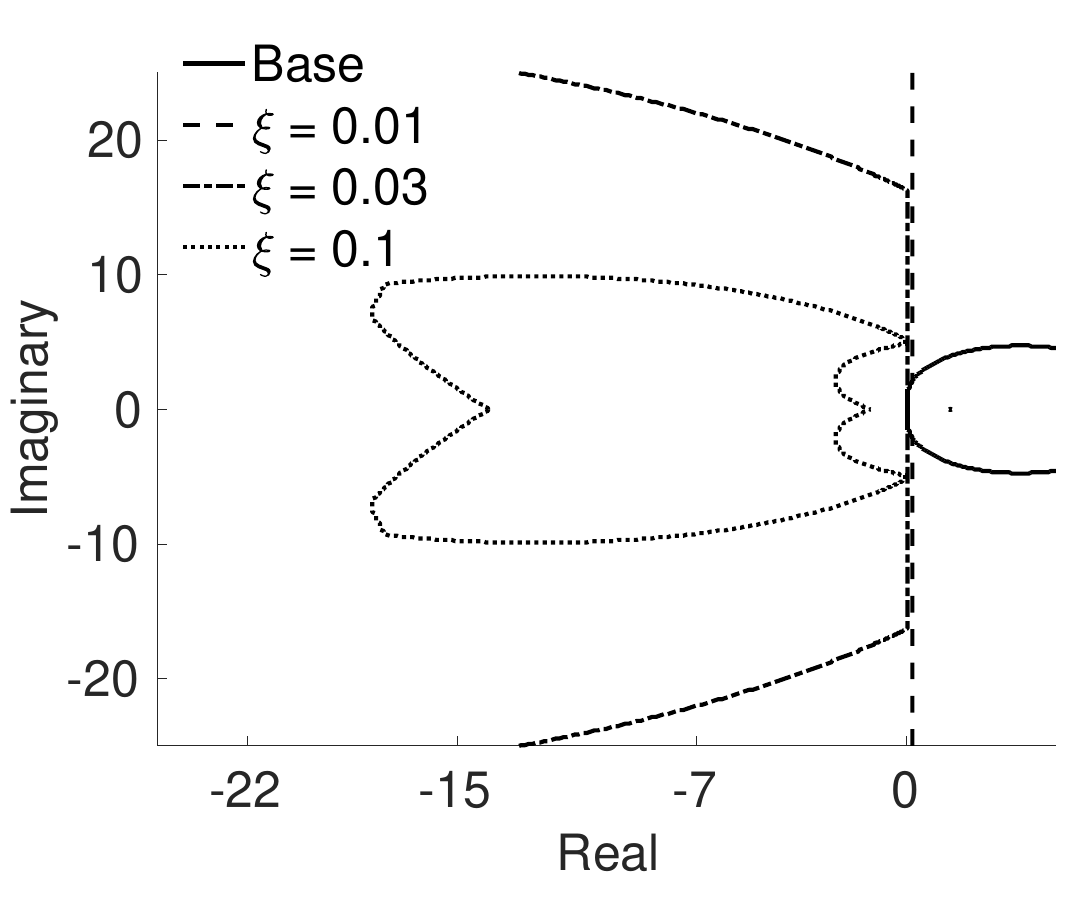}
\label{fig:stab2D-ESDIRK34a-80}
}
\fi
\ifreport
\subfigure[ESDIRK46a: $\mathcal{S}^{\textsc{2d}}_{\rho=10,\alpha=10^\circ}$]{
\includegraphics[width=0.3\linewidth]{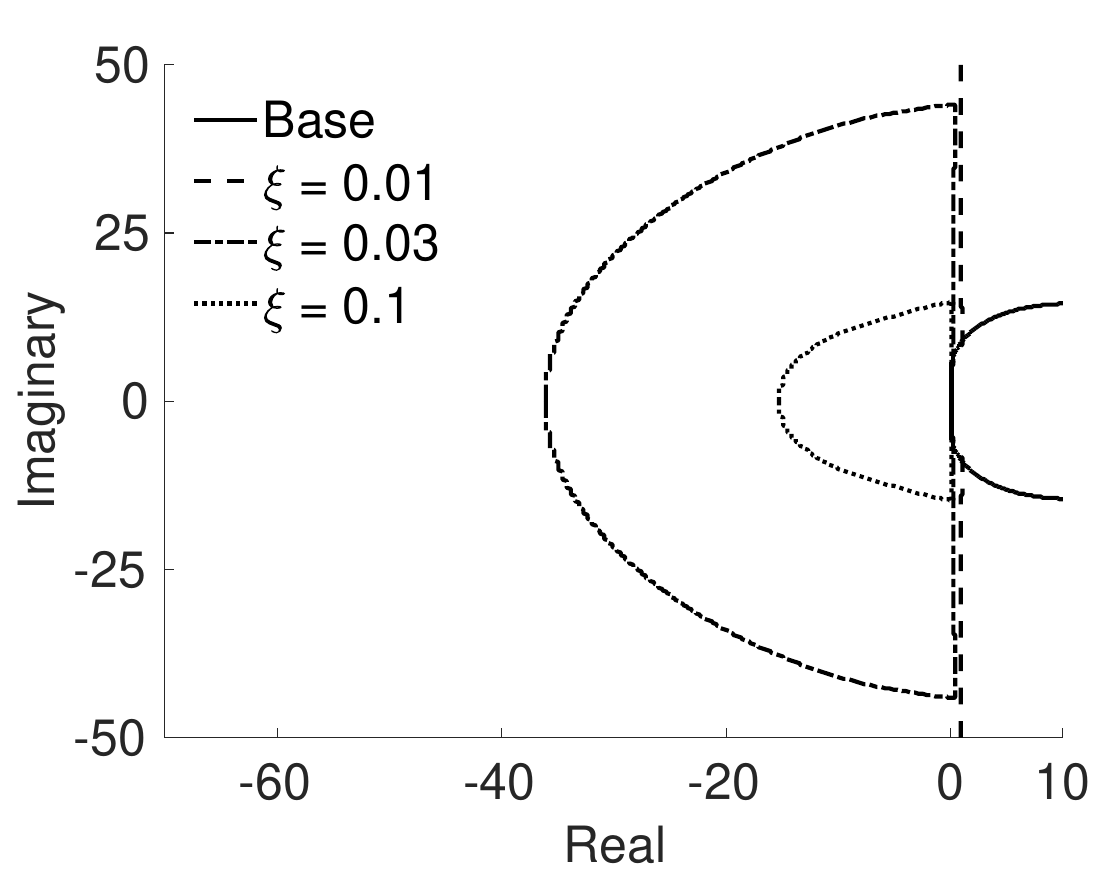}
\label{fig:stab2D-ESDIRK46a-10}
}
\fi
\subfigure[ESDIRK46a: $\mathcal{S}^{\textsc{2d}}_{\rho=10,\alpha=45^\circ}$]{
\includegraphics[width=0.3\linewidth]{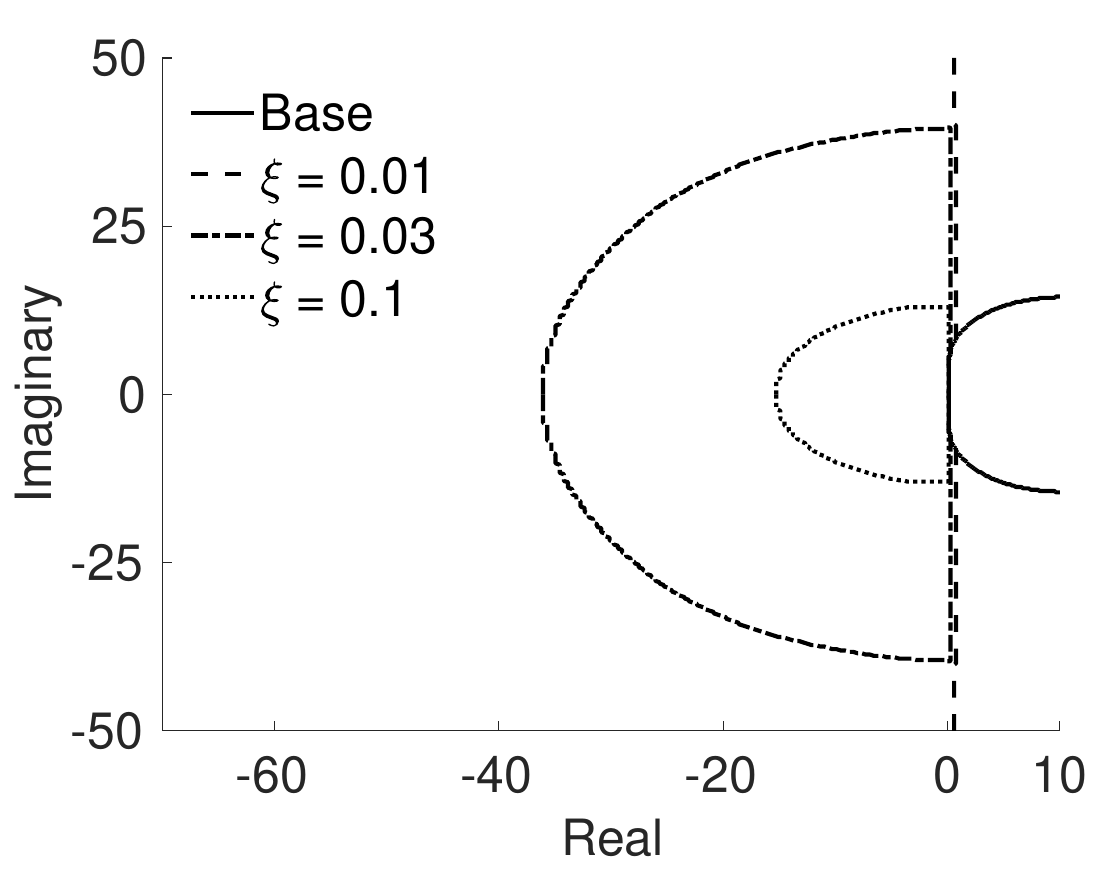}
\label{fig:stab2D-EESDIRK46a-45}
}
\ifreport
\subfigure[ESDIRK46a: $\mathcal{S}^{\textsc{2d}}_{\rho=10,\alpha=80^\circ}$]{
\includegraphics[width=0.3\linewidth]{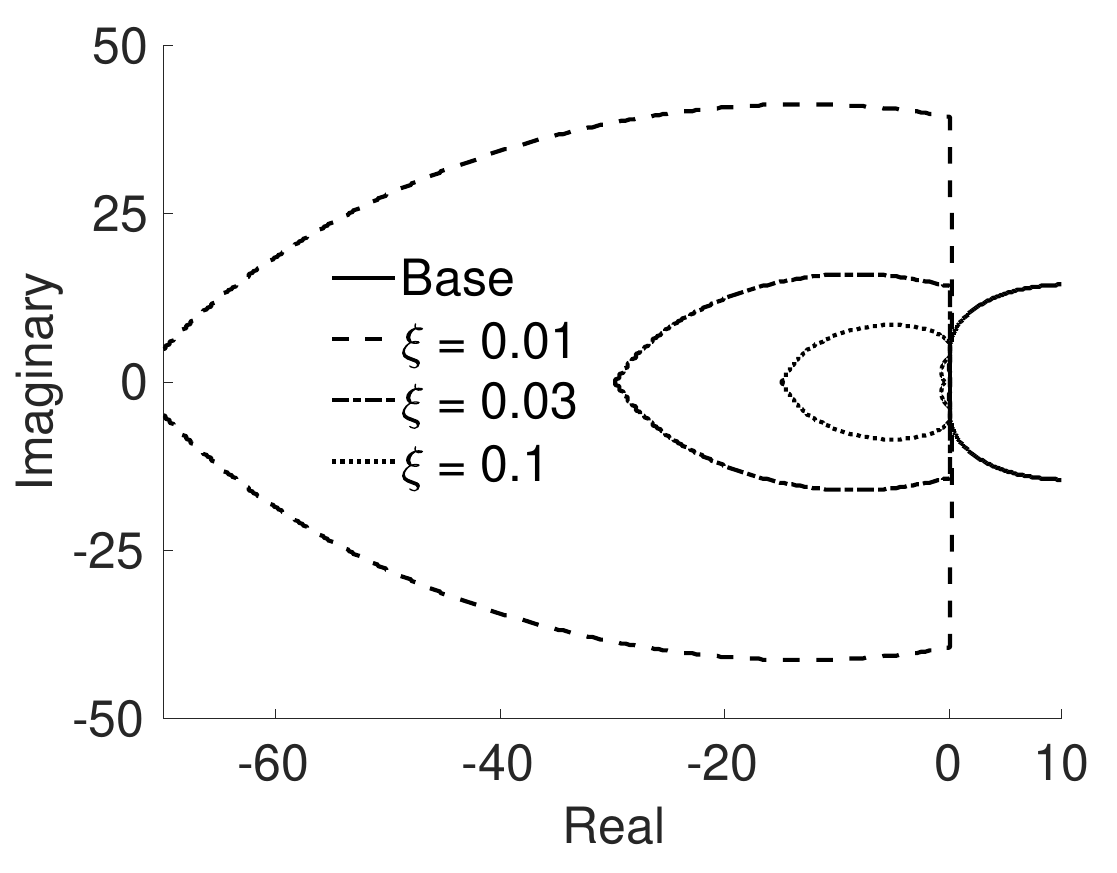}
\label{fig:stab2D-ESDIRK46a-80}
}
\fi
\caption{Matrix slow stability regions \eqref{eqn:matrix-slow-stability-region} of implicit MRI-GARK schemes. The $\xi=0$ case corresponds to the base slow method. The stability is conditional with respect to the fast variable, and degrades with increasing $\alpha$. The stability also degrades with an increasing influence $\xi$ of the fast system on the slow one.} 
\label{fig:MRI-implicit-stability_2D}
\end{figure}



\section{Numerical Results}
\label{sec:numerics}

\subsection{Additive partitioning: the Gray-Scott model}

We consider first the application of MRI-GARK methods to the  Gray-Scott reaction-diffusion PDE \cite{lee1993pattern}:
\begin{equation}
\label{eq:Gray-Scott}
\underbrace{
\begin{bmatrix} u \\ v \end{bmatrix}'
}_{y'}
=
\underbrace{
\begin{bmatrix}  \nabla \cdot ( \varepsilon_u\, \nabla u ) \\  \nabla \cdot ( \varepsilon_v\, \nabla v )  \end{bmatrix}
}_{\funs (y)}
+
\underbrace{
\begin{bmatrix}  -u\,v^2 + \mathfrak{f} (1-u) \\  u\,v^2 - (\mathfrak{f} + \mathfrak{k}) \end{bmatrix}
}_{\funf (y)}
\end{equation}
The spatial domain is the unit square discretized with second order finite differences. The model parameters are $\varepsilon_u = 0.0625$, $\varepsilon_v = 0.0312$,  $\mathfrak{k}= 0.0520$, and $\mathfrak{f}= 0.0180$. 
The system \eqref{eq:Gray-Scott} is written in the form \eqref{eqn:multirate-additive-ode} by additively splitting the right hand side into slow linear diffusion terms and fast nonlinear reaction terms. 
%
%

All numerical experiments use a simulation time interval of [0,30] time units. The fast stage integration \eqref{eqn:MIS-additive-internal-ode} is carried out using Matlab's ode45 function with tight tolerances abstol=reltol=1.e-10.
Convergence diagrams for the third and fourth order methods developed in Section \ref{sec:explicit-methods} are shown in Figure \ref{fig:MRI-GrayScott_convergence}. All methods achieve their theoretical orders of accuracy, confirming the order conditions theory developed herein.

\begin{figure}[h]
\centering
\ifreport
\subfigure[Explicit methods of order three]{
\includegraphics[width=0.375\linewidth]{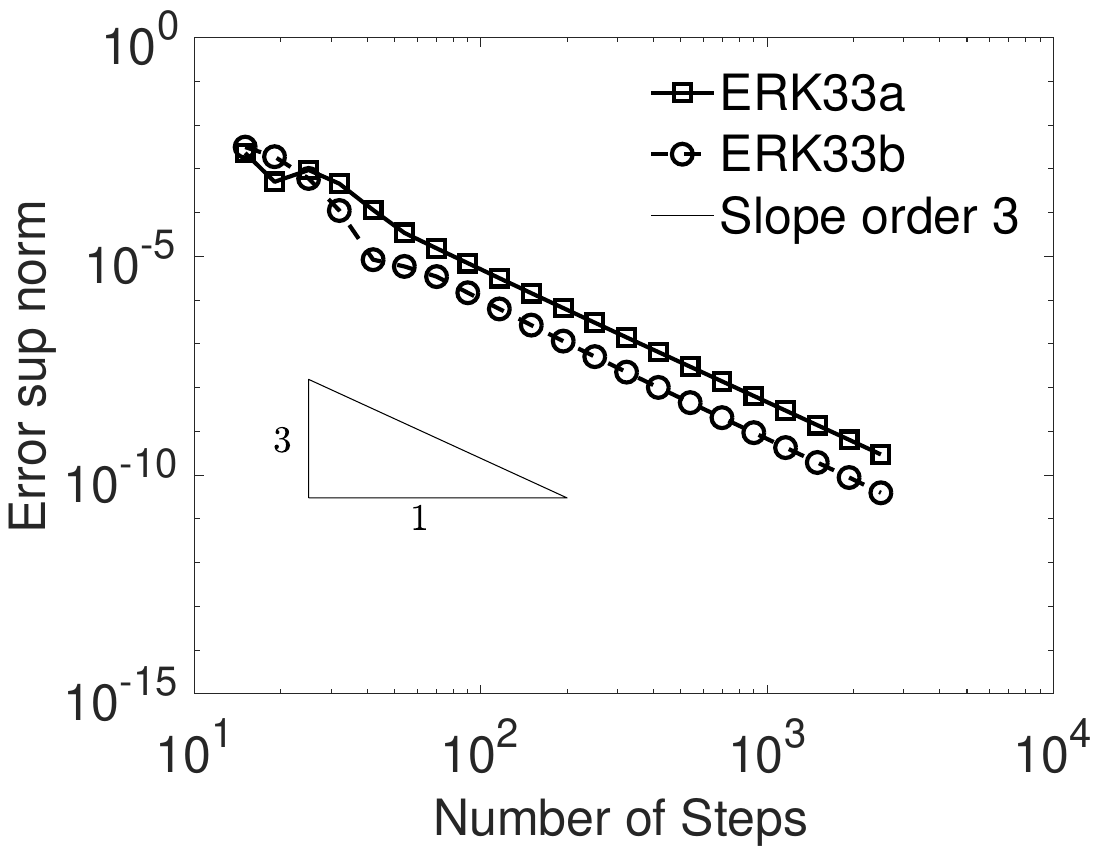}
\label{fig:GrayScott_ERK3}
}
\subfigure[Explicit methods of order four]{
\includegraphics[width=0.375\linewidth]{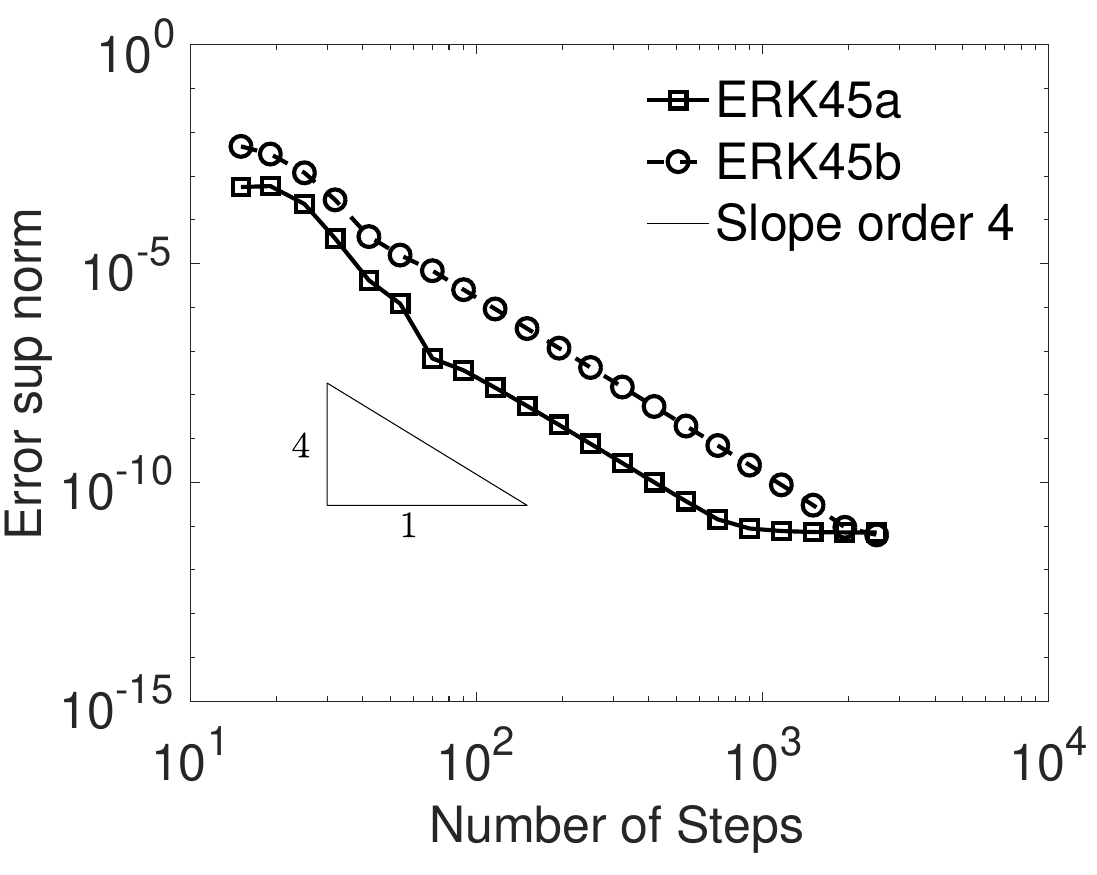}
\label{fig:GrayScott_ERK4}
}
\subfigure[Implicit methods of order three]{
\includegraphics[width=0.375\linewidth]{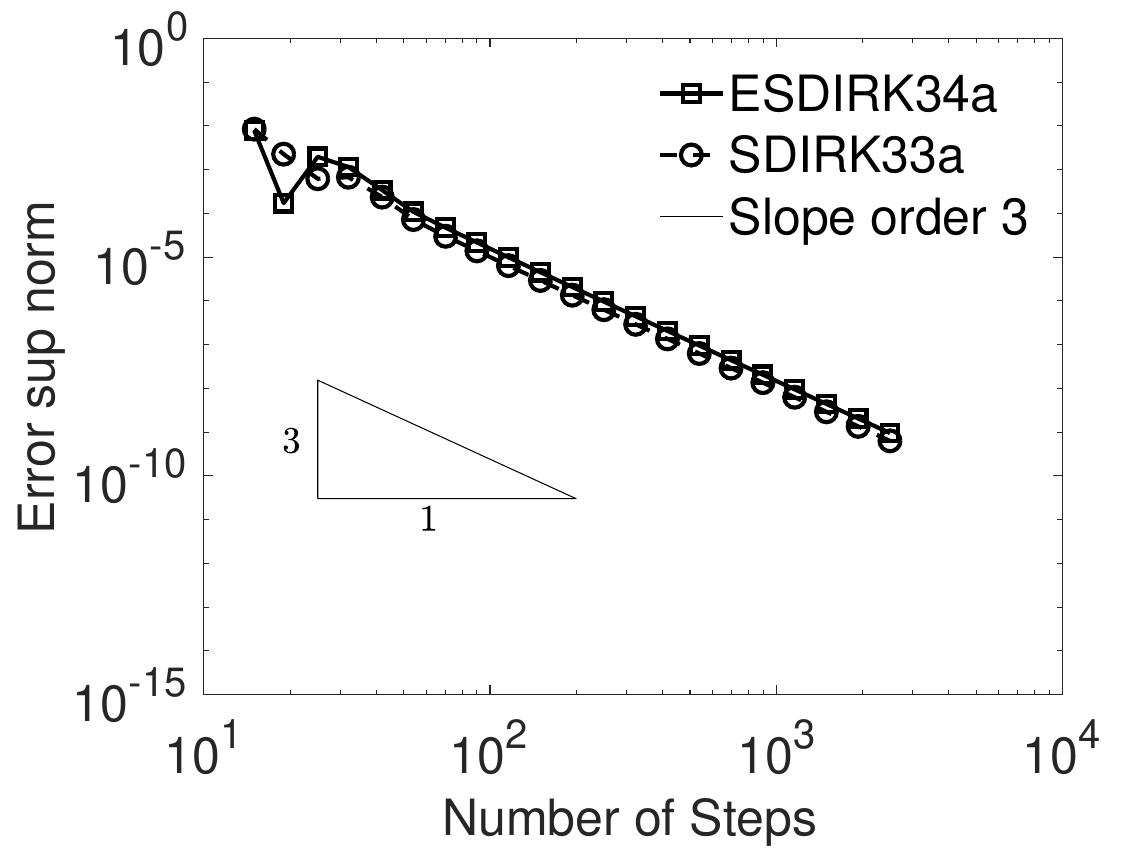}
\label{fig:GrayScott_IRK3}
}
\subfigure[Implicit method of order four]{
\includegraphics[width=0.375\linewidth]{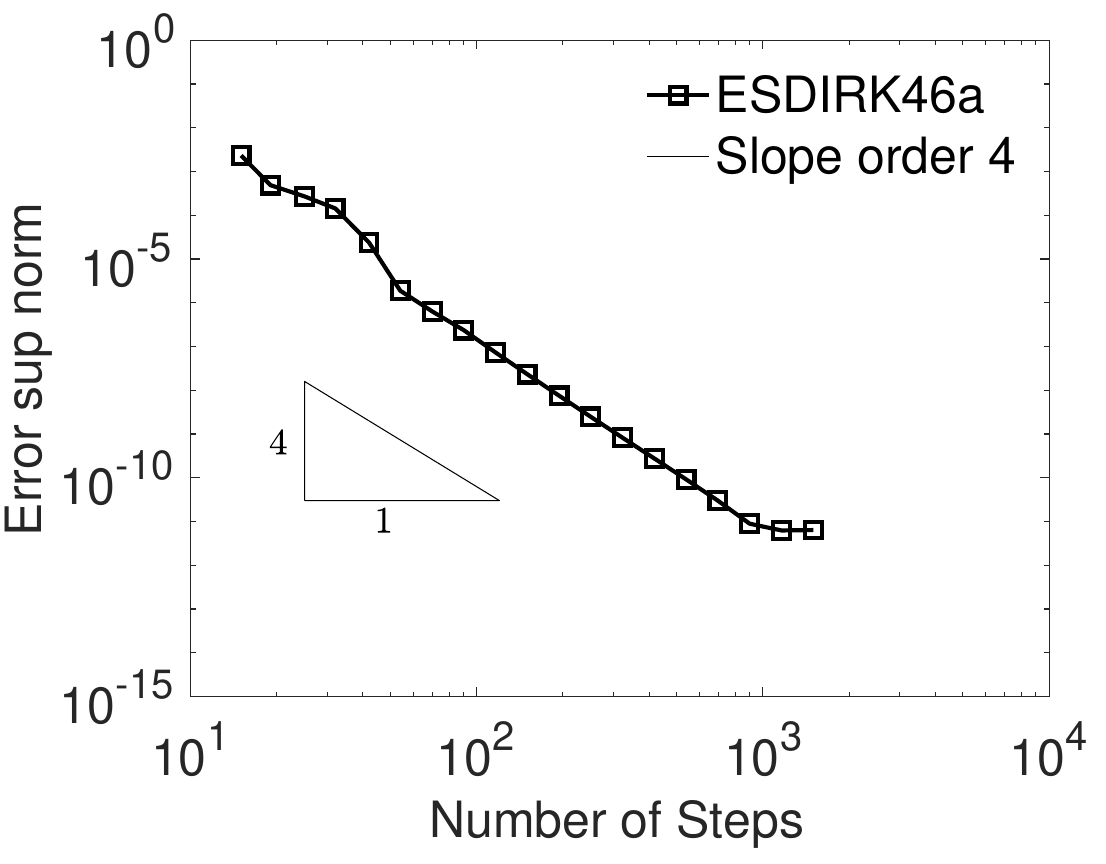}
\label{fig:GrayScott_IRK4}
}
\else
\subfigure[Methods of order three]{
\includegraphics[width=0.4\linewidth]{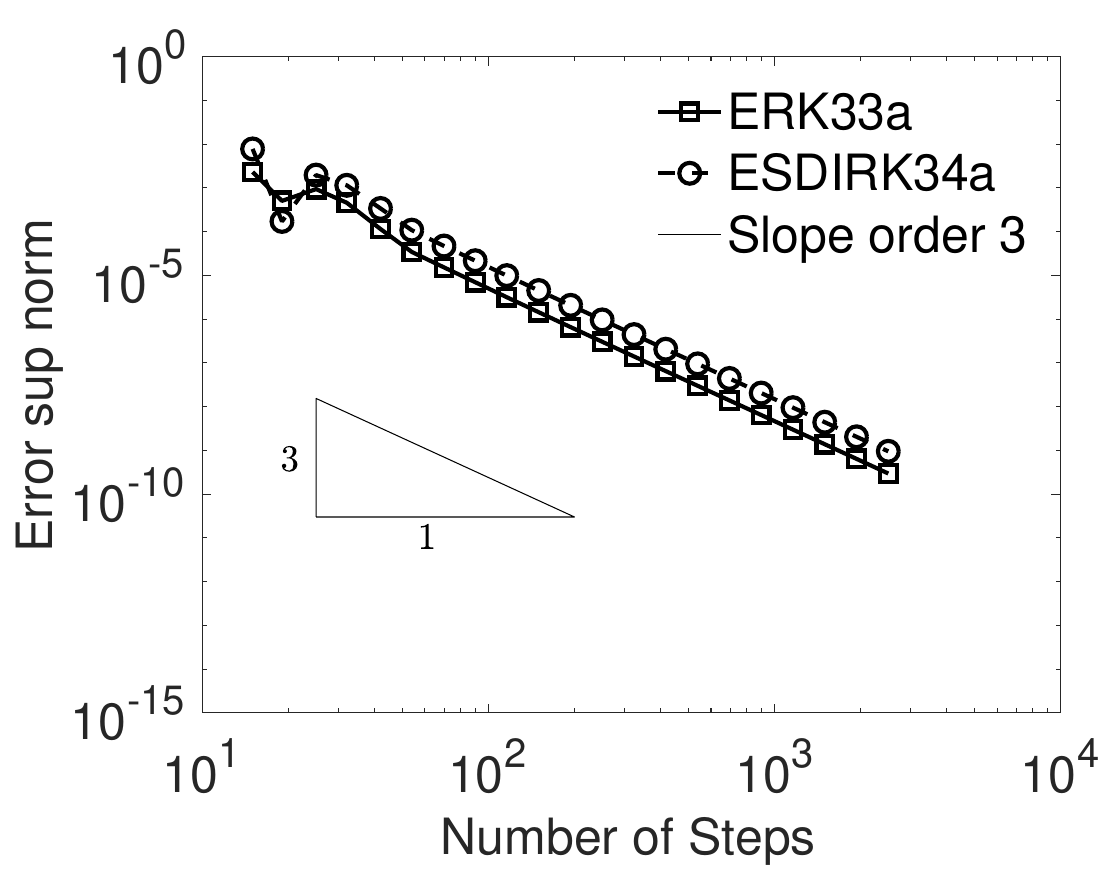}
\label{fig:GrayScott_order3}
}
\subfigure[Methods of order four]{
\includegraphics[width=0.4\linewidth]{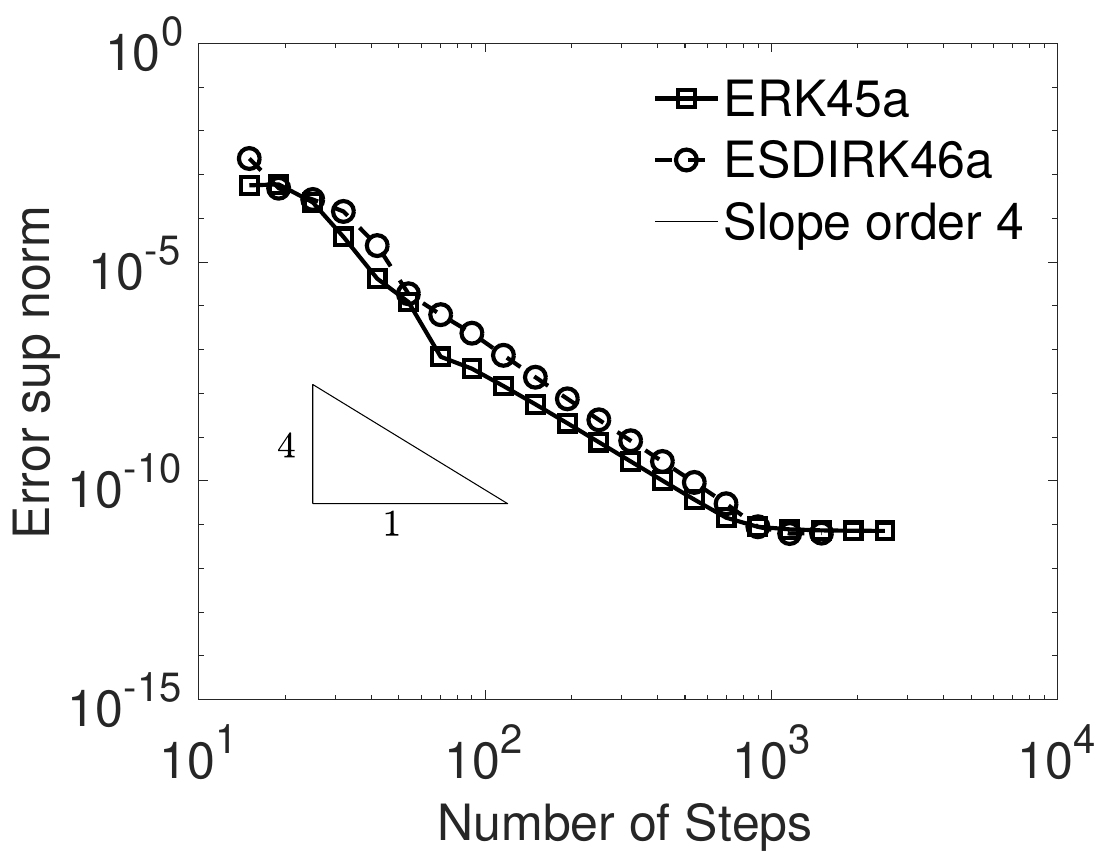}
\label{fig:GrayScott_order4}
}
\fi
\caption{Converge diagrams for MRI-GARK methods applied to the Gray-Scott model \eqref{eq:Gray-Scott}.} 
\label{fig:MRI-GrayScott_convergence}
\end{figure}

\subsection{Component partitioning: the KPR problem}
%
We next consider  the KPR problem used in \cite{Sandu_2013_MR-extrapolation}, which is a component partitioned system of the form \eqref{eqn:multirate-component-ode}. The nonlinear KPR problem is an adaptation to vector form of the scalar Prothero-Robinson \cite{Bartel_2002_MR-W,Hairer_book_II,Prothero_1974_PR} problem and is given by the following equations:
\begin{subequations}
\label{eq:PR:Nonlinear}
\begin{equation}
\label{eq:PR:Nonlinear:Equation}
\renewcommand{\arraystretch}{\astretch}
\begin{bmatrix}
y^{\{\f\}}\\
y^{\{\s\}}
\end{bmatrix}'
=
\boldsymbol{\Omega}
\cdot
\begin{bmatrix}
\frac{-3+y^{\{\f\}}\,^2 - \cos(\omega t)}{2\,y^{\{\f\}}} \\
\frac{-2+y^{\{\s\}}\,^2-\cos(t)}{2\,y^{\{\s\}}}
\end{bmatrix}
- 
\begin{bmatrix}
\frac{\omega\, \sin(\omega t)}{2\,y^{\{\f\}}} \\
\frac{\sin(t)}{2\,y^{\{\s\}}}
\end{bmatrix},
\end{equation}
where $\boldsymbol{\Omega}$ is defined in \eqref{eqn:ode2d-linear}. For the simulation we use the values $\lambdaf = -10$, $\lambdas = -1$, $\zeta \in \{0.1,0.5\}$, $\alpha \in \{1,5\}$, $\omega = 20$.
The exact solution of \eqref{eq:PR:Nonlinear:Equation} is given by:
\begin{equation}
\label{eq:PR:Nonlinear:Exact:Solution}
y^{\{\f\}}(t) = \sqrt{3+\cos(\omega t)}, \qquad
y^{\{\s\}}(t) = \sqrt{2+\cos(t)},
\end{equation}
\end{subequations}
and the initial conditions are the exact solution evaluated at the initial time. The simulation time interval is $[0,5\pi/2]$ (units).

We note that the difference between the fast and slow scales in \eqref{eq:PR:Nonlinear:Equation} is mainly driven by the ratio $\omega$ of the frequencies of the forcing terms, not by the ratio of dynamical terms $\lambdaf/\lambdas$.

The fast stage integration \eqref{eqn:MIS-component-stages} is carried out using Matlab's ode45 function with tight tolerances abstol=reltol=1.e-10. The convergence diagrams reported in Figure \ref{fig:MRI-KPR_convergence} indicate that the methods perform at their theoretical orders for this system. 

\begin{figure}[h]
\centering
\ifreport
\subfigure[Explicit methods of order three]{
\includegraphics[width=0.375\linewidth]{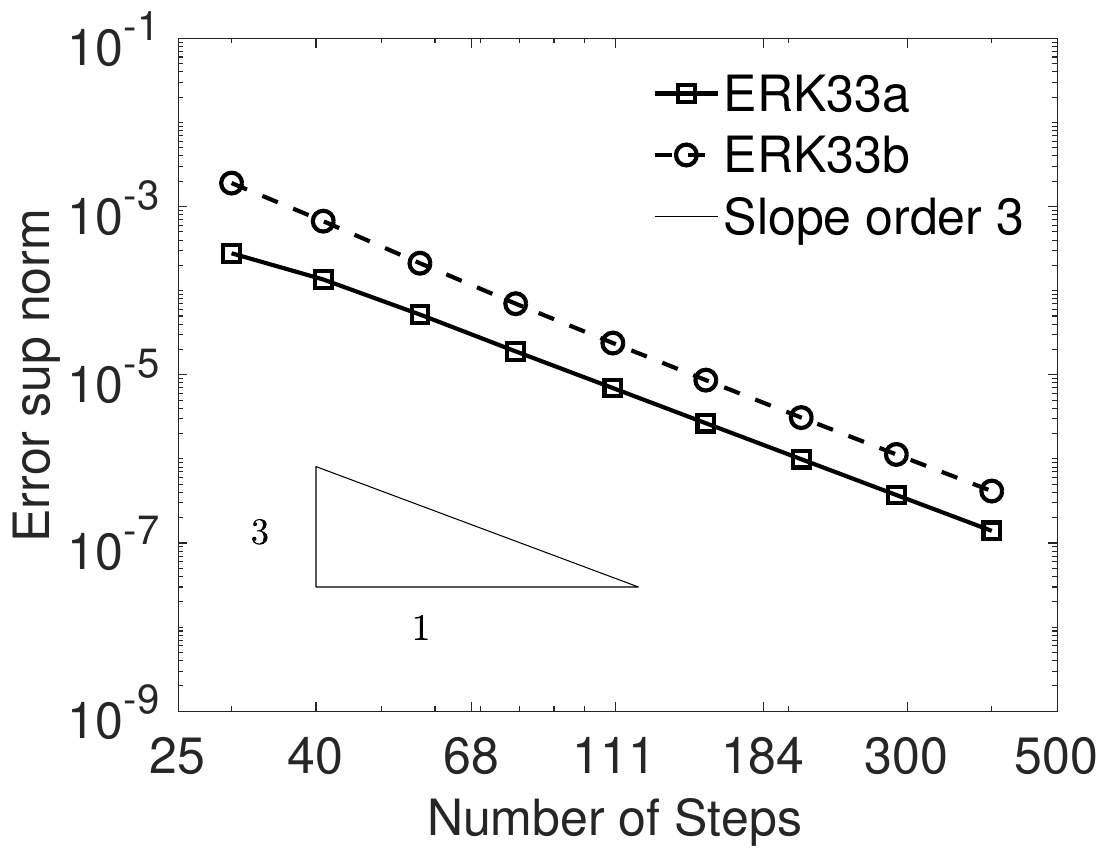}
\label{fig:KPR_ERK3}
}
\subfigure[Explicit methods of order four]{
\includegraphics[width=0.375\linewidth]{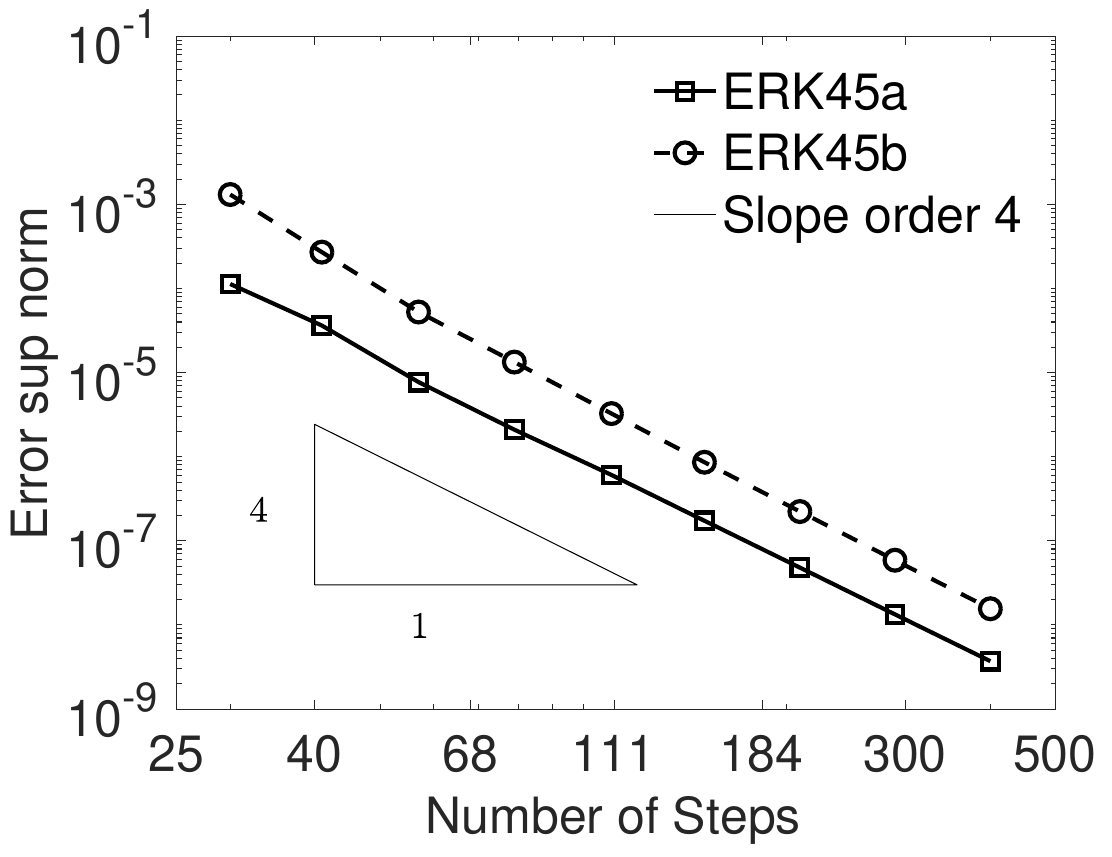}
\label{fig:KPR_ERK4}
}
\subfigure[Implicit methods of order three]{
\includegraphics[width=0.375\linewidth]{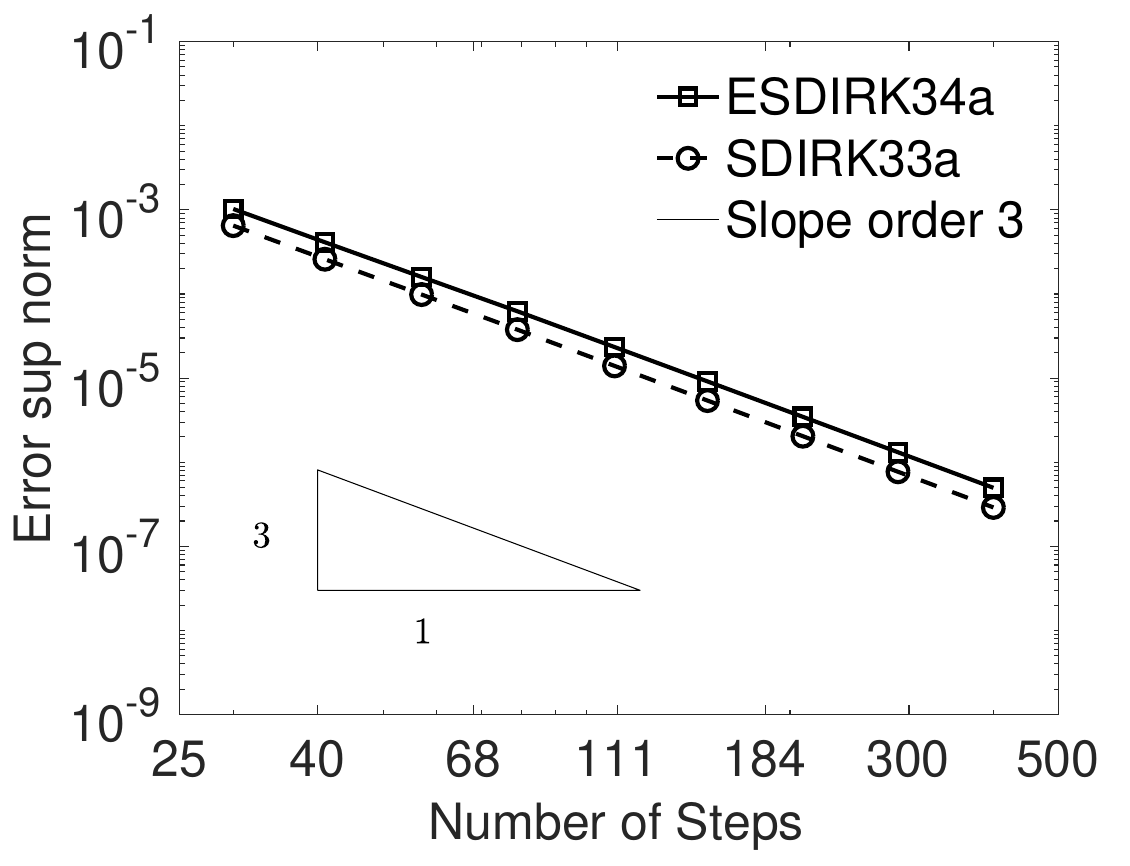}
\label{fig:KPR_IRK3}
}
\subfigure[Implicit method of order four]{
\includegraphics[width=0.375\linewidth]{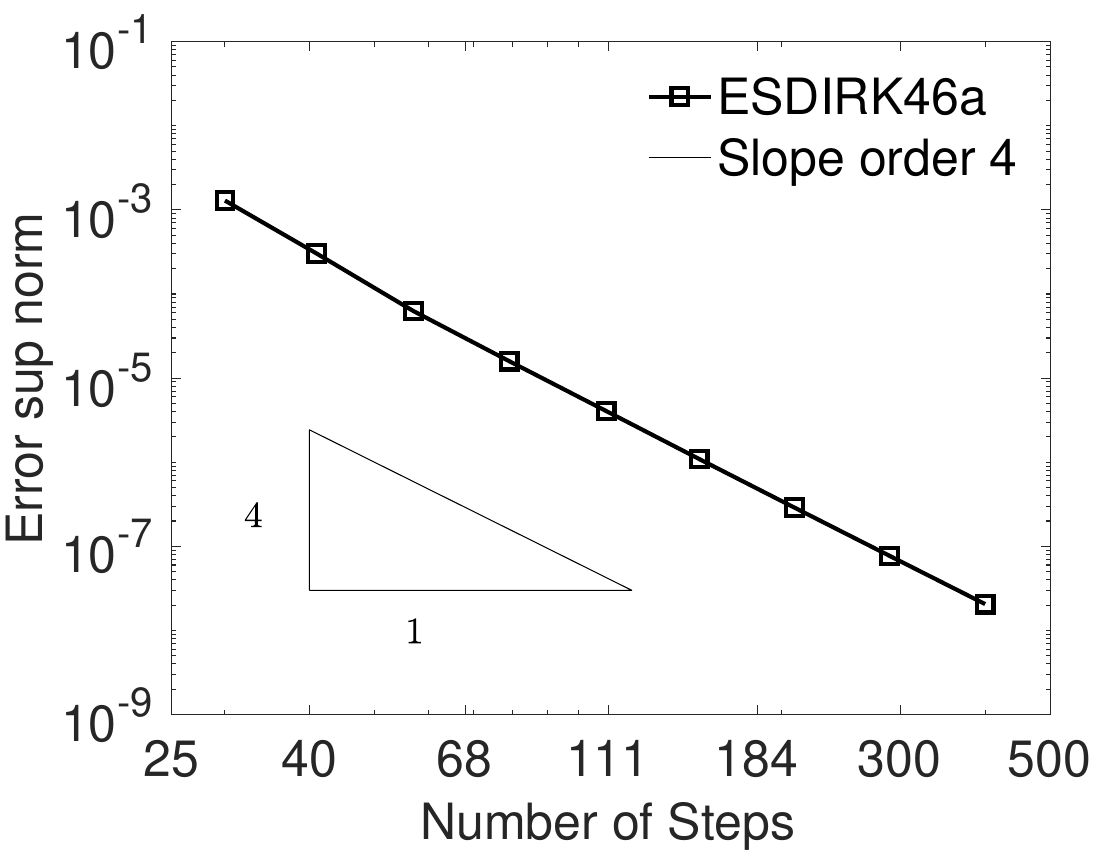}
\label{fig:KPR_IRK4}
}
\else
\subfigure[Methods of order three]{
\includegraphics[width=0.4\linewidth]{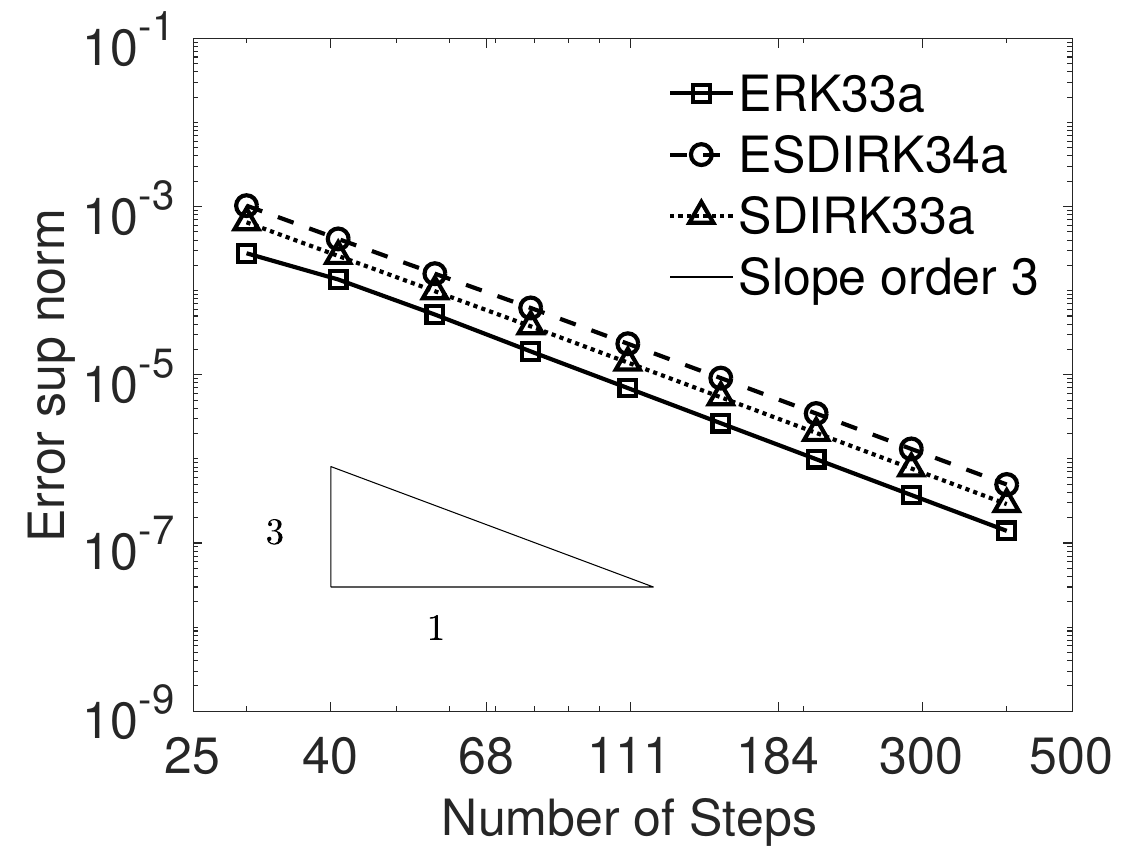}
\label{fig:KPR_order3}
}
\subfigure[Methods of order four]{
\includegraphics[width=0.4\linewidth]{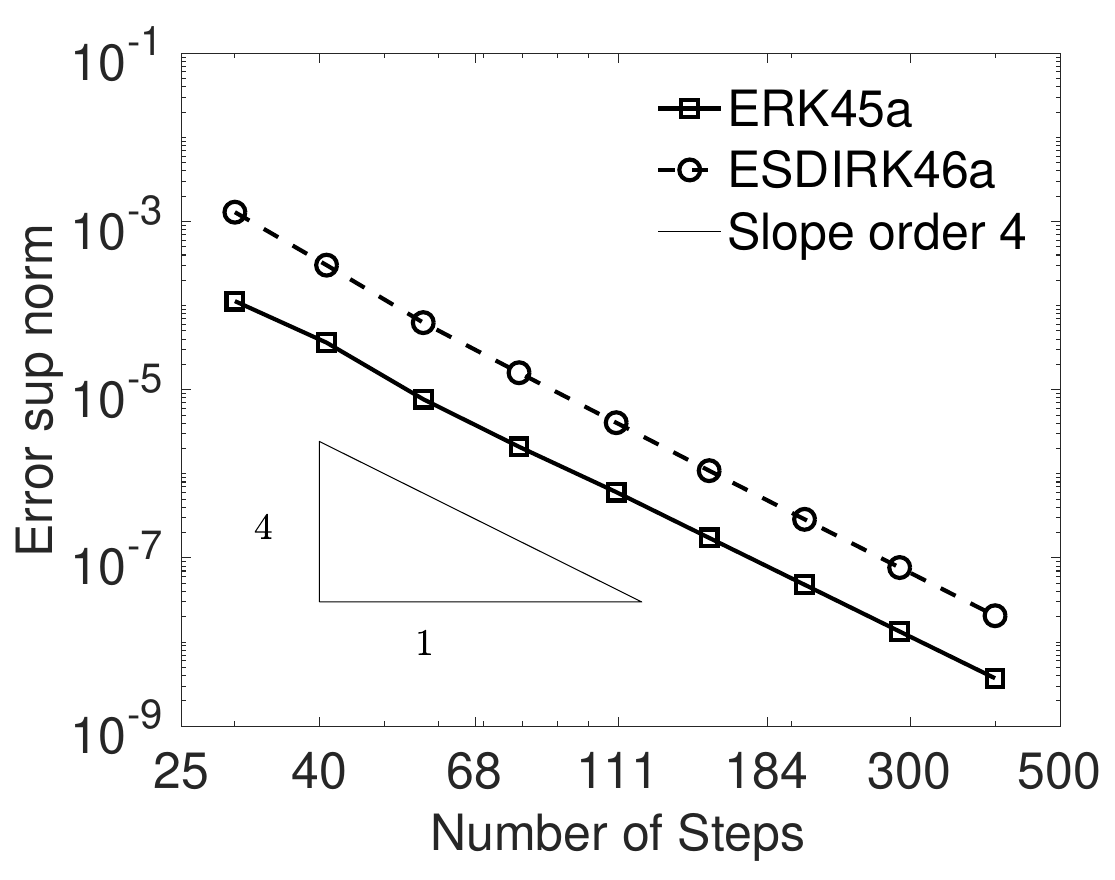}
\label{fig:KPR_order4}
}
\fi
\caption{Converge diagrams for MRI-GARK methods applied to the KPR model \eqref{eq:Gray-Scott}.} 
\label{fig:MRI-KPR_convergence}
\end{figure}


\section{Conclusions and Future Work}
\label{sec:conclusions}

We construct a class of multirate infinitesimal GARK schemes where the slow system is discretized with a Runge-Kutta base method, and a modified fast ordinary differential equations is solved to propagate information between consecutive stages of the slow method. The methods are fastest-first type, and are decoupled, in the sense that any implicit calculations are done separately for the slow or the fast components.  

This work extends the classical MIS approach \cite{Knoth_1998_MR-IMEX,Wensch_2009_MIS} in multiple ways. Time dependent coupling coefficients $\gamma(t)$ are used to incorporate  information about the slow dynamics into the modified fast system. The MRI-GARK general order conditions theory is developed by leveraging the GARK accuracy theory \cite{Sandu_2015_GARK}. This allows the construction of the first fourth order multirate infinitesimal methods. Moreover, the new framework enables the construction of the first  multirate infinitesimal schemes that are implicit in the slow component. Scalar and matrix stability analyses -- using a new simplified test problem -- are carried out. They reveal that the stability of the overall MRI-GARK process degrades for fast subsystems with low damping (the eigenvalues of the fast subsystem Jacobian are close to the imaginary axis.) 

The multirate infinitesimal framework can lead to new improved multirate discretizations for many applications, as this approach offers extreme flexibility in the choice of the numerical solution process for the fast component. First, the fast numerical discretization scheme can be explicit or implicit, and of any type that favors the application (e.g., Runge-Kutta, linear multistep, general linear, or exponential methods). In contrast, traditional fully discrete multirate approaches need to fix the fast discretization to a specific class of schemes before deriving order conditions. Second,  appropriate sequences of fast time steps, including adaptive steps for error control, can be employed in each fast subinterval. The coupling coefficients $\gamma(\tau)$ can be easily evaluated at each intermediate time point, and they remain bounded throughout.
We expect the new MRI-GARK family to be most useful for systems with widely disparate time scales, and where the fast process is dispersive 
and has a weak influence on the slow dynamics. We expect the high temporal orders of accuracy possible with MRI-GARK schemes to aid numerical simulations of time dependent partial differential equations where high order space discretizations are already used  \cite{Dawson_2014_MRDG,Goedel_2009_MRDG,Karakus_2016_MRDG,Remacle_2012_MRDG,Seny_2014_MRDG}.

Coupled implicit MRI-GARK schemes, where implicit stages are computed involving both the slow or the fast components, promise to improve the overall stability of the multirate infinitesimal schemes for oscillatory fast subsystems with low damping. The development of coupled  implicit MRI-GARK schemes will be reported elsewhere.


%

\end{document}